\documentclass[12pt,dvipsnames]{amsart}

\usepackage{amsfonts, amssymb, amsmath, amsthm}                               
\usepackage{mathrsfs, mathtools}                                              
\usepackage{geometry}                                                         
\usepackage{setspace}
\usepackage{xcolor}
\usepackage{graphicx}
\usepackage{hyperref}

\newtheorem{theorem}{Theorem} 	      	      	                              
\newtheorem{corollary}[theorem]{Corollary}     	      	      	      	      
\newtheorem{lemma}[theorem]{Lemma}     	       	      	      	      	      
\newtheorem{proposition}[theorem]{Proposition} 	      	      	      	      
\newtheorem{definition}[theorem]{Definition} 	      	      	                
\newtheorem{assumption}[theorem]{Assumption}     	      	      	      	    
\newtheorem{remark}[theorem]{Remark}                                          

\numberwithin{equation}{section}                                              
\numberwithin{theorem}{section}                                               

\newcommand{\mb}[1]{\mathbf{#1}}                                              
\newcommand{\mc}[1]{\mathcal{#1}}                                             
\newcommand{\mi}[1]{\mathscr{#1}}                                             

\newcommand{\R}{\mathbb{R}}                                                   
\newcommand{\Sph}{\mathbb{S}}                                                 

\newcommand\numberthis{\addtocounter{equation}{1}\tag{\theequation}}
\newcommand{\nb}{\nabla}

\allowdisplaybreaks
\begin{document}

\title{Control of waves on Lorentzian manifolds with curvature bounds}

\author{Vaibhav Kumar Jena}
\address{School of Mathematical Sciences\\
Queen Mary University of London\\
London E1 4NS\\ United Kingdom}
\email{v.k.jena@qmul.ac.uk}

\author{Arick Shao}
\address{School of Mathematical Sciences\\
Queen Mary University of London\\
London E1 4NS\\ United Kingdom}
\email{a.shao@qmul.ac.uk}

\begin{abstract}
We prove boundary controllability results for wave equations (with lower-order terms) on Lorentzian manifolds with time-dependent geometry satisfying suitable curvature bounds.
The main ingredient is a novel global Carleman estimate on Lorentzian manifolds that is supported in the exterior of a null (or characteristic) cone, which leads to both an observability inequality and bounds for the corresponding constant.
The Carleman estimate also yields a unique continuation result on the null cone exterior, which has applications toward inverse problems for linear waves on Lorentzian backgrounds.
\end{abstract}

\maketitle

\section{Introduction} \label{sec.intro}

The goal of this article is to obtain a novel \emph{boundary controllability} result for \emph{wave equations} on \emph{Lorentzian manifolds} with \emph{time-dependent geometry}.
To solve this control problem, we will derive a new Carleman estimate for the wave operator on Lorentzian geometries. 

To explain the problem further, let $( \mc{M}, g )$ be a $(n+1)$-dimensional Lorentzian manifold.
\footnote{By convention, we assume $g$ has signature $( -, +, \dots, + )$.}
Furthermore, let $\nb$ and $R$ denote the Levi-Civita connection and curvature for $g$, respectively.
Throughout this work, we use $\square$ to denote the usual wave operator for $g$, given by
\[
\square := g^{\alpha \beta} \nb_{\alpha\beta} \text{,}
\]
where the indices $(\alpha, \beta)$ represent spacetime components.
\footnote{We adopt Einstein summation notation throughout, hence $\alpha$ and $\beta$ and are summed over all components.}

We now give a description of the control problem. Let $\mc{U} \subset \mc{M}$ be an open subset, with smooth timelike boundary $\partial \mc{U}$.
\footnote{With regards to wave equations on $\bar{\mc{U}}$, this means one can impose boundary data on $\partial \mc{U}$.}
Consider the following wave equation,
\begin{equation} \label{eq_ctrl_wv_LM}
\begin{cases}
\square y + \nb_{\mc{X}} y + q y = 0 \text{,} \qquad & \text{on } \mc{U} \text{,}\\
y = F \textbf{1}_\Gamma \text{,} & \text{in } \partial \mc{U} \text{,}\\
(y,\partial_t y) = (y_0^-,y_1^-) \text{,} & \text{at } \tau=\tau_- \text{,}
\end{cases}
\end{equation}
where $\mc{X}$ is a smooth vector field on $\bar{\mc{U}}$, $q \in C^\infty(\bar{\mc{U}})$ is the potential, $F$ is the Dirichlet boundary data, $\Gamma \subseteq \partial \mc{U}$ is the control region, and $(y_0^-,y_1^-) \in L^2 ( \tau = \tau_- ) \times H^{-1} ( \tau = \tau_- )$ is the initial data.
Here, $\tau$ is some suitable time function in this setting.

Our main control problem is the following: \textit{Given any initial and final data $(y_0^{\pm},y_1^{\pm})$, does there exist Dirichlet boundary data $F$ such that the solution $y$ of \eqref{eq_ctrl_wv_LM} satisfies}
\begin{equation}
(y,\partial_t y) = ( y_0^+, y_1^+ ) \qquad \text{at } \tau = \tau_+ \text{?}
\end{equation}

From standard duality arguments (see \cite{dolec_russe:obs_control, lionj:control_hum, lionj:ctrlstab_hum, micu_zua:control_pde}), we know that in order to solve the control problem, we must prove a suitable \emph{observability inequality} of the form
\begin{equation} \label{eq_obs_ineq_1_LM}
\| (\phi_0,\phi_1) \|_{ H^1_0 ( \tau = \tau_- ) \times L^2 ( \tau = \tau_- ) }^2 \leqslant C \| \mc{N} \phi \|_{ L^2 (\Gamma) }^2 \text{,}
\end{equation}
for some constant $C>0$, where $\mc{N}$ denotes the outward pointing unit ($g$-)normal of $\mc{U}$, and where $\phi$ denotes the solution of the adjoint system with initial data $( \phi_0, \phi_1 )$,
\begin{equation} \label{eq_obs_wv_LM}
\begin{cases}
\square \phi + \nb_{\mc{X}} \phi + V \phi = 0 \text{,} \qquad & \text{on } \mc{U} \text{,}\\
\phi = 0 \text{,} & \text{in } \partial \mc{U} \text{,}\\
(\phi,\partial_t \phi) = (\phi_0,\phi_1) \text{,} & \text{at } \tau=\tau_- \text{,}
\end{cases}
\end{equation}
with the potential $V$ satisfying $V = q - \nb_\alpha \mc{X}^\alpha$.

Henceforth, our aim is to prove an observability inequality of the form \eqref{eq_obs_ineq_1_LM}.
This will be achieved through a new Carleman estimate for geometric wave operators.

\subsection{Literature}

There is an extensive amount of literature on controllability for wave equations.
For conciseness, we only focus here on results that are more directly relevant to the problem at hand.
For a more detailed survey of the area, see the introduction in \cite{Arick}.

First, \emph{microlocal methods} have proved rather successful for obtaining near-optimal results for controllability of wave equations.
A pioneering result in this direction is that of Bardos, Lebeau, and Rauch \cite{BLR}, which characterized the controllability of waves using the \emph{geometric control condition} (\emph{GCC}).
In the context of \eqref{eq_ctrl_wv_LM}, the GCC roughly states that every null geodesic---or bicharacteristic---in $\mc{U}$ that reflects off the boundary of $\partial \mc{U}$ must hit $\Gamma$.

Many microlocal results have further built on \cite{BLR} and the GCC.
For instance, \cite{laur_leaut:obs_unif} solved an interior control problem for operators of the form $-\partial_{tt}^2 + \Delta_k$ on $(-T,T) \times M$, where $(M, k)$ is a Riemannian manifold with or without boundary, and where $\Delta_k$ is the Laplace-Beltrami operator on $M$; the authors assume an analogue of the GCC, and they also obtain bounds for the observability constant $C$ in \eqref{eq_obs_ineq_1_LM}.
In addition, in \cite{lero_leb_terpo_trel:gcc_time}, the authors extended the above results to control regions $\Gamma$ that are time-dependent.

On the other hand, microlocal results have only been established for wave equations with time-analytic coefficients (i.e., $g$, $\mc{X}$, and $q$ in \eqref{eq_ctrl_wv_LM}), a rather restrictive class.
\footnote{This is because the argument uses the unique continuation result of Tataru \cite{tat:uc_hh} as a crucial step.}
Such methods have not been applicable to wave equations with more general time-dependent coefficients, in particular to waves on time-dependent geometric backgrounds.
Consequently, microlocal methods are not yet suitable for the boundary control problem considered here.

Another method for proving observability is using \emph{Carleman estimates} \cite{Carleman, lasie_trigg_zhang:wave_global_uc, tat:paper}, which are weighted integral inequalities that are often used for proving unique continuation properties for PDEs; see \cite{cald:unique_cauchy, carl:uc_strong, hor:lpdo4, iman}.
Carleman methods can be viewed as an extension of more traditional \emph{multiplier methods} for wave equations (see, e.g., \cite{ho:obs_wave, komo:control_mintime, lionj:ctrlstab_hum}), with the former having the advantage that it can also deal with wave equations \eqref{eq_ctrl_wv_LM}, \eqref{eq_obs_wv_LM} containing lower-order terms.
On one hand, results obtained via Carleman estimates tend to be weaker than their microlocal counterparts, in that the former does not always achieve the GCC.
However, the key advantage of Carleman estimates is that the former is applicable to a wider class of PDEs, in which the coefficients $V$ and $\mc{X}$ in \eqref{eq_obs_wv_LM} can also be non-analytic.

Carleman estimates have been used for showing observability for wave equations on time-independent geometric backgrounds $( -T, T ) \times M$, with $( M, k )$ as above, and with suitable assumptions on the geometry; see, for instance, \cite{duy_zhang_zua:obs_opt, FYZ, trigg_yao:carleman_wave_uc, yao_var_coeff}.
These articles assume the existence of a suitable function or vector field satisfying appropriate convexity properties; this can then be used to derive a suitable Carleman estimate.
However, for a given manifold, it may not always be easy to check if such a convex quantity exists.
\footnote{However, \cite{yao_var_coeff} shows that its convexity assumption is satisfied on backgrounds with positive curvature.}

On the other hand, there are nearly no corresponding results when the geometry is time-dependent, or in other words, in a general Lorentzian geometric setting.
(However, see \cite{lor_cald, taka:carleman_td} for closely related inverse problems result in a Lorentzian setting.)
To the authors' knowledge, this article provides one of the first results addressing control of wave equations \eqref{eq_ctrl_wv_LM} in Lorentzian geometries in such generality.
Another advantageous feature of our result is that it holds given sufficient bounds only on the curvature associated with the Lorentzian metric $g$, without assuming a priori the existence of a suitable convex quantity.

A closely connected result is that of \cite{Arick} by the second author.
This established boundary controllability for wave equations on flat spacetime, but on general time-dependent domains.
The present paper can be viewed as an extension of \cite{Arick} to geometric settings.
In particular, many of the ideas and tools in this article---e.g., the use of Lorentzian geometric techniques, the key features of our Carleman estimate---originated from \cite{Arick}.

Yet another closely related work is the recent inverse problems result of Alexakis, Feizmohammadi, and Oksanen \cite{lor_cald}, which established, assuming a specific set of curvature inequalities, a Lorentzian analogue of the Calder\'on problem for the wave equation on time-dependent geometries.
A key ingredient in their proof is a unique continuation result on the exteriors of null cones.
On the other hand, \cite{lor_cald} was only able to treat wave equations without first-order terms, that is, \eqref{eq_ctrl_wv_LM} with $\mc{X} = 0$.
The Carleman estimates in this paper yield a new unique continuation result in null cone exteriors that allows for $\mc{X} \neq 0$.
This, in turn, implies an analogue of the inverse problem result of \cite{lor_cald} that, under certain curvature bounds, allows for a larger class of geometric wave equations with general first-order terms.

Finally, another improvement to \cite{lor_cald} was independently obtained by Alexakis, Feizmohammadi, and Oksanen in \cite{lor_inv_new}, also by obtaining a modified unique continuation result.

\subsection{Main result}

Our next aim is to give rough statements of the main results of this article.
To help with our geometric assumptions, we recall the following:

\begin{definition} \label{def_geometric_LM}
$\mc{M}$ is called a \emph{geodesically star-shaped neighbourhood of $q \in \mc{M}$} iff every $q' \in \mc{M}$ is connected to $q$ by a unique geodesic in $\mc{M}$.
\end{definition}

In particular, $\mc{M}$ being a geodesically star-shaped neighbourhood of $q \in \mc{M}$ ensures that normal coordinates about $q$ are well-defined and smooth everywhere on $\mc{M}$.

\begin{assumption} \label{ass_LM}
The following setup will be common to all our main results:
\begin{enumerate}
\item Let $\mc{U}$ be an open subset of $\mc{M}$, with smooth timelike boundary $\partial \mc{U}$.

\item Fix $p \in \mc{M}$, and assume $\mc{M}$ is a geodesically star-shaped neighbourhood of $p$.
Moreover, fix a single normal coordinate system $( x^0 := t, x^1, \dots, x^n )$ about $p$.

\item Let $\mc{D}$ denote the exterior of the null cone in $\mc{M}$ about $p$.
\footnote{These are the points reached by the spacelike radial geodesics from $p$; see \eqref{eq.hf_D} for another description.}

\item Fix $r_0 > 0$, and assume
\[
\mc{U} \cap \mc{D} \subseteq \{ r < r_0 \} \text{,} \qquad r := \sqrt{ ( x^1 )^2 + \dots + ( x^n )^2 } \text{.}
\]

\item Assume that $\bar{\mc{D}} \cap \bar{\mc{U}}$ is ``fully contained in $\mc{M}$"---in other words, no spacelike radial geodesic from $p$ leaves $\mc{M}$ before it passes through $\partial \mc{U}$.
\footnote{An equivalent characterization is that the closure of $\bar{\mc{D}} \cap \bar{\mc{U}}$ is a compact subset of $\mc{M}$.}
\end{enumerate}
\end{assumption}

A few comments on Assumption \ref{ass_LM} are in order.
First, as before, $\mc{U}$ in (1) is the domain on which our wave equation is solved.
Observe that (2) can be achieved without any loss of generality by shrinking $\mc{M}$ if needed, as long as $\mc{U}$ remains within $\mc{M}$.

Next, (4) means that $\mc{U} \cap \mc{D}$---the part of $\mc{U}$ that we will primarily consider in our results---is spatially bounded.
Finally, (5) can be interpreted as our normal coordinate system being large enough, so that the null cone exterior $\mc{D}$ is well-defined up to all of $r \leq r_0$.

With the above setup, we now give a rough statement of our main control result:

\begin{theorem} \label{thm_ctrl_rslt_LM}
Suppose we are given the setup described in Assumption \ref{ass_LM}.
In addition:
\begin{itemize}
\item We assume there is a neighbourhood $\mc{V}$ of $p$ in $\mc{M}$ such that $\mc{M}$ is a geodesically star-shaped neighbourhood of every $q \in \mc{V}$.

\item Let $\mc{V}_-,\mc{V}_+$ be two spacelike cross-sections of $\mc{U}$,\footnote{With regards to wave equations on $\bar{\mc{U}}$, this means one can impose initial and final data on $\mc{V}^\pm$.} such that $\mc{V}_-$ lies in the past of $\mc{U} \cap \bar{\mc{D}}$ and $\mc{V}_+$ lies in the future of $\ \mc{U} \cap \bar{\mc{D}}$.
\end{itemize}
Then, there exist universal constants $\mc{C}_\dagger > 0$ and $0 < \varepsilon_0 \ll 1$, such that if
\begin{equation}\label{eq_curv_ass_LM}
|R| < \frac{\varepsilon_0 \mc{C}_\dagger}{r_0^2} \text{,} \qquad |\nb R| < \frac{\mc{C}_\dagger}{r_0^3} \text{,}
\end{equation}
then given any $(y_0^{\pm}, y_1^{\pm}) \in L^2(\mc{V}_{\pm}) \times H^{-1}(\mc{V}_{\pm})$, there exists a control function $F \in L^2(\Gamma)$, where $\Gamma$ is an appropriate subset of $\partial \mc{U}$, such that the solution $y$ of 
\begin{equation} \label{eq_ctrl_wv_1_LM}
\begin{cases}
\square y + \nb_{\mc{X}} y + q y = 0 \text{,} \qquad & \text{on } \mc{U} \text{,}\\
y = F \mathbf{1}_\Gamma \text{,} & \text{on } \partial \mc{U} \text{,} \\
(y,\partial_t y) = (y_0^-,y_1^-) \text{,} & \text{on } \mc{V}_- \text{,}
\end{cases}
\end{equation}
satisfies
\begin{equation} \label{ctrl_pblm_1_MB}
(y,\partial_t y) = (y_0^+,y_1^+) \qquad \text{on } \mc{V}_+ \text{.}
\end{equation}
\end{theorem}

The precise version of Theorem \ref{thm_ctrl_rslt_LM} is given later as Theorems \ref{thm_control_ext_mr} and \ref{thm_control_int_mr}.
Also, see these theorem statements for the precise curvature assumptions, and for the precise descriptions of $\Gamma$.
(In particular, $\Gamma$ can be chosen to lie within any neighbourhood of $\partial \mc{U} \cap \bar{\mc{D}}$ in $\partial \mc{U}$.)

\begin{remark}
One can show that the initial-boundary value problem \eqref{eq_ctrl_wv_1_LM} has a unique (weak) solution using standard well-posedness arguments.
For details see, for instance, \cite{hor:lpdo4, lionj_mage:bvp1}.
\end{remark}

\begin{remark}
Theorem \ref{thm_ctrl_rslt_LM} holds even when $\mc{V}_\pm$, where the initial and final data are imposed, do not lie inside a geodesically star-shaped neighbourhood of $p$.
Indeed, suppose $\mc{W} \subset \mc{M}$ is a geodesically star-shaped neighbourhood of $p$, and let $\mc{V}_+$, $\mc{V}_-$ lie in the future and past of $\mc{W}$, respectively.
Now, let $\mc{W}_\pm$ be two spacelike cross-sections of $\mc{U}$ in $\mc{W}$, with $\mc{W}_-$ in the future of $\mc{V}_-$ and $\mc{W}_+$ in the past of $\mc{V}_+$.
The data on $\mc{V}_\pm$ yields suitable corresponding data on $\mc{W}_\pm$ (for instance, by solving \eqref{eq_ctrl_wv_1_LM} with zero boundary data).
Then, the desired result, with data on $\mc{V}^\pm$, follows by applying Theorem \ref{thm_ctrl_rslt_LM} with manifold $\mc{W}$ and cross-sections $\mc{W}^\pm$.
\end{remark}

\begin{remark}
The assumptions in Theorem \ref{thm_ctrl_rslt_LM} of $\mc{M}$ being a geodesic star-shaped neighbourhood is imposed for convenience, as this can also be derived from appropriate geometric assumptions; see, e.g., \cite{cbl_lef:rin}.
In particular, the lack of conjugate points along radial geodesics can be deduced from the curvature bounds given in \eqref{eq_curv_ass_LM}.
\end{remark}

By the duality arguments of \cite{dolec_russe:obs_control, lionj:ctrlstab_hum}, the key step to proving Theorem \ref{thm_ctrl_rslt_LM} is to obtain an observability estimate of the form \eqref{eq_obs_ineq_1_LM}.
For this, the main ingredient is a novel \emph{Carleman estimate} that is suitable for the setting of Theorem \ref{thm_ctrl_rslt_LM}.
This Carleman estimate is the main technical contribution of the paper, and a rough version of this is stated as follows:

\begin{theorem} \label{thm_carleman_rslt_LM}
Suppose we are given the setup described in Assumption \ref{ass_LM}.
Then, there exist universal constants $\mc{C}_\dagger >0$ and $0 < \varepsilon_0 \ll 1$ such that if 
\begin{equation}\label{eq_carleman_curv_ass_LM}
|R| < \frac{\varepsilon_0 \mc{C}_\dagger}{r_0^2} \text{,} \qquad |\nb R| < \frac{\mc{C}_\dagger}{r_0^3} \text{,}
\end{equation}
then for any $\phi \in C^2(\mc{U}) \cap C^2(\bar{\mc{U}})$ satisfying
\[
\phi|_{\partial\mc{U} \cap \mc{D} } = 0 \text{,}
\]
and for any constants $a \gg 1$ and $\varepsilon_0 \ll b_0 \ll 1$, we have the estimate
\begin{align*}
\frac{ \varepsilon_0 }{ r_0^2 } \int_{\mc{U} \cap \mc{D}} \zeta & \left[ w_\rho (E_\rho \phi)^2 + w_\theta (E_\theta \phi)^2 + w_\omega \sum_A (E_A \phi)^2 \right] + \frac{ a^2 b_0 }{ r_0^2 } \int_{\mc{U} \cap \mc{D}} \zeta w_0 \phi^2 \\
&\qquad \leq \frac{1}{a} \int_{\mc{U} \cap \mc{D}} \zeta w_\square |\square \phi|^2 + \int_{\partial \mc{U} \cap \mc{D}} \zeta w_b |\mc{N} \phi|^2 \text{,} 
\end{align*}
where:
\begin{itemize}
\item $\zeta$ is a suitable Carleman weight that also depends on $a_0, b_0, \varepsilon_0$.

\item $E_\rho$, $E_\theta$, $E_A$ ($1 \leq A < n$) denote special orthogonal frames on $\mc{D}$.

\item $w_\rho$, $w_\theta$, $w_\omega$, $w_0$, $w_\square$ are positive weight functions on $\mc{D}$.

\item $w_b$ is a real-valued (not necessarily positive) weight function on $\partial \mc{U} \cap \mc{D}$.

\item $\mc{N}$ is the outward unit normal (with respect to $g$) of $\mc{U}$.
\end{itemize}
\end{theorem}

The precise version of Theorem \ref{thm_carleman_rslt_LM} is stated in Theorem \ref{thm_carl_est_LM}, in which the precise curvature assumptions\footnote{In particular, the bounds in \eqref{eq_curv_ass_LM} and \eqref{eq_carleman_curv_ass_LM} need only hold for a certain combination of frame elements.} and the precise weight functions $\zeta$, $w_\rho$, $w_\theta$, $w_\omega$, $w_0$, $w_\square$, $w_b$ are given.
For details on the frames $E_\rho$, $E_\theta$, $E_A$, see the beginning of Section \ref{ssec_hyperq_func_LM}.

\begin{remark}
The geometric assumptions in Theorem \ref{thm_ctrl_rslt_LM} can be slightly relaxed when $p \not\in \bar{\mc{U}}$.
In this case, we only require $\mc{M}$ to be a geodesically star-shaped neighbourhood of $p$ (and not nearby points).
Furthermore, we can choose $\Gamma$ to lie within $\partial \mc{U} \cap \mc{D}$.
This improvement arises since in this case, we only apply the Carleman estimate about the single point $p$.

This slightly stronger result is given in the precise theorem statement, in Theorem \ref{thm_control_ext_mr}; we omitted this in Theorem \ref{thm_ctrl_rslt_LM} in order to keep the rough version concise.
\end{remark}

\subsection{Key features}

As mentioned before, the key step in the proof of Theorem \ref{thm_ctrl_rslt_LM} is the Carleman estimate given in Theorem \ref{thm_carleman_rslt_LM}.
In particular, the main geometric assumptions of Theorem \ref{thm_ctrl_rslt_LM} arise because they are crucial for the Carleman estimate.
In the following, we turn our focus to the Carleman estimate, and we discuss its features in greater detail.

\subsubsection{The Carleman estimate}

The Carleman estimate, Theorem \ref{thm_carleman_rslt_LM}, holds for wave equations of the form \eqref{eq_obs_wv_LM} on $(n+1)$-dimensional Lorentzian manifolds with time-dependent geometry.
No analyticity or partial analyticity is assumed for any of the coefficients $g$, $\mc{X}$, or $V$.
Furthermore, this estimate allows us to control a weighted $H^1$-norm of the wave; this $H^1$-bound will then lead to the desired observability inequality \eqref{eq_obs_ineq_1_LM}.

Our Carleman estimate is naturally supported in the exterior $\mc{D}$ of a null cone about a point $p$, \emph{without imposing extra vanishing assumptions or using an additional cutoff function}.
This is achieved through the use of a specially chosen Carleman weight that vanishes on the null cone $\partial \mc{D}$.
These types of weights were originally constructed in \cite{alex_shao:uc_global}, which studied global unique continuation properties of waves, and has since been used in \cite{lor_cald, jena, Arick}.

One consequence of this feature with regards to Theorem \ref{thm_ctrl_rslt_LM} is that the control region $\Gamma$ can also be restricted to $\mc{D}$.
This was a key improvement in the controllability results of \cite{jena, Arick}, and it carries over the general Lorentzian geometric settings of this paper.

The above-mentioned Carleman weight is constructed using (a perturbation of) the hyperquadratic function $f$ about $p$.
(In particular, $f$ vanishes on the null cone $\partial \mc{D}$.)
Because of this essential role played by the hyperquadric function, we must assume, for our Carleman estimates, that the normal coordinates about $p$ are well-defined in $\mc{D}$.
This is also the main reason for the normal coordinate conditions imposed in Theorem \ref{thm_ctrl_rslt_LM}.

The curvature bounds \eqref{eq_curv_ass_LM} can be roughly interpreted as \emph{$( \mc{M}, g )$ being locally close to the Minkowski geometry}.
\eqref{eq_curv_ass_LM} is primarily used to control derivatives of $f$ and the normal time coordinate $t$, both of which are crucial for defining the pseudoconvex foliation of $\mc{D}$ that is central to our Carleman estimate.
In particular, this pseudoconvexity enables us to treat wave equations with general first-order terms and to control the $H^1$-norm of the wave on $\mc{D}$.

Finally, like in \cite{jena, Arick}, the constants in our Carleman estimate can be explicitly bounded in terms of the properties of $\mc{U}$, the geometry of $( \mc{M}, g )$, and the choice of $p \in \mc{M}$.
Using analogues of arguments given in \cite{jena, Arick}, this information can be then used to obtain explicit estimates for the observability constant $C$ in \eqref{eq_obs_ineq_1_LM} in terms of the same parameters.
This is one advantageous feature of our Carleman-based argument.

\subsubsection{Applications to inverse problems}

One immediate corollary of our Carleman estimate is a quantitative global unique continuation result for \eqref{eq_obs_wv_LM}---in particular, the Dirichlet and Neumann data on $\Gamma$ determines the solution $\phi$ on all of $\mc{U} \cap \mc{D}$.
This unique continuation property can be applied to improve the inverse problems result of \cite{lor_cald}.

To be more precise, in \cite{lor_cald}, the authors consider the wave equation
\begin{equation}
\label{eq_inverse_wave} \Box \phi + V \phi = 0 \text{.}
\end{equation}
Roughly speaking, the main results of \cite{lor_cald} state that, under appropriate geometric assumptions on $( \mc{M}, g )$, the potential $V$ can be recovered from the Dirichlet-to-Neumann map on $\partial \mc{U}$.
\footnote{For consistency, we use notations corresponding to our current setting.}
A key step in their proof is a unique continuation argument from the boundary $\partial \mc{U} \cap \mc{D}$ to the full null cone exterior, $\mc{U} \cap \mc{D}$.
\footnote{In \cite{lor_cald}, however, the centre point $p$ always lies within $\mc{U}$.}
Thus, by replacing the unique continuation argument in \cite{lor_cald} by the one implied by our Carleman estimate, we can hence obtain a variant of the main result of \cite{lor_cald} that contains some improvements; see Theorem \ref{thm_inv_pblm_VX}.

For instance, the wave equations \eqref{eq_inverse_wave} treated in \cite{lor_cald} do not contain first-order terms.
This is due to the fact that the Carleman weight used in \cite{lor_cald} (defined using exact hyperquadrics) may not be strictly pseudoconvex.
By applying Theorem \ref{thm_carleman_rslt_LM}, we can extend the result in \cite{lor_cald} to cover wave equations \eqref{eq_obs_wv_LM} that also contain first-order terms.

On the other hand, the results obtained in this article are not strict improvements of those in \cite{lor_cald}, as the two unique continuation results impose rather different assumptions on the curvature.
\footnote{On the other hand, both \cite{lor_cald} and the present article assume geodesically star-shaped neighbourhoods.}
In \cite{lor_cald}, the authors assume a bound of the form
\begin{equation} \label{eq_lor_cald_LM}
g ( R (X,Y) X, Y ) \leq K [ g (X, X) g (Y, Y) - g (X, Y)^2 ] \text{,} \qquad K \in \R \text{,}
\end{equation}
for tangent vectors $X, Y$ in $\mc{M}$.
One advantage of \eqref{eq_lor_cald_LM} is that these are one-sided bounds, in contrast to the absolute value bounds for the curvature in \eqref{eq_curv_ass_LM}.

On the other hand, the inequalities \eqref{eq_lor_cald_LM} can behave rather strangely.
In particular, it is possible that a perturbation of a manifold satisfying \eqref{eq_lor_cald_LM} can then fail to satisfy \eqref{eq_lor_cald_LM} for any $K \in \R$; see the discussion in the introduction of \cite{lor_cald}.
This is in contrast to our curvature assumptions \eqref{eq_curv_ass_LM}, which are stable under perturbations.

Lastly, we mention that the recent \cite{lor_inv_new} improved (independently from our work) the results of \cite{lor_cald} by weakening the curvature condition \eqref{eq_lor_cald_LM}.
A further conformal argument also allows them to treat perturbations of this enlarged class of spacetimes.
\footnote{See also \cite{jena, Arick}, which used conformal perturbations to treat similar problems on Minkowski spacetime.}
In particular, they obtain an improved unique continuation result for $C^2$-perturbations of Minkowski spacetime, similar to that of this paper; this once again feeds into improved inverse problems results.

We note that while the results of \cite{lor_inv_new} applies to some settings not covered here, our techniques yield a quantitative uniqueness result with explicit control on constants.

\subsection{Ideas of the proof}

We now discuss the main ideas behind the proofs of our main results.
To motivate this, we first present some relevant ideas from \cite{Arick}, which treats the analogous problems in Minkowski geometry. Then, we discuss how these concepts can be generalised to Lorentzian manifolds, as well as the difficulties we face in this process.

First, the process of using Theorem \ref{thm_carleman_rslt_LM} to prove observability (and hence controllability) is analogous to the arguments found in \cite{Arick} by the second author.
In particular, this is roughly accomplished by using the large factor $a$ in the Theorem \ref{thm_carleman_rslt_LM} to absorb lower-order terms and then by applying the standard energy inequalities.
As a result, this is not a major concern in the present discussion, hence the main challenge (as well as the current focus) is to prove a suitable Carleman estimate on Lorentzian manifolds.

To prove the Carleman estimate, one must show \emph{pseudoconvexity} for a suitable function, say $\bar{f}$.
(This function plays a major role in constructing the Carleman weight).
Pseudoconvexity can be roughly viewed as a convexity condition, but only along special null directions; this in turn leads to a good sign for some crucial terms in the Carleman estimate.

In the following discussion, we will use $\bar{f}_M$ to represent the corresponding function $\bar{f}$ used in \cite{Arick}.
The starting point is the squared Minkowski ``distance" $f_M$ from a point, which one can show just barely fails to be pseudoconvex.
To overcome this, \cite{Arick} applied a conformal transformation into a perturbed geometry and then showed that the analogue of $f_M$ in this ``warped" setting is pseudoconvex.
As pseudoconvexity is conformally invariant, the above yields a perturbation $\bar{f}_M$ of $f_M$ that is pseudoconvex.
Note that, in this setting, everything can be calculated explicitly due to the simple properties of the Minkowski metric.

On the other hand, we do not have explicit formulas in our current setting of Lorentzian manifolds.
Therefore, the first step in the analysis is to find analogues of the functions $f_M$ and $\bar{f}_M$.
For the former, a natural choice is to consider the \emph{hyperquadric} function\footnote{See Definition \ref{def_hf_f_LM} for the precise definition. This $f$ can be seen as the squared Lorentzian distance.} $f$, using normal coordinates about a fixed point $p \in \mc{M}$.
In particular, $f$ is only valid when normal coordinates are well-defined.
Since our Carleman estimate relies crucially on $f$, this puts a fundamental restriction on where our result is applicable---namely, only when the considered domain lies within a single normal coordinate system.
\footnote{In fact, the null cone exterior $\mc{D}$ is most usefully characterized as $\mc{D} := \{ q \in \mc{M} | f(q) > 0 \}$.}

Furthermore, a crucial identity from \cite{Arick} is the following:
\footnote{Here, the index $\alpha$ is raised using the Minkowski metric.}
\begin{equation} \label{eq_delf_delf_LM}
\nb^\alpha f_M \nb_\alpha f_M = f_M \text{.}
\end{equation}
Fortunately, this carries over directly to the Lorentzian setting---by the Gauss lemma, \eqref{eq_delf_delf_LM} holds with $f_M$ replaced by the above $f$, and with the Minkowski metric replaced by $g$.

Now, in general, we do not have any information about the pseudoconvexity of $f$.
To get around this, we work instead with a perturbation $\bar{f}$ of $f$, defined as
\begin{equation} \label{eq_psc_LM}
\bar{f} = f \eta^{-1} \text{,} \qquad \eta := 1-\varepsilon t^2 \text{,}
\end{equation}
where $\varepsilon > 0$ and $t$ represents the \emph{time} normal coordinate. 
We show that under the curvature assumptions \eqref{eq_curv_ass_LM}, the positive level sets of $\bar{f}$ are indeed pseudoconvex with respect to $g$.

The $\eta$ in \eqref{eq_psc_LM} is convenient because it is everywhere smooth, and it is relatively simple to work with.
As long as the normal coordinates do not deviate too much from Cartesian coordinates in Minkowski geometry, $t^2$ must be convex in the timelike and null directions.
This implies that for $g$-null vector fields $\bar{X}$ that are tangent to level sets of $\bar{f}$, we have
\[ \nb_{\bar{X}\bar{X}} \eta < 0 \text{,} \]
which leads to
\[ \nb_{\bar{X}\bar{X}} \bar{f} > 0 \text{,} \]
which is precisely the pseudoconvexity of $\bar{f}$ \cite{ler_robb:unique}.
\footnote{See Proposition \ref{thm.pc_fb_hessian} for the precise relation between the Hessian of $f$ and $\bar{f}$, and the role of $\eta$.}

The curvature assumption \eqref{eq_curv_ass_LM} ensures $t$ indeed remains like its Minkowski analogue, hence $t^2$ maintains the required convexity.
Thus, the level sets of $\bar{f}$ in $\mc{D}$ are indeed pseudoconvex, and we can proceed with the Carleman estimate derivation.
In addition, most of the nice properties of $f$ are also carried over to $\bar{f}$.
Most notably, while the Lorentzian analogue of \eqref{eq_delf_delf_LM} no longer holds exactly for $\bar{f}$, the error terms capturing the deviation from \eqref{eq_delf_delf_LM} can be adequately controlled and absorbed into leading-order terms.
(Unfortunately, treating these error terms does involve some extensive computations.)

Finally, in the proof of Theorem \ref{thm_carleman_rslt_LM}, we work with special frames that are adapted to $f$ and $\bar{f}$.
More specifically, we define these frames---denoted $E_\rho, E_\theta, E_A$ ($1 \leq A < n$)---via parallel transport along radial geodesics from the centre $p$.
\footnote{These are the same frames as in the statement of Theorem \ref{thm_carleman_rslt_LM}.}
The advantage of such parallel frames is that they automatically have the desired properties.
In particular:
\begin{itemize}
\item $E_\rho$ is normal to level sets of $f$, while $E_\theta$ and $E_A$ are tangent to level sets of $f$.

\item The $E_\alpha$'s are everywhere orthogonal.

\item $E_\theta$ and $E_A$ remain timelike and spacelike, respectively.
\footnote{Although we do not have exact formula for $E_\theta$ and $E_A$ in terms of normal coordinates, the Gauss lemma yields an exact expression for $E_\rho$, which is crucial to the proof.}
\end{itemize}
This allows us to avoid also controlling the normal coordinate vector fields in our proof.

\subsection{Outline}

The rest of the article is divided as follows:
\begin{itemize}
\item In Section \ref{sec.hf}, we present the basic geometric setting for our work.
We define the hyperquadric function $f$, and use it to foliate the null cone exterior.
We also present several results that will be used in later sections.
In particular, we bound derivatives of $f$ and $t$; this is essential for obtaining pseudoconvexity.

\item In Section \ref{sec.pc}, we show that with the assumed curvature bounds, the level sets of $\bar{f}$ are pseudoconvex.
In particular, we define and prove estimates for the function $\eta$.

\item In Section \ref{sec.CE}, we derive our main Carleman estimate.
First, we will present several preliminary identities and estimates.
Then, we use this along with the results of Section \ref{sec.pc} to complete the derivation of the Carleman estimate itself.

\item In Section \ref{sec_control_LM}, we briefly discuss the process of using the Carleman estimate to prove the observability estimate.
We also present precise statements of our main observability and controllability results, as well as our application to inverse problems.

\end{itemize}

\section{The hyperquadric foliation} \label{sec.hf}

Just as before we will use lower-case Greek letters ($\alpha, \beta, \ldots$) to signify spacetime indices. Indices will be lowered and raised using $g$. Furthermore, throughout this article (particularly within proofs), we will use $\mc{C}$ to denote various universal constants whose values can change between lines. We will first define a hyperquadric function $f$ that is a very basic requirement for the proof. For this purpose, we will make use of normal coordinates.

\subsection{Normal coordinates} \label{ssec_norm_coord_LM}
We give a brief description of some concepts that we will use in the next few sections. For details, the reader may refer \cite[Chapter 3]{oneill}.

\begin{definition}
Let $p \in \mc{M}$. For any $v \in T_p(\mc{M})$, let $\gamma_v$ denote the maximal geodesic in $\mc{M}$ such that $\gamma_v(0) = p$ and $ \gamma_v'(0)=v$. Now, let $\mc{E}_p$ be defined as
\begin{equation}
\mc{E}_p := \{ v \in T_P(\mc{M} ) \ |\ \gamma_v \text{ is defined on an interval containing } [0,1] \} \text{.}
\end{equation}
Then, we define the exponential map, $\exp_p$, as follows
\begin{equation}
\exp_p : \mc{E}_p \rightarrow \mc{M} \text{,} \qquad \exp_p(v) := \gamma_v(1) \text{.}
\end{equation}
\end{definition}

\begin{definition}
Let $p\in\mc{M}$ be fixed, and let $(e_0, e_1, \dots, e_n)$ be an orthonormal basis in $T_p \mc{M}$. Let $\mc{V}$ be a geodesically star-shaped neighbourhood of $p$. Then, the normal coordinates $( x^0, x^1, \dots, x^n )$, with respect to $(e_0, e_1, \dots, e_n)$, are given by
\[ \exp_p^{-1}(q) = \sum x^i(q) e_i \text{,} \qquad q \in \mc{V} \text{.} \]
\end{definition}
We state the Gauss lemma for the reader's convenience; see \cite[Chapter 5]{oneill} for details.
\begin{lemma}[Gauss Lemma] \label{lemma_gauss_LM}
Let $p \in \mc{M}$, and $0 \neq x \in T_p(\mc{M})$. If $ u_x, v_x \in T_x(T_p\mc{M})$ such that $u_x$ is radial, then
\begin{equation}
\langle d \exp_p(u_x) , d \exp_p(v_x) \rangle = \langle u_x , v_x \rangle \text{.}
\end{equation}
\end{lemma}

\subsection{The hyperquadric function}

In this section, we discuss some properties of a function $f$ that is used to define the Carleman weight. We have to be particularly careful about the Lorentzian manifold setting and define this function accordingly.

Fix a point $p \in \mc{M}$, as well as an orthonormal basis
\begin{equation}
\label{eq.hf_initial} e_0, e_1, \dots, e_n \in T_p \mc{M} \text{.}
\end{equation}
For notational convenience, we will assume henceforth that $\mc{M}$ itself is geodesically star-shaped about $p$ (i.e.\ every $q \in \mc{M}$ is connected to $p$ by a unique geodesic).
Moreover, let $( x^0, x^1, \dots, x^n )$ denote the normal coordinates, constructed from \eqref{eq.hf_initial}, about $p$ on $\mc{M}$.
From this, we can construct normal polar coordinates
\begin{equation}
\label{eq.hf_nc} t := x^0 \text{,} \qquad r^2 := \sum_{ k = 1 }^n ( x^k )^2 \text{,} \qquad \omega_{ t, r } = r^{-1} ( x^1, \dots, x^k ) \text{,}
\end{equation}
on $\mc{M}$. We also construct the vector fields $\partial_t$ and $\partial_r$ corresponding to the above coordinates $t$ and $r$, respectively.

\begin{definition} \label{def_hf_f_LM}
Define the \textbf{hyperquadric function} $f$ as follows
\begin{equation}
\label{eq.hf_f} f \in C^\infty ( \mc{M} ) \text{,} \qquad f := \frac{1}{4} ( r^2 - t^2 ) \text{.}
\end{equation}
Further, define the exterior domain $\mc{D}$, as follows
\begin{equation}
\label{eq.hf_D} \mc{D} := \{ q \in \mc{M} \mid f (q) > 0 \} \text{.}
\end{equation}
For convenience, we will also make use of the following function on $\mc{D}$
\begin{equation}
\label{eq.hf_rho} \rho \in C^\infty ( \mc{D} ) \text{,} \qquad \rho := 2 \sqrt{f} = \sqrt{ r^2 - t^2 } \text{.}
\end{equation}
\end{definition}
For the majority of the work, we will be concerned with the region $\mc{D}$. 

\begin{proposition} \label{thm.hf_gauss}
The following identities for $f$ and $\rho$ hold on $\mc{M}$:
\begin{equation}
\label{eq.hf_gauss} \nabla^\sharp f = \frac{1}{2} ( t \partial_t + r \partial_r ) \text{,} \qquad \nabla^\alpha f \nabla_\alpha f = f \text{.}
\end{equation}
Furthermore, the following transport equations hold on $\mc{M}$:
\begin{align}
\label{eq.hf_gauss_2} \nabla^\mu f \nabla_{ \mu \alpha } f &= \frac{1}{2} \nabla_\alpha f \text{,} \\
\notag \nabla^\mu f \nabla_{ \mu \alpha \beta } f &= \frac{1}{2} \nabla_{ \alpha \beta } f - \nabla_\alpha{}^\mu f \nabla_{ \beta \mu } f - R_{ \mu \alpha \nu \beta } \nabla^\mu f \nabla^\nu f \text{,} \\
\notag \nabla^\mu f \nabla_{ \mu \alpha \beta \lambda } f &= \frac{1}{2} \nabla_{ \alpha \beta \lambda } f - \nabla_\lambda{}^\mu f \nabla_{ \alpha \beta \mu } f - \nabla_\alpha{}^\mu f \nabla_{ \mu \beta \lambda } f - \nabla_\beta{}^\mu f \nabla_{ \alpha \mu \lambda } f \\
\notag &\qquad - \nabla_\alpha R_{ \mu \beta \nu \lambda } \nabla^\mu f \nabla^\nu f - R_{ \mu \lambda \nu \beta } \nabla_\alpha{}^\nu f \nabla^\mu f - R_{ \mu \beta \nu \lambda } \nabla_\alpha{}^\nu f \nabla^\mu f \\
\notag &\qquad - R_{ \mu \alpha \nu \beta } \nabla_\lambda {}^\nu f \nabla^\mu f - R_{ \mu \alpha \nu \lambda } \nabla_\beta{}^\nu f \nabla^\mu f \text{.}
\end{align}
\end{proposition}

\begin{proof}
First, \eqref{eq.hf_gauss} follows from the Gauss lemma---both expressions have the same values as in Minkowski spacetime and Cartesian coordinates.
Using \eqref{eq.hf_gauss}, we then compute
\[
\nabla_{ \alpha \beta } f \nabla^\beta f = \frac{1}{2} \nabla_\alpha ( \nabla^\beta f \nabla_\beta f ) = \frac{1}{2} \nabla_\alpha f \text{,}
\]
which is precisely the first identity in \eqref{eq.hf_gauss_2}.

Next, commuting derivatives and recalling the first part of \eqref{eq.hf_gauss_2}, we obtain
\begin{align*}
\nabla^\mu f \nabla_{ \mu \alpha \beta } f &= \nabla^\mu f \nabla_{ \alpha \beta \mu } f - R_{ \mu \alpha \nu \beta } \nabla^\mu f \nabla^\nu f \\
&= \nabla_\alpha ( \nabla^\mu f \nabla_{ \beta \mu } f ) - \nabla_\alpha{}^\mu f \nabla_{ \beta \mu } f - R_{ \mu \alpha \nu \beta } \nabla^\mu f \nabla^\nu f \\
&= \frac{1}{2} \nabla_{ \alpha \beta } f - \nabla_\alpha{}^\mu f \nabla_{ \beta \mu } f - R_{ \mu \alpha \nu \beta } \nabla^\mu f \nabla^\nu f \text{,}
\end{align*}
which is the second part of \eqref{eq.hf_gauss_2}.
Finally, taking another derivative yields
\begin{align*}
\nabla^\mu f \nabla_{ \mu \alpha \beta \lambda } f &= \nabla^\mu f \nabla_{ \alpha \mu \beta \lambda } f - R_{ \mu \alpha \nu \beta } \nabla^\mu f \nabla_\lambda {}^\nu f - R_{ \mu \alpha \nu \lambda } \nabla^\mu f \nabla_\beta{}^\nu f \\
&= \nabla_\alpha ( \nabla^\mu f \nabla_{ \mu \beta \lambda } f ) - \nabla_\alpha{}^\mu f \nabla_{ \mu \beta \lambda } f - R_{ \mu \alpha \nu \beta } \nabla^\mu f \nabla_\lambda {}^\nu f - R_{ \mu \alpha \nu \lambda } \nabla^\mu f \nabla_\beta{}^\nu f \\
&= \frac{1}{2} \nabla_{ \alpha \beta \lambda } f - \nabla_\alpha ( \nabla_\beta{}^\mu f \nabla_{ \lambda \mu } f + R_{ \mu \beta \nu \lambda } \nabla^\mu f \nabla^\nu f ) - \nabla_\alpha{}^\mu f \nabla_{ \mu \beta \lambda } f \\
&\qquad - R_{ \mu \alpha \nu \beta } \nabla^\mu f \nabla_\lambda {}^\nu f - R_{ \mu \alpha \nu \lambda } \nabla^\mu f \nabla_\beta{}^\nu f \\
&= \frac{1}{2} \nabla_{ \alpha \beta \lambda } f - \nabla_\lambda{}^\mu f \nabla_{ \alpha \beta \mu } f - \nabla_\alpha{}^\mu f \nabla_{ \mu \beta \lambda } f - \nabla_\beta{}^\mu f \nabla_{ \alpha \mu \lambda } f \\
&\qquad - \nabla_\alpha R_{ \mu \beta \nu \lambda } \nabla^\mu f \nabla^\nu f - R_{ \mu \lambda \nu \beta } \nabla_\alpha{}^\nu f \nabla^\mu f - R_{ \mu \beta \nu \lambda } \nabla_\alpha{}^\nu f \nabla^\mu f \\
&\qquad - R_{ \mu \alpha \nu \beta } \nabla_\lambda {}^\nu f \nabla^\mu f - R_{ \mu \alpha \nu \lambda } \nabla_\beta{}^\nu f \nabla^\mu f \text{,}
\end{align*}
which proves the last identity in \eqref{eq.hf_gauss_2}.
\end{proof}
As mentioned in the introduction, for obtaining the pseudoconvexity condition, we will be working with a function $\eta := \eta(t)$. For this purpose, we will need to find derivatives of $t$. Hence, now we prove some transport equations for the coordinate $t$. 
\begin{proposition} \label{thm.hf_transport}
The following identities hold on $\mc{M}$:
\begin{align}
\label{eq.hf_transport} \nabla^\mu f \nabla_\mu t &= \frac{1}{2} t \text{,} \\
\notag \nabla^\mu f \nabla_{ \mu \alpha } t &= \frac{1}{2} \nabla_\alpha t - \nabla_\alpha{}^\mu f \nabla_\mu t \text{,} \\
\notag \nabla^\mu f \nabla_{ \mu \alpha \beta } t^2 &= \nabla_{ \alpha \beta } t^2 - \nabla_\alpha{}^\mu f \nabla_{ \beta \mu } t^2 - \nabla_\beta{}^\mu f \nabla_{ \alpha \mu } t^2 - \nabla_{ \alpha \beta \mu } f \nabla^\mu t^2 - R_{ \mu \alpha \nu \beta } \nabla^\mu f \nabla^\nu t^2 \text{.}
\end{align}
\end{proposition}

\begin{proof}
The first equation follows from of the first part of \eqref{eq.hf_gauss}.
From this, we also obtain
\begin{align*}
\nabla^\mu f \nabla_{ \mu \alpha } t &= \nabla_\alpha ( \nabla^\mu f \nabla_\mu t ) - \nabla_\alpha{}^\mu f \nabla_\mu t \\
&= \frac{1}{2} \nabla_\alpha t - \nabla_\alpha{}^\mu f \nabla_\mu t \text{,}
\end{align*}
which proves the second identity.
Furthermore, similar computations yield
\[
\nabla^\mu f \nabla_\mu t^2 = t^2 \text{,} \qquad \nabla^\mu f \nabla_{ \mu \alpha } t^2 = \nabla_\alpha t^2 - \nabla_\alpha{}^\mu f \nabla_\mu t^2 \text{.}
\]
Finally, commuting derivatives and applying the above, we obtain
\begin{align*}
\nabla^\mu f \nabla_{ \mu \alpha \beta } t^2 &= \nabla^\mu f \nabla_{ \alpha \beta \mu } t^2 - R_{ \mu \alpha \nu \beta } \nabla^\mu f \nabla^\nu t^2 \\
&= \nabla_\alpha ( \nabla^\mu f \nabla_{ \mu \beta } t^2 ) - \nabla_\alpha{}^\mu f \nabla_{ \beta \mu } t^2 - R_{ \mu \alpha \nu \beta } \nabla^\mu f \nabla^\nu t^2 \\
&= \nabla_{ \alpha \beta } t^2 - \nabla_\alpha ( \nabla_\beta{}^\mu f \nabla_\mu t^2 ) - \nabla_\alpha{}^\mu f \nabla_{ \beta \mu } t^2 - R_{ \mu \alpha \nu \beta } \nabla^\mu f \nabla^\nu t^2 \\
&= \nabla_{ \alpha \beta } t^2 - \nabla_\alpha{}^\mu f \nabla_{ \beta \mu } t^2 - \nabla_\beta{}^\mu f \nabla_{ \alpha \mu } t^2 - \nabla_{ \alpha \beta \mu } f \nabla^\mu t^2 - R_{ \mu \alpha \nu \beta } \nabla^\mu f \nabla^\nu t^2 \text{,}
\end{align*}
which is the last part of \eqref{eq.hf_transport}.
\end{proof}

\subsection{Hyperquadric frames} \label{ssec_hyperq_func_LM}

Our next task is to construct frames on $\mc{D}$ that are adapted to the level sets of $f$. This will simplify the computations by ensuring that no (or minimal) cross terms arise from the coordinates when we raise/lower indices using $g$.

\begin{definition}
Let $e_0, \dots, e_n \in T_p \mc{M}$ be the basis from \eqref{eq.hf_initial}. Then, for each $\omega := ( \omega^0, \omega^1, \dots, \omega^n ) \in ( -1, 1 ) \times \Sph^{ n - 1 }$, we define
\begin{equation}
\label{eq.hf_e} e_{ r, \omega } := \sum_{ k = 1 }^n \omega^k e_k \text{,} \qquad e_{ \rho, \omega } := e_{ r, \omega } + \omega^0 e_0 \text{,} \qquad e_{ \theta, \omega } := e_0 + \omega^0 e_{ r, \omega } \text{.}
\end{equation}
In addition, we define the subset
\begin{equation}
\label{eq.hf_H} H_\omega := \{ v \in T_p \mc{M} \mid g ( v, e_0 ) = g ( v, e_{ r, \omega } ) = 0 \} \text{.}
\end{equation}
\end{definition}
Direct computations using \eqref{eq.hf_e} then yield
\begin{equation}
\label{eq.hf_eo} g ( e_{ \rho, \omega }, e_{ \rho, \omega } ) = 1 - ( \omega^0 )^2 \text{,} \qquad g ( e_{ \theta, \omega }, e_{ \theta, \omega } ) = -1 + ( \omega^0 )^2 \text{,} \qquad g ( e_{ \rho, \omega }, e_{ \theta, \omega } ) = 0 \text{.}
\end{equation}
In addition, we let $\gamma_\omega$ denote the radial geodesic satisfying
\begin{equation}
\label{eq.hf_geodesic} \gamma_\omega (0) := p \text{,} \qquad \gamma_\omega' (0) = e_{ \rho, \omega } \text{.}
\end{equation}
We then let $\Gamma$ denote the family of all such geodesics,
\begin{equation}
\label{eq.hf_geodesic_set} \Gamma := \{ \gamma_\omega \mid \omega \in ( -1, 1 ) \times \Sph^{ n - 1 } \} \text{.}
\end{equation}

We now define $E_0$, $E_\rho$, $E_\theta$, and $\mc{H}$ on each $\gamma_\omega \in \Gamma$ to be the parallel transports of $e_0$, $e_{ \rho, \omega }$, $e_{ \theta, \omega }$, and $H_\omega$, respectively, along $\gamma_\omega$.
This results in smooth vector fields $E_0$, $E_\rho$, and $E_\theta$ on $\mc{D}$, as well as an $(n - 1)$-dimensional smooth distribution $\mc{H}$ on $\mc{D}$.

Finally, we complete $E_\rho$ and $E_\theta$ using local ($g$-)orthonormal frames $E_1, \dots, E_{ n - 1 } \in \mc{H}$, which we can also assume to be parallel transported along each $\gamma_\omega \in \Gamma$.
From now on, we will use uppercase Latin letters $A, B, \dots$ to denote frame indices in the $\mc{H}$-components (e.g.\ $E_A \in \mc{H}$), taking values between $1$ and $n - 1$, inclusive. The next proposition gives us more details about the frames $E_\rho, E_\theta, E_A$.

\begin{proposition} \label{thm.hf_E}
The following properties hold along each $\gamma_\omega \in \Gamma$:
\begin{itemize}
\item The following identities hold:
\begin{equation}
\label{eq.hf_E_rho} \frac{t}{r} = \omega^0 \text{,} \qquad \frac{ \rho^2 }{ r^2 } = 1 - ( \omega^0 )^2 \text{,} \qquad E_\rho = \partial_r + \frac{t}{r} \partial_t \text{.}
\end{equation}

\item $E_\rho$, $E_\theta$, $E_A$ forms a (local) orthogonal frame.

\item The following identities hold:
\begin{equation}
\label{eq.hf_E} g ( E_\rho, E_\rho ) = \frac{ \rho^2 }{ r^2 } \text{,} \qquad g ( E_\theta, E_\theta ) = - \frac{ \rho^2 }{ r^2 } \text{,} \qquad g ( E_A, E_B ) = \delta_{ A B } \text{.}
\end{equation}

\item $E_\rho$ is normal to the level sets of $f$, while $E_\theta$, $E_A$ are tangent to the level sets of $f$.

\item The limits of $E_\rho$, $E_\theta$, $E_A$ at $p$ are uniformly bounded with respect to $\omega$.

\item The following identity holds:
\begin{equation}
\label{eq.hf_E_special} E_\theta = \frac{t}{r} E_\rho + \frac{ \rho^2 }{ r^2 } E_0 \text{.}
\end{equation}
\end{itemize}
\end{proposition}

\begin{proof}
First, by the Gauss lemma, $\nabla^\sharp f$ points in the same direction as $\gamma_\omega'$.
As a result, that $\frac{t}{r}$ is constant along $\gamma_\omega$ follows immediately from the first part of \eqref{eq.hf_gauss}:
\[
( \nabla^\sharp f ) \left( \frac{t}{r} \right) = \frac{1}{2} ( r \partial_r + t \partial_t ) \left( \frac{t}{r} \right) = 0 \text{.}
\]
Moreover, by the definition \eqref{eq.hf_nc} of our normal polar coordinates, we have that
\begin{equation}
\label{eql.hf_E_1} \frac{2}{r} \nabla^\sharp f = \partial_r + \frac{t}{r} \partial_t \rightarrow e_{ r, \omega } + \frac{t}{r} e_0 \text{,}
\end{equation}
as one tends to $p$ along $\gamma_\omega$.
Since the left-hand side of \eqref{eql.hf_E_1} points in the same direction as $\gamma_\omega'$, the right-hand side of \eqref{eql.hf_E_1} also points in the same direction as $e_{ \rho, \omega }$, and thus the definition \eqref{eq.hf_e} of $e_{ \rho, \omega }$ yields the first part of \eqref{eq.hf_E_rho}.
In addition, we have that
\[
\frac{ \rho^2 }{ r^2 } = \frac{ r^2 - t^2 }{ r^2 } = 1 - ( \omega^0 )^2 \text{,}
\]
which is the second part of \eqref{eq.hf_E_rho}.
Finally, since $E_\rho$ is radial, then the Gauss lemma and the first part of \eqref{eq.hf_E_rho} together yield the third part of \eqref{eq.hf_E_rho}.

Next, that $E_\rho$, $E_\theta$, $E_A$ are orthogonal follows from their definitions via parallel transport.
Moreover, since both $E_\rho$ and $\nabla^\sharp f$ are radial, then $E_\rho$ is normal to the level sets of $f$, and hence $E_\theta$ and $E_A$ are everywhere tangent to the level sets of $f$.
Notice also that the limits of $E_\rho$, $E_\theta$, and $E_A$ at $p$ along $\gamma_\omega$ are, by definition, $e_{ \rho, \omega }$, $e_{ \theta, \omega }$, and a unit vector in $H_\omega$, respectively, which are indeed uniformly bounded for all $\omega \in ( -1, 1 ) \times \Sph^{ n - 1 }$.

Next, using \eqref{eq.hf_eo}, \eqref{eq.hf_E_rho}, and the parallel transport definition of $E_\rho, E_\theta, E_A$ shows that \eqref{eq.hf_E} holds.

Finally, for \eqref{eq.hf_E_special}, we first notice that
\[
\omega^0 e_{ \rho, \omega } + [ 1 - ( \omega^0 )^2 ] e_0 = \omega^0 ( e_{r,\omega} + \omega^0 e_0 ) + e_0 - ( \omega^0 )^2 e_0 = e_{ \theta, \omega } \text{.}
\]
Parallel transporting the above along $\gamma_\omega$ and recalling \eqref{eq.hf_E_rho} then yields \eqref{eq.hf_E_special}.
\end{proof}

\begin{remark}
In the special case of Minkowski spacetime $( \mc{M}, g ) := ( \R^{1+n}, g_0 )$, we have
\begin{equation}
\label{eq.hf_E_mink} E_\rho = \partial_r + \frac{t}{r} \partial_t \text{,} \qquad E_\theta = \partial_t + \frac{t}{r} \partial_r \text{,}
\end{equation}
and we have that the $E_A$'s are always tangent to level sets of $( t, r )$.
\end{remark}

In subsequent sections, we will adopt the following frame indexing conventions:
\begin{itemize}
\item Uppercase Latin letters are with respect to the frames $\{ E_A \}$ and their coframes.

\item Lowercase Latin letters are with respect to the frames $\{ E_\theta, E_A \}$ and their coframes.

\item We use the indices $\theta$ and $\rho$ to refer to $E_\theta$ and $E_\rho$, respectively.
\end{itemize}
Furthermore, we again adopt Einstein summation notation---repeated Latin indices denote summations over all above-mentioned frame elements.

\subsection{Vertex limits}

Next, we derive limits for derivatives of $f$ at the center $p$.
The following propositions follow from standard properties of normal coordinates:

\begin{proposition} \label{thm.hf_f_limit}
The following identities hold at $p$:
\begin{equation}
\label{eq.hf_f_limit} \nabla f |_p = 0 \text{,} \qquad \nabla^2 f |_p = \frac{1}{2} g |_p \text{,} \qquad \nabla^3 f |_p = 0 \text{.}
\end{equation}
In addition, the following identity holds at $p$:
\begin{equation}
\label{eq.hf_f_limit_2} \nabla_{ \alpha \beta \lambda \delta } f |_p = - \frac{1}{6} ( R_{ \alpha \lambda \beta \delta } + R_{ \alpha \delta \beta \lambda } ) |_p \text{.}
\end{equation}
\end{proposition}

\begin{proof}
Throughout the proof, we let $\gamma$ be any radial geodesic in $\mc{M}$, with $\gamma (0) = p$.
The first equation in \eqref{eq.hf_f_limit} follows immediately from the first part of \eqref{eq.hf_gauss}.
Moreover, since $\nabla^\sharp f$ is radial, then it follows from the first equation in \eqref{eq.hf_gauss_2} that, along $\gamma$,
\[
\nabla_{ \alpha \beta } f \, ( \gamma' )^\beta = \frac{1}{2} ( \gamma' )_\alpha = \frac{1}{2} g_{ \alpha \beta } \, ( \gamma' )^\beta \text{.}
\]
Applying the above at $\gamma (0)$ and for all radial geodesics $\gamma$ yields the second part of \eqref{eq.hf_f_limit}.
Next, we evaluate the third equation of \eqref{eq.hf_gauss_2} at $p$ and apply the first two parts of \eqref{eq.hf_f_limit}:
\begin{align*}
0 &= \frac{1}{2} \nabla_{ \alpha \beta \lambda } f |_p - \frac{1}{2} ( g_\lambda{}^\mu \nabla_{ \alpha \beta \mu } f ) |_p - \frac{1}{2} ( g_\alpha{}^\mu \nabla_{ \mu \beta \lambda } f ) |_p - \frac{1}{2} ( g_\beta{}^\mu \nabla_{ \alpha \mu \lambda } f ) |_p \\
&= - \nabla_{ \alpha \beta \lambda } f |_p \text{.}
\end{align*}
This is precisely the last part of \eqref{eq.hf_f_limit}.

Finally, for \eqref{eq.hf_f_limit_2}, we first observe that
\begin{align*}
\nabla^\mu f \nabla_{ \mu \alpha \beta \lambda \delta } f &= \nabla^\mu f \nabla_{ \alpha \mu \beta \lambda \delta } f - R_{ \mu \alpha \nu \beta } \nabla^\nu{}_{ \lambda \delta } f - R_{ \mu \alpha \nu \lambda } \nabla_\beta{}^\nu{}_\delta f - R_{ \mu \alpha \nu \delta } \nabla_{ \beta \gamma }{}^\nu f \\
&= \nabla_\alpha ( \nabla^\mu f \nabla_{ \mu \beta \lambda \delta } f ) - \nabla_\alpha{}^\mu f \nabla_{ \mu \beta \lambda \delta } f - R_{ \mu \alpha \nu \beta } \nabla^\nu{}_{ \lambda \delta } f \\
&\qquad - R_{ \mu \alpha \nu \lambda } \nabla_\beta{}^\nu{}_\delta f - R_{ \mu \alpha \nu \delta } \nabla_{ \beta \gamma }{}^\nu f \text{.}
\end{align*}
The first term can be further expanded by applying the last equation in \eqref{eq.hf_gauss_2} and applying the product rule.
Evaluating at $p$ and recalling from \eqref{eq.hf_f_limit} that any term containing one or three covariant derivatives of $f$ vanish, we then obtain from the above that
\begin{align*}
0 &= \frac{1}{2} \nabla_{ \alpha \beta \lambda \delta } f |_p - ( \nabla_\delta{}^\mu f \nabla_{ \alpha \beta \lambda \mu } f ) |_p - ( \nabla_\beta{}^\mu f \nabla_{ \alpha \mu \lambda \delta } f ) |_p - ( \nabla_\lambda{}^\mu f \nabla_{ \alpha \beta \mu \delta } f ) |_p \\
&\qquad - ( R_{ \mu \delta \nu \lambda } \nabla_\beta{}^\nu f \nabla_\alpha{}^\mu f ) |_p - ( R_{ \mu \lambda \nu \delta } \nabla_\beta{}^\nu f \nabla_\alpha{}^\mu f ) |_p - ( R_{ \mu \beta \nu \lambda } \nabla_\delta{}^\nu f \nabla_\alpha{}^\mu f ) |_p \\
&\qquad - ( R_{ \mu \beta \nu \delta } \nabla_\lambda{}^\nu f \nabla_\alpha{}^\mu f ) |_p - ( \nabla_\alpha{}^\mu f \nabla_{ \mu \beta \lambda \delta } f ) |_p \\
&= \frac{1}{2} \nabla_{ \alpha \beta \lambda \delta } f |_p - \frac{1}{2} \nabla_{ \alpha \beta \lambda \delta } f |_p - \frac{1}{2} \nabla_{ \alpha \beta \lambda \delta } f |_p - \frac{1}{2} \nabla_{ \alpha \beta \lambda \delta } f |_p - \frac{1}{2} \nabla_{ \alpha \beta \lambda \delta } f |_p \\
&\qquad - \frac{1}{4} R_{ \alpha \delta \beta \lambda } |_p - \frac{1}{4} R_{ \alpha \lambda \beta \delta } |_p - \frac{1}{4} R_{ \alpha \beta \delta \lambda } |_p - \frac{1}{4} R_{ \alpha \beta \lambda \delta } |_p \\
&= - \frac{3}{2} \nabla_{ \alpha \beta \lambda \delta } f |_p + \frac{1}{4} R_{ \alpha \delta \lambda \beta } |_p + \frac{1}{4} R_{ \alpha \lambda \delta \beta } |_p \text{.}
\end{align*}
Rearranging the above results in \eqref{eq.hf_f_limit_2}.
\end{proof}

\begin{proposition} \label{thm.hf_f_asymp}
Let $\gamma_\omega \in \Gamma$, let $x, y, z \in T_p \mc{M}$, and let $X, Y, Z$ denote the parallel transports of $x, y, z$ along $\gamma_\omega$, respectively.
Then, the following limits hold along $\gamma_\omega$:
\begin{align}
\label{eq.hf_f_asymp} \lim_p r^{-2} \left( \nabla^2 f - \frac{1}{2} g \right) ( X, Y ) &= - \frac{1}{6} R |_p ( e_{ \rho, \omega }, x, e_{ \rho, \omega }, y ) \text{,} \\
\notag \lim_p r^{-1} \, \nabla^3 f ( X, Y, Z ) &= - \frac{1}{6} R |_p ( e_{ \rho, \omega }, y, x, z ) - \frac{1}{6} R |_p ( e_{ \rho, \omega }, z, x, y ) \text{.}
\end{align}
\end{proposition}

\begin{proof}
The last part of \eqref{eq.hf_f_limit}, \eqref{eq.hf_f_limit_2}, and the fact that $E_\rho$ is parallel transported yield
\begin{align*}
\left. E_\rho \left[ \left( \nabla^2 f - \frac{1}{2} g \right) ( X, Y ) \right] \right|_p &= 0 \text{,} \\
\left. E_\rho^2 \left[ \left( \nabla^2 f - \frac{1}{2} g \right) ( X, Y ) \right] \right|_p &= - \frac{1}{3} R |_p ( e_{ \rho, \omega }, x, e_{ \rho, \omega }, y ) \text{.}
\end{align*}
The first identity in \eqref{eq.hf_f_asymp} now follows by applying Taylor's theorem along $\gamma_\omega$, and by noting from \eqref{eq.hf_E_rho} that $E_\rho r = 1$.
Similarly, another application of \eqref{eq.hf_f_limit_2} yields
\[
E_\rho [ \nabla^3 f ( X, Y, Z ) ] |_p = - \frac{1}{6} R |_p ( e_{ \rho, \omega }, y, x, z ) - \frac{1}{6} R |_p ( e_{ \rho, \omega }, z, x, y ) \text{.}
\]
The second equation in \eqref{eq.hf_f_asymp} then follows from Taylor's theorem and the above.
\end{proof}

\begin{remark}
A variant of \eqref{eq.hf_f_asymp} holds along every geodesic in $\mc{M}$.
Here, the restriction to $\mc{D}$ only arises since we are parametrizing the radial geodesics by $r$.
\end{remark}

\begin{corollary} \label{thm.hf_f_asymp_2}
The following limits hold along any $\gamma_\omega \in \Gamma$:
\begin{align}
\label{eq.hf_f_asymp_2} \lim_p r^{-2} \left( \nabla^2 f - \frac{1}{2} g \right) ( E_A, E_B ) &= - \frac{1}{6} R |_p ( e_{ \rho, \omega }, e_{ A, \omega }, e_{ \rho, \omega }, e_{ B, \omega } ) \text{,} \\
\notag \lim_p r^{-2} \left( \nabla^2 f - \frac{1}{2} g \right) ( E_\theta, E_A ) &= - \frac{1}{6} \frac{ \rho^2 }{ r^2 } \, R |_p ( e_{ \rho, \omega }, e_0, e_{ \rho, \omega }, e_{ A, \omega } ) \text{,} \\
\notag \lim_p r^{-2} \left( \nabla^2 f - \frac{1}{2} g \right) ( E_\theta, E_\theta ) &= - \frac{1}{6} \frac{ \rho^4 }{ r^4 } \, R |_p ( e_{ \rho, \omega }, e_0, e_{ \rho, \omega }, e_0 ) \text{.}
\end{align}
Furthermore, we have, along any $\gamma_\omega$, that
\begin{align}
\notag \lim_p r^{-1} \nabla^3 f ( E_A, E_B, E_C ) &= - \frac{1}{6} [ R |_p ( e_{ \rho, \omega }, e_{ B, \omega }, e_{ A, \omega }, e_{ C, \omega } ) + R |_p ( e_{ \rho, \omega }, e_{ C, \omega }, e_{ A, \omega }, e_{ B, \omega } ) ] \text{,} \\
\notag \lim_p r^{-1} \nabla^3 f ( E_\theta, E_B, E_C ) &= - \frac{1}{6} [ R |_p ( e_{ \rho, \omega }, e_{ B, \omega }, e_{ \theta, \omega }, e_{ C, \omega } ) + R |_p ( e_{ \rho, \omega }, e_{ C, \omega }, e_{ \theta, \omega }, e_{ B, \omega } ) ] \text{,} \\
\notag \lim_p r^{-1} \nabla^3 f ( E_A, E_\theta, E_C ) &= - \frac{1}{6} [ R |_p ( e_{ \rho, \omega }, e_{ \theta, \omega }, e_{ A, \omega }, e_{ C, \omega } ) + R |_p ( e_{ \rho, \omega }, e_{ C, \omega }, e_{ A, \omega }, e_{ \theta, \omega } ) ] \text{,} \\
\notag \lim_p r^{-1} \nabla^3 f ( E_\theta, E_\theta, E_C ) &= - \frac{1}{6} \frac{ \rho^2 }{ r^2 } \, R |_p ( e_{ \rho, \omega }, e_0, e_{ \theta, \omega }, e_{ C, \omega } ) \text{,} \\
\notag \lim_p r^{-1} \nabla^3 f ( E_A, E_\theta, E_\theta ) &= - \frac{1}{6} \frac{ \rho^2 }{ r^2 } \, R |_p ( e_{ \rho, \omega }, e_0, e_{ A, \omega }, e_{ \theta, \omega } ) \text{,} \\
\label{eq.hf_f_asymp_3} \lim_p r^{-1} \nabla^3 f ( E_\theta, E_\theta, E_\theta ) &= 0 \text{.}
\end{align}
\end{corollary}

\begin{proof}
These follow by combining \eqref{eq.hf_f_asymp} with \eqref{eq.hf_E_special}.
\end{proof}

\subsection{Estimates for $f$}

Our next task is establish estimates for the hyperquadric function $f$, under certain geometric assumptions.
\begin{definition}
Let $r_0 > 0$ be fixed. Then, we define the region $\mc{D}_{r_0}$ as follows
\begin{equation}
\label{eq.hf_DR} \mc{D}_{ r_0 } := \mc{D} \cap \{ r < r_0 \} \text{.}
\end{equation}
\end{definition}

The estimates we obtain will be on the region $\mc{D}_{r_0}$. Now, our key curvature assumption will be the following:

\begin{assumption} \label{ass.hf_curv}
The following hold on $\mc{D}_{ r_0 }$, with constants $r_0 > 0$ and $\mc{C}_0, \mc{C}_1 > 0$:
\footnote{Note we use $E_0$ rather than $E_\theta$ in \eqref{eq.hf_curv} and \eqref{eq.hf_curvd}. This is because $E_\theta$ and $E_\rho$ become the same vector as one approaches $\partial \mc{D}$, hence the frame $\{ E_\rho, E_\theta, E_A \}$ does not generate all the directions in a uniform manner.}
\begin{align}
\label{eq.hf_curv} \sup_{ X, Y, Z \in \{ E_\rho, E_0, E_A \} } | R ( E_\rho, X, Y, Z ) | & \leq \frac{ \mc{C}_0 }{ n r_0^2 } \text{,} \\
\label{eq.hf_curvd} \sup_{ X, Y, Z \in \{ E_\rho, E_0, E_A \} } | \nabla_Z R ( E_\rho, X, E_\rho, Y ) | & \leq \frac{ \mc{C}_1 }{ n r_0^3 } \text{.}
\end{align}
\end{assumption}

For convenience, we also define the deviation of $\nabla^2 f$ from its Minkowski value:
\begin{equation}
\label{eq.hf_q} q := \nabla^2 f - \frac{1}{2} g \text{.}
\end{equation}
We can now state our main estimates for derivatives of $f$.

\begin{proposition} \label{thm.hf_q_est}
There is a universal constant $\mc{C}_\ast > 0$ such that if Assumption \ref{ass.hf_curv} holds and $0 < \mc{C}_0 < \mc{C}_\ast$, then the following inequalities also hold on $\mc{D}_{ r_0 }$:
\begin{equation}
\label{eq.hf_q_est} | q ( E_A, E_B ) | \leq \frac{ \mc{C}_0 }{ 3 n } \frac{ r^2 }{ r_0^2 } \text{,} \qquad | q ( E_\theta, E_A ) | \leq \frac{ \rho^2 }{ r^2 } \, \frac{ \mc{C}_0 }{ 3 n } \frac{ r^2 }{ r_0^2 } \text{,} \qquad | q ( E_\theta, E_\theta ) | \leq \frac{ \rho^4 }{ r^4 } \, \frac{ \mc{C}_0 }{ 3 n } \frac{ r^2 }{ r_0^2 } \text{.}
\end{equation}
\end{proposition}

\begin{proof}
First, note from \eqref{eq.hf_q} that the second identity in \eqref{eq.hf_gauss_2} can be rewritten as
\[
\nabla^\mu f \nabla_\mu q_{ \alpha \beta } = - \frac{1}{2} q_{ \alpha \beta } - q_\alpha{}^\mu q_{ \beta \mu } - R_{ \mu \alpha \nu \beta } \nabla^\mu f \nabla^\nu f \text{.}
\]
Since $E_\rho = 2 r^{-1} \nabla^\sharp f$ and $E_\rho r = 1$, by \eqref{eq.hf_gauss} and \eqref{eq.hf_E_rho}, the above becomes
\begin{align}
\label{eql.hf_q_est_0} E_\rho ( r q_{ a b } ) &= - 2 q_a{}^\mu q_{ b \mu } - \frac{r^2}{2} R_{ \rho a \rho b } \\
\notag &= 2 \frac{ r^2 }{ \rho^2 } \, q_{ a \theta } q_{ b \theta } - 2 \sum_A q_{ a A } q_{ b A } - \frac{r^2}{2} R_{ \rho a \rho b } \text{,}
\end{align}
where in the last step, we expanded using \eqref{eq.hf_E}.

We now make the bootstrap assumptions
\begin{equation}
\label{eql.hf_q_est_1} | q_{ A B } | \leq \frac{ \delta r^2 }{ r_0^2 } \text{,} \qquad | q_{ \theta A } | \leq \frac{ \rho^2 }{ r^2 } \, \frac{ \delta r^2 }{ r_0^2 } \text{,} \qquad | q_{ \theta \theta } | \leq \frac{ \rho^4 }{ r^4 } \, \frac{ \delta r^2 }{ r_0^2 } \text{,} \qquad \delta > 0 \text{,}
\end{equation}
and we integrate \eqref{eql.hf_q_est_0} from $p$ along each geodesic $\gamma \in \Gamma$.
Observing from \eqref{eq.hf_f_asymp_2} that $r q_{ a b }$ vanishes at $p$ along each $\gamma$, while applying \eqref{eq.hf_E_special}, \eqref{eq.hf_curv}, and \eqref{eql.hf_q_est_1}, we see that
\begin{align}
\label{eql.hf_q_est_2} | r q_{ A B } | \leq \frac{ 2 n }{5} \, \frac{ \delta^2 r^5 }{ r_0^4 } + \frac{ \mc{C}_0 }{ 6 n } \, \frac{ r^3 }{ r_0^2 } \text{,} &\qquad | q_{ A B } | \leq \frac{ 2 n }{5} \, \frac{ \delta^2 r^4 }{ r_0^4 } + \frac{ \mc{C}_0 }{ 6 n } \, \frac{ r^2 }{ r_0^2 } \text{,} \\
\notag | r q_{ \theta A } | \leq \frac{ \rho^2 }{ r^2 } \left( \frac{ 2 n }{5} \, \frac{ \delta^2 r^5 }{ r_0^4 } + \frac{ \mc{C}_0 }{ 6 n } \, \frac{ r^3 }{ r_0^2 } \right) \text{,} &\qquad | q_{ \theta A } | \leq \frac{ \rho^2 }{ r^2 } \left( \frac{ 2 n }{5} \, \frac{ \delta^2 r^4 }{ r_0^4 } + \frac{ \mc{C}_0 }{ 6 n } \, \frac{ r^2 }{ r_0^2 } \right) \text{,} \\
\notag | r q_{ \theta \theta } | \leq \frac{ \rho^4 }{ r^4 } \left( \frac{ 2 n }{5} \, \frac{ \delta^2 r^5 }{ r_0^4 } + \frac{ \mc{C}_0 }{ 6 n } \, \frac{ r^3 }{ r_0^2 } \right) \text{,} &\qquad | q_{ \theta \theta } | \leq \frac{ \rho^4 }{ r^4 } \left( \frac{ 2 n }{5} \, \frac{ \delta^2 r^4 }{ r_0^4 } + \frac{ \mc{C}_0 }{ 6 n } \, \frac{ r^2 }{ r_0^2 } \right) \text{.}
\end{align}

From the quadratic equation, we see that if $\mc{C}_0 < \mc{C}_\ast := \frac{15}{4}$, then
\[
\frac{ 2 n }{5} \, \frac{ \delta^2 r^4 }{ r_0^4 } + \frac{ \mc{C}_0 }{ 6 n } \, \frac{ r^2 }{ r_0^2 } < \frac{ \delta r^2 }{ r_0^2 } \text{,}
\]
as long as we take
\[
\delta := \frac{ \mc{C}_0 }{ 3 n } > \sup_{ 0 < r < r_0 } \frac{5}{4n} \frac{ r_0^2 }{ r^2 } \left( 1 - \sqrt{ 1 - \frac{ 4 \mc{C}_0 }{ 15 } \frac{ r^2 }{ r_0^2 } } \right) \text{.}
\]
In particular, for these values of $\mc{C}_0$ and $\delta$, we have that \eqref{eql.hf_q_est_2} strictly improves the bootstrap assumptions \eqref{eql.hf_q_est_1}.
A standard continuity argument now yields \eqref{eq.hf_q_est}.
\end{proof}

\begin{proposition} \label{thm.hf_qd_est}
There is a universal constant $\mc{C}_\ast > 0$ such that if Assumption \ref{ass.hf_curv} holds and $0 < \mc{C}_0 < \mc{C}_\ast$, then the following inequalities hold on $\mc{D}_{ r_0 }$ for some universal $\mc{C} > 0$:
\begin{align}
\label{eq.hf_qd_est} | \nabla^3 f ( E_A, E_B, E_C ) | &\leq \frac{ \mc{C} }{n} \left( \frac{ \mc{C}_0 r }{ r_0^2 } + \frac{ \mc{C}_1 r^2 }{ r_0^3 } \right) \text{,} \\
\notag | \nabla^3 f ( E_\theta, E_A, E_B ) | &\leq \frac{ \mc{C} }{n} \left( \frac{ \mc{C}_0 r }{ r_0^2 } + \frac{ \mc{C}_1 r^2 }{ r_0^3 } \right) \text{,} \\
\notag | \nabla^3 f ( E_A, E_\theta, E_B ) | &\leq \frac{ \mc{C} }{n} \left( \frac{ \mc{C}_0 r }{ r_0^2 } + \frac{ \mc{C}_1 r^2 }{ r_0^3 } \right) \text{,} \\
\notag | \nabla^3 f ( E_\theta, E_\theta, E_A ) | &\leq \frac{ \mc{C} }{n} \frac{ \rho^2 }{ r^2 } \left( \frac{ \mc{C}_0 r }{ r_0^2 } + \frac{ \mc{C}_1 r^2 }{ r_0^3 } \right) \text{,} \\
\notag | \nabla^3 f ( E_A, E_\theta, E_\theta ) | &\leq \frac{ \mc{C} }{n} \frac{ \rho^2 }{ r^2 } \left( \frac{ \mc{C}_0 r }{ r_0^2 } + \frac{ \mc{C}_1 r^2 }{ r_0^3 } \right) \text{,} \\
\notag | \nabla^3 f ( E_\theta, E_\theta, E_\theta ) | &\leq \frac{ \mc{C} }{n} \frac{ \rho^4 }{ r^4 } \left( \frac{ \mc{C}_0 r }{ r_0^2 } + \frac{ \mc{C}_1 r^2 }{ r_0^3 } \right) \text{.}
\end{align}
\end{proposition}

\begin{proof}
First, using \eqref{eq.hf_q}, we can write the last equation in \eqref{eq.hf_gauss_2} as
\begin{align*}
\nabla^\mu f \nabla_{ \mu \alpha \beta \lambda } f &= - \nabla_{ \alpha \beta \lambda } f - q_\lambda{}^\mu \nabla_{ \alpha \beta \mu } f - q_\alpha{}^\mu \nabla_{ \mu \beta \lambda } f - q_\beta{}^\mu \nabla_{ \alpha \mu \lambda } f \\
\notag &\qquad - \nabla_\alpha R_{ \mu \beta \nu \lambda } \nabla^\mu f \nabla^\nu f - \frac{1}{2} R_{ \mu \lambda \alpha \beta } \nabla^\mu f - \frac{1}{2} R_{ \mu \beta \alpha \lambda } \nabla^\mu f \\
\notag &\qquad - \frac{1}{2} R_{ \mu \alpha \lambda \beta } \nabla^\mu f - \frac{1}{2} R_{ \mu \alpha \beta \lambda } \nabla^\mu f - R_{ \mu \lambda \nu \beta } q_\alpha{}^\nu \nabla^\mu f \\
\notag &\qquad - R_{ \mu \beta \nu \lambda } q_\alpha{}^\nu \nabla^\mu f - R_{ \mu \alpha \nu \beta } q_\lambda {}^\nu \nabla^\mu f - R_{ \mu \alpha \nu \lambda } q_\beta{}^\nu \nabla^\mu f \\
\notag &= - \nabla_{ \alpha \beta \lambda } f - q_\lambda{}^\mu \nabla_{ \alpha \beta \mu } f - q_\alpha{}^\mu \nabla_{ \mu \beta \lambda } f - q_\beta{}^\mu \nabla_{ \alpha \mu \lambda } f \\
\notag &\qquad + \frac{1}{2} ( R_{ \mu \beta \lambda \alpha } + R_{ \mu \lambda \beta \alpha } ) \nabla^\mu f - ( R_{ \mu \lambda \nu \beta } + R_{ \mu \beta \nu \lambda } ) q_\alpha{}^\nu \nabla^\mu f \\
\notag &\qquad - R_{ \mu \alpha \nu \beta } q_\lambda {}^\nu \nabla^\mu f - R_{ \mu \alpha \nu \lambda } q_\beta{}^\nu \nabla^\mu f - \nabla_\alpha R_{ \mu \beta \nu \lambda } \nabla^\mu f \nabla^\nu f \text{.}
\end{align*}
Recalling Proposition \ref{thm.hf_gauss} and \eqref{eq.hf_E_rho}, while contracting with $E_\theta$ and $E_A$, yields
\begin{align*}
E_\rho ( r^2 \nabla_{ a b c } f ) &= - 2 r ( q_c{}^d \nabla_{ a b d } f + q_a{}^d \nabla_{ d b c } f + q_b{}^d \nabla_{ a d c } f ) + \frac{r^2}{2} ( R_{ \rho b c a } + R_{ \rho c b a } ) \\
\notag &\qquad - r^2 ( R_{ \rho c d b } + R_{ \rho b d c } ) q_a{}^d - r^2 R_{ \rho a d b } q_c{}^d - r^2 R_{ \rho a d c } q_b{}^d - \frac{r^3}{2} \nabla_a R_{ \rho b \rho c } \text{.}
\end{align*}

Splitting the above into cases and recalling Propositions \ref{thm.hf_E} and \ref{thm.hf_q_est}, we then have
\begin{align}
\label{eql.hf_qd_est_1} | E_\rho ( r^2 \nabla_{ A B C } f ) | &\leq \mc{C} \mc{C}_\ast \frac{ r }{ r_0^2 } \max_d ( | r^2 \nabla_{ A B d } f | + | r^2 \nabla_{ d B C } f | + | r^2 \nabla_{ A d C } f | ) \\
\notag &\qquad + \frac{ \mc{C} \mc{C}_0 }{n} \frac{ r^2 }{ r_0^2 } + \frac{ \mc{C} \mc{C}_1 }{n} \frac{ r^3 }{ r_0^3 } \text{,} \\
\notag | E_\rho ( r^2 \nabla_{ \theta B C } f ) | &\leq \mc{C} \mc{C}_\ast \frac{ r }{ r_0^2 } \max_d \left( | r^2 \nabla_{ \theta B d } f | + \frac{ \rho^2 }{ r^2 } | r^2 \nabla_{ d B C } f | + | r^2 \nabla_{ \theta d C } f | \right) \\
\notag &\qquad + \frac{ \mc{C} \mc{C}_0 }{n} \frac{ r^2 }{ r_0^2 } + \frac{ \mc{C} \mc{C}_1 }{n} \frac{ r^3 }{ r_0^3 } \text{,} \\
\notag | E_\rho ( r^2 \nabla_{ A \theta C } f ) | &\leq \mc{C} \mc{C}_\ast \frac{ r }{ r_0^2 } \max_d \left( | r^2 \nabla_{ A \theta d } f | + | r^2 \nabla_{ d \theta C } f | + \frac{ \rho^2 }{ r^2 } | r^2 \nabla_{ A d C } f | \right) \\
\notag &\qquad + \frac{ \mc{C} \mc{C}_0 }{n} \frac{ r^2 }{ r_0^2 } + \frac{ \mc{C} \mc{C}_1 }{n} \frac{ \rho^2 }{ r^2 } \frac{ r^3 }{ r_0^3 } \text{,} \\
\notag | E_\rho ( r^2 \nabla_{ \theta \theta C } f ) | &\leq \mc{C} \mc{C}_\ast \frac{ r }{ r_0^2 } \max_d \left( | r^2 \nabla_{ \theta \theta d } f | + \frac{ \rho^2 }{ r^2 } | r^2 \nabla_{ d \theta C } f | + \frac{ \rho^2 }{ r^2 } | r^2 \nabla_{ \theta d C } f | \right) \\
\notag &\qquad + \frac{ \mc{C} \mc{C}_0 }{n} \frac{ \rho^2 }{ r^2 } \frac{ r^2 }{ r_0^2 } + \frac{ \mc{C} \mc{C}_1 }{n} \frac{ \rho^2 }{ r^2 } \frac{ r^3 }{ r_0^3 } \text{,} \\
\notag | E_\rho ( r^2 \nabla_{ A \theta \theta } f ) | &\leq \mc{C} \mc{C}_\ast \frac{ r }{ r_0^2 } \max_d \left( \frac{ \rho^2 }{ r^2 } | r^2 \nabla_{ A \theta d } f | + | r^2 \nabla_{ d \theta \theta } f | + \frac{ \rho^2 }{ r^2 } | r^2 \nabla_{ \theta d C } f | \right) \\
\notag &\qquad + \frac{ \mc{C} \mc{C}_0 }{n} \frac{ \rho^2 }{ r^2 } \frac{ r^2 }{ r_0^2 } + \frac{ \mc{C} \mc{C}_1 }{n} \frac{ \rho^4 }{ r^4 } \frac{ r^3 }{ r_0^3 } \text{,} \\
\notag | E_\rho ( r^2 \nabla_{ \theta \theta \theta } f ) | &\leq \mc{C} \mc{C}_\ast \frac{ r }{ r_0^2 } \max_d \left( \frac{ \rho^2 }{ r^2 } | r^2 \nabla_{ \theta \theta d } f | + \frac{ \rho^2 }{ r^2 } | r^2 \nabla_{ d \theta \theta } f | + \frac{ \rho^2 }{ r^2 } | r^2 \nabla_{ \theta d \theta } f | \right) \\
\notag &\qquad + \frac{ \mc{C} \mc{C}_0 }{n} \frac{ \rho^4 }{ r^4 } \frac{ r^2 }{ r_0^2 } + \frac{ \mc{C} \mc{C}_1 }{n} \frac{ \rho^4 }{ r^4 } \frac{ r^3 }{ r_0^3 } \text{.}
\end{align}

Consider now the quantity
\begin{align}
\label{eql.hf_qd_est_2} \mc{I} &:= \max_{ A, B, C } \left( \frac{ \rho^4 }{ r^4 } | r^2 \nabla_{ A B C } f |, \frac{ \rho^4 }{ r^4 } | r^2 \nabla_{ \theta A B } f |, \frac{ \rho^4 }{ r^4 } | r^2 \nabla_{ A \theta B } f |, \right. \\
\notag &\qquad\qquad\quad \left. \frac{ \rho^2 }{ r^2 } | r^2 \nabla_{ \theta \theta A } f |, \frac{ \rho^2 }{ r^2 } | r^2 \nabla_{ A \theta \theta } f |, | r^2 \nabla_{ \theta \theta \theta } f | \right) \text{.}
\end{align}
Integrating each inequality in \eqref{eql.hf_qd_est_1} from $p$ along each $\gamma \in \Gamma$ (and recalling the limits \eqref{eq.hf_f_asymp_3}), and then applying the Gronwall inequality yields the bound
\[
\mc{I} \leq \frac{ \mc{C} }{n} \frac{ \rho^4 }{ r^4 } \left( \frac{ \mc{C}_0 r^3 }{ r_0^2 } + \frac{ \mc{C}_1 r^4 }{ r_0^3 } \right) \text{.}
\]
The desired result \eqref{eq.hf_qd_est} follows immediately from \eqref{eql.hf_qd_est_2} and the above.
\end{proof}

\subsection{Estimates for $t$}

We now apply the estimates for $f$ to control derivatives of $t$ and $t^2$. These estimates will be used to derive the pseudoconvexity condition for an appropriate perturbation of $f$ (see Section \ref{sec.pc} below).

\begin{proposition} \label{thm.hf_t_est}
There exists a universal $\mc{C}_\ast > 0$ such that if Assumption \ref{ass.hf_curv} holds and $0 < \mc{C}_0 < \mc{C}_\ast$, then the following also hold on $\mc{D}_{ r_0 }$ for some universal $\mc{C} > 0$:
\begin{equation}
\label{eq.hf_t_est} | \nabla t ( E_A ) | \leq \frac{ \mc{C} \mc{C}_0 }{n} \, \frac{ r^2 }{ r_0^2 } \text{,} \qquad | \nabla t ( E_\theta ) - 1 | \leq \frac{ \mc{C} \mc{C}_0 }{n} \frac{ \rho^2 }{ r^2 } \frac{ r^2 }{ r_0^2 } \text{.}
\end{equation}
\end{proposition}

\begin{proof}
Let $\mc{C}_\ast$ be as in Proposition \ref{thm.hf_q_est}, and consider an arbitrary geodesic $\gamma \in \Gamma$.
Note that by the basic properties of normal coordinates, we have the following limits at $p$ along $\gamma$:
\begin{equation}
\label{eql.hf_t_est_1} \lim_p \nabla_{ E_A } t = - g ( e_0 , e_{ A, \omega } ) = 0 \text{,} \qquad \lim_p \nabla_{ E_\theta } t = - g ( e_0, e_{ \theta, \omega } ) = 1 \text{.}
\end{equation}

Next, using \eqref{eq.hf_gauss}, \eqref{eq.hf_E_rho}, and \eqref{eq.hf_q}, we can rewrite the second equation in \eqref{eq.hf_transport} as
\[
\nabla^\mu f \nabla_{ \mu \alpha } t = - q_\alpha{}^\mu \nabla_\mu t \text{,} \qquad E_\rho ( \nabla_\alpha t ) = - \frac{2}{r} q_\alpha{}^\mu \nabla_\mu t \text{.}
\]
Integrating the second equation above and recalling \eqref{eq.hf_q_est} and \eqref{eql.hf_t_est_1} then yields
\begin{align*}
| \nabla_A t | &\leq \frac{ r^2 }{ \rho^2 } \int_0^r \frac{2}{s} | q_{ A \theta } | | \nabla_\theta t | \, ds + \int_0^r \frac{2}{s} \sum_B | q_{ A B } | | \nabla_B t | \, ds \\
&\leq \frac{ \mc{C}_0 }{ 3 n } \int_0^r \frac{ 2 s }{ r_0^2 } \left( n \sum_B | \nabla_B t | + | \nabla_\theta t - 1 | + 1 \right) ds \text{,} \\
| \nabla_\theta t - 1 | &\leq \frac{ r^2 }{ \rho^2 } \int_0^r \frac{2}{s} | q_{ \theta \theta } | | \nabla_\theta t | ds + \int_0^r \frac{2}{s} \sum_B | q_{ \theta B } | | \nabla_B t | \, ds \\
&\leq \frac{ \rho^2 }{ r^2 } \, \frac{ \mc{C}_0 }{ 3 n } \int_0^r \frac{ 2 s }{ r_0^2 } \left( n \sum_B | \nabla_B t | + | \nabla_\theta t - 1 | + 1 \right) ds \text{,}
\end{align*}
where the integrals are understood to be along the geodesic $\gamma$ and parametrized by $r$.
(Recall also that $\rho^2 r^{-2}$ is constant along $\gamma$.)
Combining the above inequalities then yields
\begin{equation}
\label{eql.hf_t_est_2} \max_A \left( \frac{ \rho^2 }{ r^2 } | \nabla_A t |, | \nabla_\theta t - 1 | \right) \leq \frac{ \mc{C}_0 }{ 3 n } \int_0^r \frac{ 2 s }{ r_0^2 } \left[ n \max_A \left( \frac{ \rho^2 }{ r^2 } | \nabla_A t |, | \nabla_\theta t - 1 | \right) + \frac{ \rho^2 }{ r^2 } \right] ds \text{.}
\end{equation}

Applying the Gronwall inequality to the above and recalling that $\mc{C}_0 < \mc{C}_\ast$ yields
\[
\max_A \left( \frac{ \rho^2 }{ r^2 } | \nabla_A t |, | \nabla_\theta t - 1 | \right) \leq \frac{ \rho^2 }{ r^2 } \, \frac{ \mc{C} \mc{C}_0 }{ n } \int_0^r \frac{ 2 s }{ r_0^2 } \, ds = \frac{ \rho^2 }{ r^2 } \, \frac{ \mc{C} \mc{C}_0 }{ n } \, \frac{ r^2 }{ r_0^2 } \text{.}
\]
The above immediately implies the desired bound \eqref{eq.hf_t_est}.
\end{proof}

\begin{corollary} \label{thm.hf_t2_est}
There exists a universal constant $\mc{C}_\ast > 0$ such that if Assumption \ref{ass.hf_curv} holds and $0 < \mc{C}_0 < \mc{C}_\ast$, then the following also hold on $\mc{D}_{ r_0 }$ for some universal $\mc{C} > 0$:
\begin{equation}
\label{eq.hf_t2_est} | \nabla t^2 ( E_A ) | \leq \frac{ \mc{C} \mc{C}_0 }{n} \, \frac{ r^2 t }{ r_0^2 } \text{,} \qquad | \nabla t^2 ( E_\theta ) - 2 t | \leq \frac{ \mc{C} \mc{C}_0 }{n} \frac{ \rho^2 }{ r^2 } \frac{ r^2 t }{ r_0^2 } \text{.}
\end{equation}
\end{corollary}

\begin{proof}
This follows immediately from Proposition \ref{thm.hf_t_est}.
\end{proof}

\begin{proposition} \label{thm.hf_td_est}
There exists a universal $\mc{C}_\ast > 0$ such that if Assumption \ref{ass.hf_curv} holds and $0 < \mc{C}_0 < \mc{C}_\ast$, then the following also hold on $\mc{D}_{ r_0 }$ for some universal $\mc{C} > 0$:
\begin{align}
\label{eq.hf_td_est} | \nabla^2 t^2 ( E_A, E_B ) | &\leq \frac{ \mc{C} \mc{C}_0 }{n} \, \frac{ r^2 }{ \rho^2 } \frac{ r^2 }{ r_0^2 } + \frac{ \mc{C} \mc{C}_1 }{n} \, \frac{ r^2 }{ \rho^2 } \frac{ r^2 t }{ r_0^3 } \text{,} \\
\notag | \nabla^2 t^2 ( E_\theta, E_A ) | &\leq \frac{ \mc{C} \mc{C}_0 }{n} \, \frac{ r^2 }{ r_0^2 } + \frac{ \mc{C} \mc{C}_1 }{n} \, \frac{ r^2 t }{ r_0^3 } \text{,} \\
\notag | \nabla^2 t^2 ( E_\theta, E_\theta ) - 2 | &\leq \frac{ \mc{C} \mc{C}_0 }{n} \, \frac{ \rho^2 }{ r^2 } \frac{ r^2 }{ r_0^2 } + \frac{ \mc{C} \mc{C}_1 }{n} \, \frac{ \rho^2 }{ r^2 } \frac{ r^2 t }{ r_0^3 } \text{.}
\end{align}
\end{proposition}

\begin{proof}
First, using \eqref{eq.hf_q}, we can rewrite the last equation in \eqref{eq.hf_transport} as
\[
\nabla^\mu f \nabla_{ \mu \alpha \beta } t^2 = q_\alpha{}^\mu \nabla_{ \beta \mu } t^2 - q_\beta{}^\mu \nabla_{ \alpha \mu } t^2 - \nabla_{ \alpha \beta \mu } f \nabla^\mu t^2 - R_{ \mu \alpha \nu \beta } \nabla^\mu f \nabla^\nu t^2 \text{.}
\]
Moreover, contracting with $E_\theta$ and $E_A$, and applying \eqref{eq.hf_gauss} and \eqref{eq.hf_E_rho}, the above becomes
\begin{equation}
\label{eql.hf_td_est_1} E_\rho ( \nabla_{ a b } t^2 ) = \frac{2}{r} q_a{}^c \nabla_{ b c } t^2 - \frac{2}{r} q_b{}^c \nabla_{ a c } t^2 - \frac{2}{r} \nabla_{ a b \mu } f \nabla^\mu t^2 - R_{ \rho a \nu b } \nabla^\nu t^2 \text{.}
\end{equation}
A similar computation using the second part of \eqref{eq.hf_transport} yields
\begin{equation}
\label{eql.hf_td_est_2} E_\rho ( \nabla_a t \nabla_b t ) = - \frac{2}{r} q_a{}^c \nabla_b t \nabla_c t - \frac{2}{r} q_b{}^c \nabla_a t \nabla_c t \text{.}
\end{equation}

Let us now define the tensor field
\begin{equation}
\label{eql.hf_td_est_10} T := \nabla^2 t^2 - 2 ( \nabla t \otimes \nabla t ) \text{.}
\end{equation}
Observe that by the properties of normal coordinates, we have that
\begin{equation}
\label{eql.hf_td_est_11} T |_p = 0 \text{.}
\end{equation}
Moreover, combining \eqref{eql.hf_td_est_1} and \eqref{eql.hf_td_est_2}, we obtain
\begin{align}
\label{eql.hf_td_est_12} E_\rho ( T_{ a b } ) &= \frac{2}{r} q_a{}^c T_{ b c } - \frac{2}{r} q_b{}^c T_{ a c } + \mc{B}_{ a b } \text{,} \\
\notag \mc{B}_{ a b } &:= - \frac{2}{r} \nabla_{ a b \mu } f \nabla^\mu t^2 - R_{ \rho a \nu b } \nabla^\nu t^2 \\
\notag &= - \left( \frac{2}{r} \nabla_{ \mu a b } f + R_{ \rho b \mu a } + R_{ \rho a \mu b } \right) \nabla^\mu t^2 \text{.}
\end{align}

Expanding \eqref{eql.hf_td_est_12} and recalling \eqref{eq.hf_rho}, \eqref{eq.hf_gauss}, \eqref{eq.hf_transport}, and \eqref{eq.hf_E_rho}, we obtain
\begin{align*}
\mc{B}_{ a b } &= - f^{-1} \left( \frac{2}{r} \nabla_{ \mu a b } f + R_{ \rho b \mu a } + R_{ \rho a \mu b } \right) \nabla^\mu f \nabla^\nu f \nabla_\nu t^2 - \left( \frac{2}{r} \nabla_{ c a b } f + R_{ \rho b c a } + R_{ \rho a c b } \right) \nabla^c t^2 \\
&= - \frac{ 4 t^2 }{ \rho^2 } \left( \frac{2}{r} \nabla^\mu f \nabla_{ \mu a b } f + r R_{ \rho a \rho b } \right) - \left( \frac{2}{r} \nabla_{ c a b } f + R_{ \rho b c a } + R_{ \rho a c b } \right) \nabla^c t^2 \\
&= \frac{ 4 t^2 }{ \rho^2 } \left( \frac{1}{r} q_{ a b } + \frac{2}{r} q_a{}^c q_{ b c } - \frac{r}{2} R_{ \rho a \rho b } \right) + \frac{ 2 r^2 }{ \rho^2 } \left( \frac{ 2 t }{r} \nabla_{ \theta a b } f + t R_{ \rho b \theta a } + t R_{ \rho a \theta b } \right) \\
&\qquad + \frac{ r^2 }{ \rho^2 } \left( \frac{2}{r} \nabla_{ \theta a b } f + R_{ \rho b \theta a } + R_{ \rho a \theta b } \right) ( \nabla_\theta t^2 - 2 t ) - \left( \frac{2}{r} \nabla_{ C a b } f + R_{ \rho b C a } + R_{ \rho a C b } \right) \nabla^C t^2 \text{.}
\end{align*}
Applying Propositions \ref{thm.hf_E}, \ref{thm.hf_q_est}, \ref{thm.hf_qd_est}, and \ref{thm.hf_t_est} to the above, in conjunction with Assumption \ref{ass.hf_curv}, we obtain the following bounds for the components of $\mc{B}$:
\begin{align}
\label{eql.hf_td_est_20} | \mc{B}_{ A B } | &\leq \frac{ \mc{C} \mc{C}_0 }{ n } \frac{ r^2 }{ \rho^2 } \frac{ t }{ r_0^2 } + \frac{ \mc{C} \mc{C}_1 }{ n } \frac{ r^2 }{ \rho^2 } \frac{ r t }{ r_0^3 } \\
\notag | \mc{B}_{ \theta A } | &\leq \frac{ \mc{C} \mc{C}_0 }{ n } \frac{ t }{ r_0^2 } + \frac{ \mc{C} \mc{C}_1 }{ n } \frac{ r t }{ r_0^3 } \\
\notag | \mc{B}_{ \theta \theta } | &\leq \frac{ \mc{C} \mc{C}_0 }{ n } \frac{ \rho^2 }{ r^2 } \frac{ t }{ r_0^2 } + \frac{ \mc{C} \mc{C}_1 }{ n } \frac{ \rho^2 }{ r^2 } \frac{ r t }{ r_0^3 }
\end{align}

Combining \eqref{eql.hf_td_est_12} with the estimates \eqref{eq.hf_q_est} and \eqref{eql.hf_td_est_20}, we obtain
\begin{align*}
| E_\rho ( T_{ A B } ) | &\leq \mc{C} \frac{ r }{ r_0^2 } \max_c ( | T_{ A c } | + | T_{ B c } | ) + \frac{ \mc{C} \mc{C}_0 }{ n } \frac{ r^2 }{ \rho^2 } \frac{ t }{ r_0^2 } + \frac{ \mc{C} \mc{C}_1 }{ n } \frac{ r^2 }{ \rho^2 } \frac{ r t }{ r_0^3 } \text{,} \\
| E_\rho ( T_{ \theta A } ) | &\leq \mc{C} \frac{ r }{ r_0^2 } \max_c \left( | T_{ \theta c } | + \frac{ \rho^2 }{ r^2 } | T_{ A c } | \right) + \frac{ \mc{C} \mc{C}_0 }{ n } \frac{ t }{ r_0^2 } + \frac{ \mc{C} \mc{C}_1 }{ n } \frac{ r t }{ r_0^3 } \text{,} \\
| E_\rho ( T_{ \theta \theta } ) | &\leq \mc{C} \frac{ r }{ r_0^2 } \max_c \frac{ \rho^2 }{ r^2 } | T_{ \theta c } | + \frac{ \mc{C} \mc{C}_0 }{ n } \frac{ \rho^2 }{ r^2 } \frac{ t }{ r_0^2 } + \frac{ \mc{C} \mc{C}_1 }{ n } \frac{ \rho^2 }{ r^2 } \frac{ r t }{ r_0^3 } \text{.}
\end{align*}
We can now integrate each of the above inequalities from $p$ along each $\gamma \in \Gamma$.
Recalling the vertex limit \eqref{eql.hf_td_est_11} and applying the Gronwall inequality yields
\[
\max_{ A, B } \left( \frac{ \rho^4 }{ r^4 } | T_{ A B } |, \frac{ \rho^2 }{ r^2 } | T_{ \theta A } |, | T_{ \theta \theta } | \right) \leq \frac{ \mc{C} \mc{C}_0 }{ n } \frac{ \rho^2 }{ r^2 } \frac{ r t }{ r_0^2 } + \frac{ \mc{C} \mc{C}_1 }{ n } \frac{ \rho^2 }{ r^2 } \frac{ r^2 t }{ r_0^3 } \text{.}
\]
The desired results now follow by combining the above with \eqref{eq.hf_t_est} and \eqref{eql.hf_td_est_10}.
\end{proof}

\section{Pseudoconvexity} \label{sec.pc}

Throughout this section, we once again assume the setting and definitions of Section \ref{sec.hf}.
In addition, we fix a constant $r_0 > 0$, and we consider the function
\begin{equation}
\label{eq.pc_eta} \eta := 1 - \varepsilon t^2 \in C^\infty ( \mc{M} ) \text{,} \qquad \varepsilon := \varepsilon_0 r_0^{-2} \text{,} \qquad 0 < \varepsilon_0 \ll 1 \text{,}
\end{equation}
where the exact value of $\varepsilon_0$ will be determined later in the notes.
Using this particular $\eta$, we define the corresponding \emph{shifted hyperquadric} function:
\begin{equation}
\label{eq.pc_fb} \bar{f} := f \eta^{-1} \in C^\infty ( \bar{\mc{D}}_{ r_0 } ) \text{.}
\end{equation}
In this section, we explore when $\bar{f}$ is pseudoconvex with respect to $\Box_g$ on $\mc{D}_{ r_0 }$.

\subsection{Shifted hyperquadrics}

We begin by establishing various algebraic observations involving level sets of $f$ and $\bar{f}$.
In general, we say a vector field $X$ (resp.\ $\bar{X}$) on $\mc{D}_{ r_0 }$ is $f$-tangent (resp.\ $\bar{f}$-tangent) iff it is everywhere tangent to level sets of $f$ (resp.\ $\bar{f}$).
Our first step is to develop notations for working with such $f$-tangent and $\bar{f}$-tangent vector fields.

Given $f$-tangent and $\bar{f}$-tangent vector fields $X$ and $\bar{X}$, respectively, we define
\begin{align}
\label{eq.pc_P} P X &:= X + ( \eta - \nabla^\alpha f \nabla_\alpha \eta )^{-1} X \eta \cdot \nabla^\sharp f = X + X \eta \cdot \nabla^\sharp f \text{,} \\
\notag \bar{P} \bar{X} &:= \bar{X} - \eta^{-1} \bar{X} \eta \cdot \nabla^\sharp f \text{.}
\end{align}
(Note the extra equality in \eqref{eq.pc_P} holds, since \eqref{eq.hf_transport} and \eqref{eq.pc_eta} imply $\eta - \nabla^\alpha f \nabla_\alpha \eta = 1$.)

The following proposition shows that $P$ and $\bar{P}$ are inverses of each other.
\begin{proposition} \label{thm.pc_P}
Given an $f$-tangent $X$ and an $\bar{f}$-tangent $\bar{X}$:
\begin{itemize}
\item $P X$ is $\bar{f}$-tangent.

\item $\bar{P} \bar{X}$ is $f$-tangent.

\item The following identities hold:
\begin{equation}
\label{eq.pc_P_id} P \bar{P} \bar{X} = \bar{X} \text{,} \qquad \bar{P} P X = X \text{.}
\end{equation}
\end{itemize}
\end{proposition}

\begin{proof}
A direct computation using \eqref{eq.hf_gauss} shows that
\[
( \bar{P} \bar{X} ) f = \bar{X} ( \eta \bar{f} ) - \eta^{-1} \bar{X} \eta \cdot \nabla^\alpha f \nabla_\alpha f = 0 \text{,}
\]
as well as
\begin{align*}
( P X ) \bar{f} &= X ( \eta^{-1} f ) + ( \eta - \nabla^\alpha f \nabla_\alpha \eta )^{-1} X \eta \cdot \nabla^\alpha f \nabla_\alpha ( f \eta^{-1} ) \\
&= - f \eta^{-2} X \eta + ( \eta - \nabla^\alpha f \nabla_\alpha \eta )^{-1} X \eta \cdot f ( \eta^{-1} - \eta^{-2} \nabla^\alpha f \nabla_\alpha \eta ) \\
&= 0 \text{.}
\end{align*}
This proves the first two points above.
The identities \eqref{eq.pc_P_id} are now direct computations:
\begin{align*}
P \bar{P} \bar{X} &= ( \bar{P} \bar{X} ) + ( \eta - \nabla^\alpha f \nabla_\alpha \eta )^{-1} ( \bar{P} \bar{X} ) \eta \cdot \nabla^\sharp f \\
&= \bar{X} - \eta^{-1} \bar{X} \eta \cdot \nabla^\sharp f + ( \eta - \nabla^\alpha f \nabla_\alpha \eta )^{-1} ( \bar{X} \eta - \eta^{-1} \bar{X} \eta \cdot \nabla^\alpha f \nabla_\alpha \eta ) \cdot \nabla^\sharp f \\
&= \bar{X} - \eta^{-1} \bar{X} \eta \cdot \nabla^\sharp f + \eta^{-1} \bar{X} \eta \cdot \nabla^\sharp f \\
&= \bar{X} \text{,} \\
\bar{P} P X &= P X - \eta^{-1} ( P X ) \eta \cdot \nabla^\sharp f \\
&= X + ( \eta - \nabla^\alpha f \nabla_\alpha \eta )^{-1} X \eta \cdot \nabla^\sharp f - \eta^{-1} [ X \eta + ( \eta - \nabla^\alpha f \nabla_\alpha \eta )^{-1} X \eta \nabla^\alpha f \nabla_\alpha \eta ] \cdot \nabla^\sharp f \\
&= X + \eta^{-1} X \eta \cdot \nabla^\sharp f - \eta^{-1} X \eta \cdot \nabla^\sharp f \\
&= X \text{.} \qedhere
\end{align*}
\end{proof}
Our next proposition gives us a relation between $\nb^2 f$ and $\nb^2 \bar{f}$, and their deviation from the Minkowski value.
\begin{proposition} \label{thm.pc_fb_hessian}
The following holds for $\bar{f}$-tangent vector fields $\bar{X}$, $\bar{Y}$:
\begin{align}
\label{eq.pc_fb_hessian} \left( \nabla^2 \bar{f} - \frac{1}{2} \eta^{-1} g + \eta^{-1} \bar{f} \nabla^2 \eta \right) ( \bar{X}, \bar{Y} ) &= \eta^{-1} \left( \nabla^2 f - \frac{1}{2} g \right) ( \bar{P} \bar{X}, \bar{P} \bar{Y} ) \text{.}
\end{align}
\end{proposition}

\begin{proof}
We first use \eqref{eq.pc_fb} to compute
\begin{align*}
\left( \nabla^2 \bar{f} - \frac{1}{2} \eta^{-1} g \right) ( \bar{X}, \bar{Y} ) &= \bar{X}^\alpha \bar{Y}^\beta \nabla_\alpha ( \eta^{-1} \nabla_\beta f - \eta^{-2} f \nabla_\beta \eta ) - \frac{1}{2} \eta^{-1} g ( \bar{X}, \bar{Y} ) \\
&= \eta^{-1} \left( \nabla^2 f - \frac{1}{2} g \right) ( \bar{X}, \bar{Y} ) - \eta^{-2} ( \bar{X} \eta \bar{Y} f + \bar{X} f \bar{Y} \eta ) \\
&\qquad + 2 \eta^{-3} f \bar{X} \eta \bar{Y} \eta - \eta^{-2} f \nabla^2 \eta ( \bar{X}, \bar{Y} ) \text{.}
\end{align*}
Recalling Proposition \ref{thm.hf_gauss}, \eqref{eq.pc_eta}, \eqref{eq.pc_P}, and Proposition \ref{thm.pc_P}, we see that
\begin{align*}
\eta^{-1} \left( \nabla^2 f - \frac{1}{2} g \right) ( \bar{X}, \bar{Y} ) &= \eta^{-1} \left( \nabla^2 f - \frac{1}{2} g \right) ( \bar{P} \bar{X}, \bar{P} \bar{Y} ) \text{,} \\
- \eta^{-2} ( \bar{X} \eta \bar{Y} f + \bar{X} f \bar{Y} \eta ) &= - 2 \eta^{-3} f \bar{X} \eta \bar{Y} \eta \text{.}
\end{align*}
Combining the above shows that \eqref{eq.pc_fb_hessian} holds.
\end{proof}

Next, we define the metric $g_+$ on level sets of $f$ by
\begin{equation}
\label{eq.pc_gp} g_+ := g + \frac{ 2 r^2 }{ \rho^2 } \, E_\theta^\flat \otimes E_\theta^\flat \text{,}
\end{equation}
where the $1$-form $E_\theta^\flat$ denotes the $g$-metric dual to $E_\theta$.
Note the last term in \eqref{eq.pc_gp} simply reverses the sign of $g ( E_\theta, E_\theta )$, hence $g_+$ will be positive-definite.
In addition, we define
\begin{equation}
\label{eq.pc_gpb} \bar{g}_+ := g_+ ( \bar{P} \cdot, \bar{P} \cdot ) \text{,}
\end{equation}
which then gives positive-definite metrics on the level sets of $\bar{f}$.

Next, letting $E_\theta$, $ E_A$ be as before, we define the corresponding frames
\begin{equation}\label{eq.pc_Eb}
\begin{aligned}
\bar{E}_\theta &:= P E_\theta = E_\theta + \frac{r}{2} E_\theta \eta \cdot E_\rho \text{,} \\
\bar{E}_A &:= P E_A = E_A + \frac{r}{2} E_A \eta \cdot E_\rho \text{,}
\end{aligned}
\end{equation}
which are $\bar{f}$-tangent. Note the two equalities in \eqref{eq.pc_Eb} follow from \eqref{eq.hf_gauss}, \eqref{eq.hf_E_rho}, and \eqref{eq.pc_P}.

\begin{proposition} \label{thm.pc_gp_frame}
Let $X$ and $Y$ be $f$-tangent vector fields, let $\bar{X}$ and $\bar{Y}$ be $\bar{f}$-tangent vector fields, and suppose $X$, $Y$, $\bar{X}$, $\bar{Y}$ can be expanded locally in terms of frames as
\begin{align}
\label{eq.pc_gp_ass} X := X^\theta E_\theta + \sum_A X^A E_A \text{,} &\qquad Y := Y^\theta E_\theta + \sum_A Y^A E_A \text{,} \\
\notag \bar{X} := \bar{X}^\theta \bar{E}_\theta + \sum_A \bar{Y}^A \bar{E}_A \text{,} &\qquad \bar{Y} := \bar{Y}^\theta \bar{E}_\theta + \sum_A \bar{Y}^A \bar{E}_A \text{.}
\end{align}
Then, the following formulas hold:
\begin{equation}
\label{eq.pc_gp_frame} g_+ ( X, Y ) = \frac{ \rho^2 }{ r^2 } X^\theta Y^\theta + \sum_A X^A Y^A \text{,} \qquad \bar{g}_+ ( \bar{X}, \bar{Y} ) = \frac{ \rho^2 }{ r^2 } \bar{X}^\theta \bar{Y}^\theta + \sum_A \bar{X}^A \bar{Y}^A \text{.}
\end{equation}
\end{proposition}

\begin{proof}
Using \eqref{eq.hf_E} and the expansions \eqref{eq.pc_gp_ass} for $X$ and $Y$, we compute
\begin{align*}
g_+ ( X, Y ) &= g ( X, Y ) + \frac{ 2 r^2 }{ \rho^2 } \, g ( E_\theta, X ) g ( E_\theta, Y ) \\
&= - \frac{ \rho^2 }{ r^2 } X^\theta Y^\theta + \sum_A X^A Y^A + \frac{ 2 r^2 }{ \rho^2 } \cdot \frac{ \rho^2 }{ r^2 } X^\theta \cdot \frac{ \rho^2 }{ r^2 } Y^\theta \\
&= \frac{ \rho^2 }{ r^2 } X^\theta Y^\theta + \sum_A X^A Y^A \text{,}
\end{align*}
which is the first identity in \eqref{eq.pc_gp_frame}.
The second part of \eqref{eq.pc_gp_frame} then immediately follows from the first, in addition to the observation that $\bar{P}$ is $C^\infty$-linear:
\[
\bar{P} \bar{X} := \bar{X}^\theta E_\theta + \sum_A \bar{Y}^A E_A \text{,} \qquad \bar{P} \bar{Y} := \bar{Y}^\theta E_\theta + \sum_A \bar{Y}^A E_A \text{.} \qedhere
\]
\end{proof}

\subsection{Estimates for $\eta$}

In what follows, we use Assumption \ref{ass.hf_curv} along with our estimates for $t$ in Propositions \ref{thm.hf_t_est} and \ref{thm.hf_td_est} in order to obtain estimates for derivatives of $\eta$. We will derive the estimates corresponding to $f$-frames as well as $\bar{f}$-frames.

\begin{proposition} \label{thm.pc_eta_est}
There exists a universal constant $\mc{C}_\ast > 0$ so that if Assumption \ref{ass.hf_curv} holds and $0 < \mc{C}_0 < \mc{C}_\ast$, then the following hold on $\mc{D}_{ r_0 }$ for some universal $\mc{C} > 0$:
\begin{itemize}
\item $f$-frame estimates for $\nabla \eta$:
\begin{align}
\label{eq.pc_eta_est} | \nabla \eta ( E_A ) | &\leq \frac{ \mc{C} \mc{C}_0 }{ n } \frac{ r^2 }{ r_0^2 } \cdot \varepsilon t \text{,} \\
\notag | \nabla \eta ( E_\theta ) + 2 \varepsilon t | &\leq \frac{ \mc{C} \mc{C}_0 }{ n } \frac{ \rho^2 }{ r^2 } \frac{ r^2 }{ r_0^2 } \cdot \varepsilon t \text{,} \\
\notag \nabla \eta ( E_\rho ) + \frac{2 t}{r} \cdot \varepsilon t &= 0 \text{.}
\end{align}

\item $f$-frame estimates for $\nabla^2 \eta$:
\begin{align}
\label{eq.pc_eta_estd} | \nabla^2 \eta ( E_A, E_B ) | &\leq \frac{ \mc{C} ( \mc{C}_0 + \mc{C}_1 ) }{ n } \frac{ r^2 }{ \rho^2 } \frac{ r^2 }{ r_0^2 } \cdot \varepsilon \text{,} \\
\notag | \nabla^2 \eta ( E_\theta, E_A ) | &\leq \frac{ \mc{C} ( \mc{C}_0 + \mc{C}_1 ) }{ n } \frac{ r^2 }{ r_0^2 } \cdot \varepsilon \text{,} \\
\notag | \nabla^2 \eta ( E_\theta, E_\theta ) + 2 \varepsilon | &\leq \frac{ \mc{C} ( \mc{C}_0 + \mc{C}_1 ) }{ n } \frac{ \rho^2 }{ r^2 } \frac{ r^2 }{ r_0^2 } \cdot \varepsilon \text{,} \\
\notag \nabla^2 \eta ( E_\rho, E_\rho ) + \frac{ t^2 }{ r^2 } \cdot 2 \varepsilon &= 0 \text{,} \\
\notag | \nabla^2 \eta ( E_\rho, E_A ) | &\leq \frac{ \mc{C} \mc{C}_0 }{ n } \frac{t}{r} \frac{ r^2 }{ r_0^2 } \cdot \varepsilon \text{,} \\
\notag \left| \nabla^2 \eta ( E_\rho, E_\theta ) + \frac{t}{r} \cdot 2 \varepsilon \right| &\leq \frac{ \mc{C} \mc{C}_0 }{ n } \frac{t}{r} \frac{ \rho^2 }{ r^2 } \frac{ r^2 }{ r_0^2 } \cdot \varepsilon \text{.}
\end{align}
\end{itemize}
\end{proposition}

\begin{proof}
From Corollary \ref{thm.hf_t2_est}, Proposition \ref{thm.hf_td_est}, and \eqref{eq.pc_eta}, we see that there is a universal constant $\mc{C}_\ast > 0$ such that as long as $\mc{C}_0 < \mc{C}_\ast$, the following inequalities hold:
\begin{align}
\label{eql.pc_eta_est_1} | \nabla_A \eta | &\leq \frac{ \varepsilon \mc{C} \mc{C}_0 }{n} \, \frac{ r^2 t }{ r_0^2 } \text{,} \\
\notag | \nabla_\theta \eta + 2 \varepsilon t | &\leq \frac{ \varepsilon \mc{C} \mc{C}_0 }{n} \frac{ \rho^2 }{ r^2 } \frac{ r^2 t }{ r_0^2 } \text{,} \\
\notag | \nabla_{ A B } \eta | &\leq \frac{ \varepsilon \mc{C} \mc{C}_0 }{n} \frac{ r^2 }{ \rho^2 } \frac{ r^2 }{ r_0^2 } + \frac{ \varepsilon \mc{C} \mc{C}_1 }{n} \frac{ r^2 }{ \rho^2 } \frac{ r^2 t }{ r_0^3 } \text{,} \\
\notag | \nabla_{ \theta A } \eta | &\leq \frac{ \varepsilon \mc{C} \mc{C}_0 }{n} \frac{ r^2 }{ r_0^2 } + \frac{ \varepsilon \mc{C} \mc{C}_1 }{n} \frac{ r^2 t }{ r_0^3 } \text{,} \\
\notag | \nabla_{ \theta \theta } \eta + 2 \varepsilon | &\leq \frac{ \varepsilon \mc{C} \mc{C}_0 }{n} \frac{ \rho^2 }{ r^2 } \frac{ r^2 }{ r_0^2 } + \frac{ \varepsilon \mc{C} \mc{C}_1 }{n} \frac{ \rho^2 }{ r^2 } \frac{ r^2 t }{ r_0^3 } \text{.}
\end{align}
From the above, we obtain the first two parts of \eqref{eq.pc_eta_est} and the first three parts of \eqref{eq.pc_eta_estd}.

Now, the last part of \eqref{eq.pc_eta_est} follows from \eqref{eq.pc_eta} and \eqref{eq.hf_E_rho}, which implies
\begin{equation}
\label{eql.pc_eta_est_10} \nabla_\rho t = \frac{t}{r} \text{;}
\end{equation}
this proves \eqref{eq.pc_eta_est}.
Next, from \eqref{eq.hf_gauss}, the second part of \eqref{eq.hf_transport}, and \eqref{eq.hf_E_rho}, we have
\[
\nabla_{ \rho \alpha } t = - \frac{2}{r} q_\alpha{}^\mu \nabla_\mu t \text{,}
\]
From \eqref{eql.pc_eta_est_10} and the above, we then conclude
\begin{equation}
\label{eql.pc_eta_est_11} \nabla_{ \rho \alpha } t^2 = \frac{ 2 t }{r} \nabla_\alpha t - \frac{ 4 t }{r} q_\alpha{}^\mu \nabla_\mu t \text{,}
\end{equation}
with $q$ as in \eqref{eq.hf_q}.
By Proposition \ref{thm.hf_gauss}, the second part of \eqref{eq.hf_transport}, and \eqref{eql.pc_eta_est_11}, we obtain
\[
\nabla_{ \rho \rho } t^2 = \frac{ 2 t^2 }{r^2} \text{,}
\]
and the fourth part of \eqref{eq.pc_eta_estd} follows immediately due to \eqref{eq.pc_eta}.

Next, from Proposition \ref{thm.hf_q_est}, Proposition \ref{thm.hf_t_est}, and \eqref{eql.pc_eta_est_11}, we have
\begin{align*}
| \nabla_{ \rho A } t^2 | &\leq \frac{ 2 t }{r} | \nabla_A t | + \frac{ 4 t }{r} \frac{ r^2 }{ \rho^2 } | q_{ A \theta } | ( 1 + | \nabla_\theta t - 1 | ) + \frac{ 4 t }{r} \sum_B | q_{ A B } | | \nabla_B t | \\
&\leq \frac{ \mc{C} \mc{C}_0 }{n} \frac{t}{r} \frac{ r^2 }{ r_0^2 } + \frac{ ( \mc{C} \mc{C}_0 )^2 }{n} \frac{t}{r} \frac{ r^4 }{ r_0^4 } \text{.}
\end{align*}
Since $\mc{C}_0 < \mc{C}_\ast \lesssim 1$, recalling \eqref{eq.pc_eta} immediately results in the fifth part of \eqref{eq.pc_eta_estd}.
Finally, by Proposition \ref{thm.hf_q_est}, Proposition \ref{thm.hf_t_est}, and \eqref{eql.pc_eta_est_11} again,
\begin{align*}
\left| \nabla_{ \rho \theta } t^2 - \frac{ 2 t }{r} \right| &\leq \frac{ 2 t }{r} | \nabla_\theta t - 1 | + \frac{ 4 t }{r} \frac{ r^2 }{ \rho^2 } | q_{ \theta \theta } | ( 1 + | \nabla_\theta t - 1 | ) + \frac{ 4 t }{r} \sum_A | q_{ \theta A } | | \nabla_A t | \\
&\leq \frac{ \mc{C} \mc{C}_0 }{n} \frac{t}{r} \frac{ \rho^2 }{ r^2 } \frac{ r^2 }{ r_0^2 } + \frac{ ( \mc{C} \mc{C}_0 )^2 }{n} \frac{t}{r} \frac{ \rho^2 }{ r^2 } \frac{ r^4 }{ r_0^4 } \text{.}
\end{align*}
Recalling $\mc{C}_0 < \mc{C}_\ast \lesssim 1$ once again proves the last part of \eqref{eq.pc_eta_estd}.
\end{proof}

\begin{proposition} \label{thm.pc_etab_est}
There exists a universal constant $\mc{C}_\ast > 0$ so that if Assumption \ref{ass.hf_curv} holds and $0 < \mc{C}_0 < \mc{C}_\ast$, then the following hold on $\mc{D}_{ r_0 }$ for some universal $\mc{C} > 0$:
\begin{itemize}
\item $\bar{f}$-frame estimates for $\nabla \eta$:
\begin{align}
\label{eq.pc_etab_est} | \nabla \eta ( \bar{E}_A ) | &\leq \frac{ \mc{C} \mc{C}_0 }{ n } \frac{ r^2 }{ r_0^2 } \cdot \varepsilon t \text{,} \\
\notag | \nabla \eta ( \bar{E}_\theta ) + \eta \cdot 2 \varepsilon t | &\leq \frac{ \mc{C} \mc{C}_0 }{ n } \frac{ \rho^2 }{ r^2 } \frac{ r^2 }{ r_0^2 } \cdot \varepsilon t \text{.}
\end{align}

\item $\bar{f}$-frame estimates for $\nabla^2 \eta$:
\begin{align}
\label{eq.pc_etab_estd} | \nabla^2 \eta ( \bar{E}_A, \bar{E}_B ) | &\leq \frac{ \mc{C} ( \mc{C}_0 + \mc{C}_1 ) }{ n } \frac{ r^2 }{ \rho^2 } \frac{ r^2 }{ r_0^2 } \cdot \varepsilon \text{,} \\
\notag | \nabla^2 \eta ( \bar{E}_\theta, \bar{E}_A ) | &\leq \frac{ \mc{C} ( \mc{C}_0 + \mc{C}_1 ) }{ n } \frac{ r^2 }{ r_0^2 } \cdot \varepsilon \text{,} \\
\notag | \nabla^2 \eta ( \bar{E}_\theta, \bar{E}_\theta ) + \eta^2 \cdot 2 \varepsilon | &\leq \frac{ \mc{C} ( \mc{C}_0 + \mc{C}_1 ) }{ n } \frac{ \rho^2 }{ r^2 } \frac{ r^2 }{ r_0^2 } \cdot \varepsilon \text{.}
\end{align}
\end{itemize}
\end{proposition}

\begin{proof}
First, by \eqref{eq.pc_Eb} and \eqref{eq.pc_eta_est}, we obtain
\[
\bar{E}_A \eta = E_A \eta + \frac{r}{2} E_A \eta \cdot E_\rho \eta = \eta E_A \eta \text{.}
\]
The above and the definition \eqref{eq.pc_eta} of $\varepsilon$ then yield the first part of \eqref{eq.pc_etab_est}.
Similarly,
\[
\bar{E}_\theta \eta + \eta \cdot 2 \varepsilon t = \eta ( E_\theta \eta + 2 \varepsilon t ) \text{,}
\]
and the second estimate in \eqref{eq.pc_etab_est} also follows.

Next, for the first part of \eqref{eq.pc_etab_estd}, we expand using \eqref{eq.pc_Eb}:
\begin{align*}
\nabla_{ \bar{E}_A \bar{E}_B } \eta &= \nabla_{ E_A E_B } \eta + \frac{r}{2} E_A \eta \cdot \nabla_{ E_\rho E_B } \eta + \frac{r}{2} E_B \eta \cdot \nabla_{ E_\rho E_A } \eta + \frac{ r^2 }{4} E_A \eta E_B \eta \cdot \nabla_{ E_\rho E_\rho } \eta \\
&= \nabla_{ E_A E_B } \eta + \frac{r}{2} E_A \eta \cdot \nabla_{ E_\rho E_B } \eta + \frac{r}{2} E_B \eta \cdot \nabla_{ E_\rho E_A } \eta - \frac{ \varepsilon t^2 }{2} \cdot E_A \eta E_B \eta \text{.}
\end{align*}
Each term on the right-hand side can then be bounded using \eqref{eq.pc_eta} and Proposition \ref{thm.pc_eta_est},
\begin{align*}
| \nabla_{ \bar{E}_A \bar{E}_B } \eta | &\leq \frac{ \mc{C} ( \mc{C}_0 + \mc{C}_1 ) }{ n } \frac{ r^2 }{ \rho^2 } \frac{ r^2 }{ r_0^2 } \cdot \varepsilon + \frac{r}{2} \cdot \frac{ \mc{C} \mc{C}_0 }{ n } \frac{ r^2 }{ r_0^2 } \, \varepsilon t \cdot \frac{ \mc{C} \mc{C}_0 }{ n } \frac{t}{r} \frac{ r^2 }{ r_0^2 } \, \varepsilon \\
&\qquad + \frac{ \varepsilon t^2 }{2} \cdot \frac{ \mc{C} \mc{C}_0 }{ n } \frac{ r^2 }{ r_0^2 } \cdot \varepsilon t \cdot \frac{ \mc{C} \mc{C}_0 }{ n } \frac{ r^2 }{ r_0^2 } \cdot \varepsilon t \\
&\leq \frac{ \mc{C} ( \mc{C}_0 + \mc{C}_1 ) }{ n } \frac{ r^2 }{ \rho^2 } \frac{ r^2 }{ r_0^2 } \cdot \varepsilon \text{,}
\end{align*}
resulting in the first estimate in \eqref{eq.pc_etab_estd}.
Furthermore, a similar computation yields
\begin{align*}
| \nabla_{ \bar{E}_\theta \bar{E}_A } \eta | &\leq | \nabla_{ E_\theta E_A } \eta | + \frac{r}{2} | E_\theta \eta | | \nabla_{ E_\rho E_A } \eta | + \frac{r}{2} | E_A \eta | | \nabla_{ E_\rho E_\theta } \eta | + \frac{ \varepsilon t^2 }{2} | E_\theta \eta | | E_A \eta | \text{.}
\end{align*}
More applications of \eqref{eq.pc_eta} and Proposition \ref{thm.pc_eta_est} then lead to the second part of \eqref{eq.pc_etab_estd}:
\begin{align*}
| \nabla_{ \bar{E}_\theta \bar{E}_A } \eta | &\leq \frac{ \mc{C} ( \mc{C}_0 + \mc{C}_1 ) }{ n } \frac{ r^2 }{ r_0^2 } \cdot \varepsilon + \frac{r}{2} \left| - 2 \varepsilon t + \frac{ \mc{C} \mc{C}_0 }{ n } \frac{ \rho^2 }{ r^2 } \frac{ r^2 }{ r_0^2 } \, \varepsilon t \right| \frac{ \mc{C} \mc{C}_0 }{ n } \frac{t}{r} \frac{ r^2 }{ r_0^2 } \, \varepsilon \\
&\qquad + \frac{r}{2} \cdot \frac{ \mc{C} \mc{C}_0 }{ n } \frac{ r^2 }{ r_0^2 } \, \varepsilon t \left| - \frac{t}{r} \, 2 \varepsilon + \frac{ \mc{C} \mc{C}_0 }{ n } \frac{t}{r} \frac{ \rho^2 }{ r^2 } \frac{ r^2 }{ r_0^2 } \, \varepsilon \right| \\
&\qquad + \frac{ \varepsilon t^2 }{2} \left| - 2 \varepsilon t + \frac{ \mc{C} \mc{C}_0 }{ n } \frac{ \rho^2 }{ r^2 } \frac{ r^2 }{ r_0^2 } \, \varepsilon t \right| \frac{ \mc{C} \mc{C}_0 }{ n } \frac{ r^2 }{ r_0^2 } \, \varepsilon t \\
&\leq \frac{ \mc{C} ( \mc{C}_0 + \mc{C}_1 ) }{ n } \frac{ r^2 }{ r_0^2 } \cdot \varepsilon \text{.}
\end{align*}

For the final estimate of \eqref{eq.pc_etab_estd}, we again expand
\begin{equation}
\label{eql.pc_etab_est_1} \nabla_{ \bar{E}_\theta \bar{E}_\theta } \eta = \nabla_{ E_\theta E_\theta } \eta + r E_\theta \eta \cdot \nabla_{ E_\rho E_\theta } \eta - \frac{ \varepsilon t^2 }{2} \cdot E_\theta \eta E_\theta \eta \text{.}
\end{equation}
Notice that using \eqref{eq.pc_eta} and Proposition \ref{thm.pc_eta_est}, we can estimate
\begin{align*}
| \nabla_{ E_\theta E_\theta } \eta + 2 \varepsilon | &\leq \frac{ \mc{C} ( \mc{C}_0 + \mc{C}_1 ) }{ n } \frac{ \rho^2 }{ r^2 } \frac{ r^2 }{ r_0^2 } \cdot \varepsilon \text{,} \\
| r E_\theta \eta \cdot \nabla_{ E_\rho E_\theta } \eta - 2 \varepsilon t^2 \cdot 2 \varepsilon | &\leq \frac{ \mc{C} \mc{C}_0 }{ n } \frac{ \rho^2 }{ r^2 } \frac{ r^2 }{ r_0^2 } \cdot \varepsilon \text{,} \\
\left| - \frac{ \varepsilon t^2 }{2} ( E_\theta \eta E_\theta \eta - 4 \varepsilon^2 t^2 ) \right| &\leq \frac{ \mc{C} \mc{C}_0 }{ n } \frac{ \rho^2 }{ r^2 } \frac{ r^2 }{ r_0^2 } \cdot \varepsilon \text{.}
\end{align*}
Thus, combining \eqref{eql.pc_etab_est_1} with the above, we obtain, as desired,
\begin{align*}
| \nabla_{ \bar{E}_\theta \bar{E}_\theta } \eta + ( 1 - 2 \varepsilon t^2 + \varepsilon^2 t^4 ) 2 \varepsilon | &\leq | \nabla_{ E_\theta E_\theta } \eta + 2 \varepsilon | + | r E_\theta \eta \cdot \nabla_{ E_\rho E_\theta } \eta - 2 \varepsilon t^2 \cdot 2 \varepsilon | \\
&\qquad + \left| - \frac{ \varepsilon t^2 }{2} ( E_\theta \eta E_\theta \eta - 4 \varepsilon^2 t^2 ) \right| \\
&\leq \frac{ \mc{C} ( \mc{C}_0 + \mc{C}_1 ) }{ n } \frac{ \rho^2 }{ r^2 } \frac{ r^2 }{ r_0^2 } \cdot \varepsilon \text{.} \qedhere
\end{align*}
\end{proof}

\subsection{Pseudoconvexity}

We now derive, under Assumption \ref{ass.hf_curv}, pseudoconvexity properties for the modified hyperquadric $\bar{f}$.
The following proposition, which in particular implies the level sets of $\bar{f}$ are pseudoconvex, will play an essential role in our Carleman estimate:

\begin{theorem} \label{thm.pc}
There exist universal $\mc{C}_\ast, \mc{C}_\dagger > 0$ such that if Assumption \ref{ass.hf_curv} holds, and
\begin{equation}
\label{eq.pc_ass} 0 < \mc{C}_0 < \mc{C}_\ast \text{,} \qquad \mc{C}_0 \leq n \varepsilon_0 \mc{C}_\dagger \text{,} \qquad \mc{C}_1 \leq n \mc{C}_\dagger \text{,}
\end{equation}
then we have the inequality
\begin{equation}
\label{eq.pc} ( \nabla^2 \bar{f} - \bar{h} \cdot g ) ( \bar{X}, \bar{X} ) \geq \frac{ \varepsilon_0 }{8} \frac{ r^2 }{ r_0^2 } \cdot \bar{g}_+ ( \bar{X}, \bar{X} ) \text{,} \qquad \bar{h} := \frac{1}{2} \eta^{-1} - \frac{1}{4} \varepsilon r^2 \in C^\infty ( \bar{\mc{D}} ) \text{,}
\end{equation}
for any $\bar{f}$-tangent vector field $\bar{X}$ on $\mc{D}_{ r_0 }$.
\end{theorem}

\begin{proof}
First, we expand $\bar{X}$ in terms of the usual frames as
\begin{equation}
\label{eql.pc_XE} \bar{X} := X^\theta \bar{E}_\theta + \sum_A X^A \bar{E}_A \text{,}
\end{equation}
and we set $X := \bar{P} \bar{X}$.
By \eqref{eq.hf_q}, Propositions \ref{thm.pc_P} and \ref{thm.pc_fb_hessian}, and \eqref{eq.pc_fb}, we have
\begin{equation}
\label{eql.pc_1} ( \nabla^2 \bar{f} - \bar{h} g ) ( \bar{X}, \bar{X} ) = \eta^{-1} \cdot q ( X, X ) - \eta^{-2} f \cdot \nabla_{ \bar{X} \bar{X} } \eta + \frac{1}{4} \varepsilon r^2 \cdot g ( \bar{X}, \bar{X} ) \text{.}
\end{equation}

Using Proposition \ref{thm.hf_q_est}, \eqref{eq.pc_eta}, and Proposition \ref{thm.pc_gp_frame}, we bound
\begin{align}
\label{eql.pc_2} | \eta^{-1} \, q ( X, X ) | &\leq \mc{C} \left[ ( X^\theta )^2 | q_{ \theta \theta } | + 2 \sum_A | X^A X^\theta | | q_{ \theta A } | + \sum_{ A, B } | X^A X^B | | q_{ A B } | \right] \\
\notag &\leq \mc{C} \bar{g}_+ ( \bar{X}, \bar{X} ) \left[ \frac{ r^2 }{ \rho^2 } | q_{ \theta \theta } | + 2 \frac{ r }{ \rho } \sup_A | q_{ \theta A } | + \sup_{ A, B } | q_{ A B } | \right] \\
\notag &\leq \frac{ \mc{C} \mc{C}_0 }{n} \frac{ r^2 }{ r_0^2 } \cdot \bar{g}_+ ( \bar{X}, \bar{X} ) \text{.}
\end{align}
Moreover, we can write
\begin{align}
\notag - \eta^{-2} f \cdot \nabla_{ \bar{X} \bar{X} } \eta &= - \eta^{-2} f \left[ ( X^\theta )^2 \nabla_{ \bar{E}_\theta \bar{E}_\theta } \eta + \sum_A X^A X^\theta \nabla_{ \bar{E}_\theta \bar{E}_A } \eta + \sum_{ A, B } X^A X^B \nabla_{ \bar{E}_A \bar{E}_B } \eta \right] \\
\label{eql.pc_3} &= ( X^\theta )^2 f \cdot 2 \varepsilon + \mc{A} \text{,}
\end{align}
where the quantity $\mc{A}$ satisfies, due to \eqref{eq.pc_eta} and \eqref{eq.pc_etab_estd}, the estimate
\begin{align}
\notag | \mc{A} | &\leq \mc{C} f \left[ ( X^\theta )^2 | \nabla_{ \bar{E}_\theta \bar{E}_\theta } \eta + \eta^2 \, 2 \varepsilon | + \sum_A X^A X^\theta | \nabla_{ \bar{E}_\theta \bar{E}_A } \eta | + \sum_{ A, B } X^A X^B | \nabla_{ \bar{E}_A \bar{E}_B } \eta | \right] \\
\notag &\leq \mc{C} f \, \bar{g}_+ ( \bar{X}, \bar{X} ) \left[ \frac{ r^2 }{ \rho^2 } | \nabla_{ \bar{E}_\theta \bar{E}_\theta } \eta + \eta^2 \, 2 \varepsilon | + \frac{ r }{ \rho } \sup_A | \nabla_{ \bar{E}_\theta \bar{E}_A } \eta | + \sup_{ A, B } | \nabla_{ \bar{E}_A \bar{E}_B } \eta | \right] \\
\notag &\leq f \, \bar{g}_+ ( \bar{X}, \bar{X} ) \cdot \frac{ \mc{C} ( \mc{C}_0 + \mc{C}_1 ) }{ n } \frac{ r^2 }{ \rho^2 } \frac{ r^2 }{ r_0^2 } \, \varepsilon \\
\label{eql.pc_4} &\leq \frac{ \varepsilon_0 \mc{C} ( \mc{C}_0 + \mc{C}_1 ) }{ n } \frac{ r^2 }{ r_0^2 } \cdot \bar{g}_+ ( \bar{X}, \bar{X} ) \text{.}
\end{align}
Combining \eqref{eql.pc_1}--\eqref{eql.pc_4}, we then see that
\begin{equation}
\label{eql.pc_10} ( \nabla^2 \bar{f} - \bar{h} g ) ( \bar{X}, \bar{X} ) \geq 2 \varepsilon f \cdot ( X^\theta )^2 + \frac{1}{4} \varepsilon r^2 \cdot g ( \bar{X}, \bar{X} ) - \frac{ \mc{C} ( \mc{C}_0 + \varepsilon_0 \mc{C}_1 ) }{n} \frac{ r^2 }{ r_0^2 } \cdot \bar{g}_+ ( \bar{X}, \bar{X} ) \text{.}
\end{equation}

Next, from \eqref{eq.hf_gauss}, \eqref{eq.pc_eta}, \eqref{eq.pc_P}, Proposition \ref{thm.pc_P}, and Proposition \ref{thm.pc_gp_frame}, we obtain
\begin{align}
\label{eql.pc_11} g ( \bar{X}, \bar{X} ) &= g ( X + X \eta \cdot \nabla^\sharp f, X + X \eta \cdot \nabla^\sharp f ) \\
\notag &= g ( X, X ) + f ( X \eta )^2 \\
\notag &\geq - \frac{ \rho^2 }{ r^2 } ( X^\theta )^2 + \sum_A ( X^A )^2 \text{.}
\end{align}
Finally, from \eqref{eq.hf_rho}, \eqref{eq.pc_eta}, \eqref{eql.pc_10}, and \eqref{eql.pc_11}, we derive the bound
\begin{align*}
( \nabla^2 \bar{f} - \bar{h} g ) ( \bar{X}, \bar{X} ) &\geq 2 \varepsilon f ( X^\theta )^2 - \frac{1}{4} \varepsilon r^2 \, \frac{ \rho^2 }{ r^2 } ( X^\theta )^2 + \frac{1}{4} \varepsilon r^2 \sum_A ( X^A )^2 \\
&\qquad - \frac{ \mc{C} ( \mc{C}_0 + \varepsilon_0 \mc{C}_1 ) }{n} \frac{ r^2 }{ r_0^2 } \cdot \bar{g}_+ ( \bar{X}, \bar{X} ) \\
&= \frac{1}{4} \varepsilon \rho^2 ( X^\theta )^2 + \frac{1}{4} \varepsilon r^2 \sum_A ( X^A )^2 - \frac{ \mc{C} ( \mc{C}_0 + \varepsilon_0 \mc{C}_1 ) }{n} \frac{ r^2 }{ r_0^2 } \cdot \bar{g}_+ ( \bar{X}, \bar{X} ) \text{.}
\end{align*}

Recalling \eqref{eq.pc_eta} and Proposition \ref{thm.pc_gp_frame}, the above becomes
\begin{align*}
( \nabla^2 \bar{f} - \bar{h} g ) ( \bar{X}, \bar{X} ) &\geq \frac{1}{4} \varepsilon r^2 \cdot \bar{g}_+ ( \bar{X}, \bar{X} ) - \frac{ \mc{C} ( \mc{C}_0 + \varepsilon_0 \mc{C}_1 ) }{n} \frac{ r^2 }{ r_0^2 } \cdot \bar{g}_+ ( \bar{X}, \bar{X} ) \\
&= \frac{ r^2 }{ r_0^2 } \left[ \frac{ \varepsilon_0 }{ 4 } - \frac{ \mc{C} ( \mc{C}_0 + \varepsilon_0 \mc{C}_1 ) }{n} \right] \bar{g}_+ ( \bar{X}, \bar{X} ) \text{.}
\end{align*}
Thus, as long as $16 \mc{C} \mc{C}_0 \leq n \varepsilon_0$ and $16 \mc{C} \mc{C}_1 \leq n$, the above implies \eqref{eq.pc}, as desired.
\end{proof}

\section{Carleman estimate} \label{sec.CE}

In this section, we will prove the following Carleman estimate:

\begin{theorem} \label{thm_carl_est_LM}
Assume the setting described in Assumption \ref{ass_LM}; in particular, let $\mc{D}$ be as in \eqref{eq.hf_D}.
Then, there exist universal constants $\mc{C}_\dagger >0$ and $0 < \varepsilon_0 \ll 1$ such that if
\begin{equation} \label{eq_thm_carl_cc0_LM}
\begin{aligned}
\sup_{ X, Y, Z \in \{ E_\rho, E_0, E_A \} } | R ( E_\rho, X, Y, Z ) | & < \frac{\varepsilon_0 \mc{C}_\dagger}{r_0^2}, \\ \sup_{ X, Y, Z \in \{ E_\rho, E_0, E_A \} } | \nabla_Z R ( E_\rho, X, E_\rho, Y ) | & < \frac{\mc{C}_\dagger}{r_0^3},
\end{aligned}
\end{equation}
then for any $\phi \in C^2(\mc{U}) \cap C^1(\bar{\mc{U}})$ satisfying
\begin{equation} \label{eq_thm_carl_p0_LM}
\phi|_{\partial\mc{U} \cap \mc{D}} = 0 \text{,}
\end{equation}
and for any constants $a, b_0, b, \varepsilon > 0$ satisfying
\begin{equation} \label{eq_thm_carl_a_LM}
a \geqslant n^2 \text{,} \qquad \varepsilon_0 \ll b_0 \ll 1 \text{,} \qquad \varepsilon := \varepsilon_0 r_0^{-2} \text{,} \qquad b := b_0 r_0^{-2} \text{,}
\end{equation}
we have the Carleman estimate
\begin{align} \label{eq_thm_carl_est_LM} 
\frac{\varepsilon}{64} \int_{\mc{U} \cap \mc{D}} \zeta r^2 & \left[ \frac{r^2}{\rho^2} (E_\rho \phi)^2 + \frac{r^2}{\rho^2} (E_\theta \phi)^2 + \sum_A (E_A \phi)^2 \right]+ \frac{1}{8} a^2 b \int_{\mc{U} \cap \mc{D}} \zeta \phi^2 \\
& \qquad \leqslant \frac{1}{4a} \int_{\mc{U} \cap \mc{D}} \zeta f |\square \phi|^2 + \frac{1}{2} \int_{\partial \mc{U} \cap \mc{D}} \zeta \mc{N} (f [1- \varepsilon t^2 ]^{-1}) |\mc{N} \phi|^2 \text{,} \notag
\end{align}
where $\zeta$ is the Carleman weight given by
\begin{equation} \label{eq_thm_carl_def_LM}
\zeta = \left\{ \frac{f}{(1-\varepsilon t^2)} \cdot \exp\left[ \frac{bf}{(1-\varepsilon t^2)}\right] \right\}^{2a},
\end{equation}
and where $\mc{N}$ is the outer-pointing unit normal of $\mc{U}$.
\end{theorem}

Throughout the rest of this section we assume that the hypothesis of Theorem \ref{thm_carl_est_LM} is true. Note that, \eqref{eq_thm_carl_cc0_LM} implies that as long as $\mc{C}_\dagger$ is not too large, there exist constants $\mc{C}_0$, $\mc{C}_1>0$, such that Assumption \ref{ass.hf_curv} is satisfied and equation \eqref{eq.pc_ass} holds.

\begin{remark}
Note that since $\mc{U} \cap \mc{D} \subset \mc{D}_{r_0}$ from Assumption \ref{ass_LM}, all the geometric estimates of the previous sections are valid on $\mc{U} \cap \mc{D}$.
\end{remark}

\begin{remark}
The $E_\alpha$ derivatives in \eqref{eq_thm_carl_est_LM} can be reformulated in terms of the normal coordinate derivatives, with some additional estimates.
\end{remark}

\subsection{Preliminary estimates}

First, we provide some additional computations and estimates that will serve as crucial ingredients for proving the above Carleman estimate.
Throughout, we will use $\mc{O} ( \xi )$, where $\xi \in C^\infty ( \mc{D}_{ r_0 } )$, to denote any function $\chi\in C^\infty ( \mc{D}_{ r_0 } )$ that satisfies, for some universal $\mc{C} > 0$,
\begin{equation}
\label{eq.misc_O} | \chi | \leq \mc{C} \xi \text{.}
\end{equation}
Henceforth, we will make repeated use of Proposition \ref{thm.pc_eta_est} and Proposition \ref{thm.pc_etab_est} to estimate the derivatives of $\eta$, whenever necessary. Moreover, we will use \eqref{eq_thm_carl_cc0_LM} and the $\mc{O}(\cdot)$ notation to simplify these estimates.

Now, in Proposition \ref{thm.pc} we showed that $\bar{f}$ satisfies the pseudoconvexity requirement. Hence, to proceed with our proof we now need to derive some estimates for $\bar{f}$. One important step is to quantify how much $\bar{f}$ deviates from the Gauss lemma properties satisfied by $f$ in Proposition \ref{thm.hf_gauss}. 

\begin{proposition}[Shifted Gauss Lemma] \label{thm.misc_gauss}
The following identities hold on $\mc{D}_{ r_0 }$,
\begin{align}
\label{eq.misc_gauss} \nabla^\alpha \bar{f} \nabla_\alpha \bar{f} &= \bar{f} \eta^{-2} ( 1 + \varepsilon t^2 + \mc{E}_1 ) \text{,} \\
\notag \nabla_{ \alpha \beta } \bar{f} \nabla^\alpha \bar{f} \nabla^\beta \bar{f} &= \frac{1}{2} \bar{f} \eta^{-4} ( 1 + 5 \varepsilon t^2 + 2 \varepsilon^2 t^4 + \mc{E}_2 ) \text{,}
\end{align}
where the quantities $\mc{E}_1$ and $\mc{E}_2$ satisfy, for some universal $\mc{C} > 0$,
\begin{equation}
\label{eq.misc_gauss_error} | \mc{E}_1 | \leq \mc{C} \varepsilon^2 t^2 f \text{,} \qquad | \mc{E}_2 | \leq \mc{C} \varepsilon^2 t^2 f \text{.}
\end{equation}
\end{proposition}

\begin{proof}
First, expanding \eqref{eq.pc_fb} and applying \eqref{eq.hf_gauss}, we have that
\begin{align}
\label{eql.misc_gauss_1} \nabla^\alpha \bar{f} \nabla_\alpha \bar{f} &= \eta^{-2} \nabla^\alpha f \nabla_\alpha f - 2 f \eta^{-3} \nabla^\alpha \eta \nabla_\alpha f + f^2 \eta^{-4} \nabla^\alpha \eta \nabla_\alpha \eta \\
\notag &= f \eta^{-4} ( \eta^2 - 2 \eta \nabla^\alpha \eta \nabla_\alpha f + f \nabla^\alpha \eta \nabla_\alpha \eta ) \text{.}
\end{align}
In addition, notice from \eqref{eq.hf_gauss}, \eqref{eq.hf_E_rho}, and \eqref{eq.pc_eta_est} that
\begin{equation}
\label{eql.misc_gauss_2} \nabla^\alpha \eta \nabla_\alpha f = \frac{ r^2 }{ \rho^2 } \nabla_\rho \eta \nabla_\rho f = - \varepsilon t^2 \text{.}
\end{equation}
Similarly, by \eqref{eq.hf_E} and \eqref{eq.pc_eta_est}, we have
\begin{align}
\label{eql.misc_gauss_3} \nabla^\alpha \eta \nabla_\alpha \eta &= \frac{ r^2 }{ \rho^2 } ( \nabla_\rho \eta )^2 - \frac{ r^2 }{ \rho^2 } ( \nabla_\theta \eta )^2 + \sum_A ( \nabla_A \eta )^2 \\
\notag &= \frac{ 4 \varepsilon^2 t^4 }{ \rho^2 } - \frac{ 4 \varepsilon^2 t^2 r^2 }{ \rho^2 } + \mc{O} \left( \frac{ \mc{C}_0 }{n} \frac{ r^2 }{ r_0^2 } \cdot \varepsilon^2 t^2 \right) \\
\notag &= - 4 \varepsilon^2 t^2 + \mc{O} \left( \frac{ \mc{C}_0 }{n} \frac{ r^2 }{ r_0^2 } \cdot \varepsilon^2 t^2 \right) \text{.}
\end{align}
Combining \eqref{eql.misc_gauss_1}--\eqref{eql.misc_gauss_3} yields
\begin{align}
\label{eql.misc_gauss_4} \nabla^\alpha \bar{f} \nabla_\alpha \bar{f} &= f \eta^{-4} \left[ \eta^2 + 2 \eta \cdot \varepsilon t^2 - 4 f \varepsilon^2 t^2 + \mc{O} \left( \frac{ \mc{C}_0 }{n} \frac{ \rho^2 r^2 }{ r_0^2 } \cdot \varepsilon^2 t^2 \right) \right] \\
\notag &= f \eta^{-4} [ \eta ( 1 + \varepsilon t^2 ) + \mc{O} ( \varepsilon^2 t^2 \rho^2 ) ] \text{,}
\end{align}
and the first parts of \eqref{eq.misc_gauss} and \eqref{eq.misc_gauss_error} follow.

Next, we expand $\nabla^2 \bar{f}$ as
\begin{align*}
\nabla_{ \alpha \beta } \bar{f} &= \nabla_\alpha ( \eta^{-1} \nabla_\beta f ) - \nabla_\alpha ( f \eta^{-2} \nabla_\beta \eta ) \\
&= \eta^{-1} \nabla_{ \alpha \beta } f - \eta^{-2} \nabla_\alpha \eta \nabla_\beta f - \eta^{-2} \nabla_\alpha f \nabla_\beta \eta \\
&\qquad + 2 f \eta^{-3} \nabla_\alpha \eta \nabla_\beta \eta - f \eta^{-2} \nabla_{ \alpha \beta } \eta \text{.}
\end{align*}
The above then implies
\begin{align}
\label{eql.misc_gauss_10} \nabla_{ \alpha \beta } \bar{f} \nabla^\alpha \bar{f} \nabla^\beta \bar{f} &= ( \eta^{-1} \nabla_{ \alpha \beta } f - 2 \eta^{-2} \nabla_\alpha \eta \nabla_\beta f + 2 f \eta^{-3} \nabla_\alpha \eta \nabla_\beta \eta - f \eta^{-2} \nabla_{ \alpha \beta } \eta ) \\
\notag &\qquad \cdot ( \eta^{-1} \nabla^\alpha f - f \eta^{-2} \nabla^\alpha \eta ) ( \eta^{-1} \nabla^\beta f - f \eta^{-2} \nabla^\beta \eta ) \\
\notag &= \eta^{-3} \nabla_{ \alpha \beta } f \nabla^\alpha f \nabla^\beta f - 2 f \eta^{-4} \nabla_{ \alpha \beta } f \nabla^\alpha f \nabla^\beta \eta + f^2 \eta^{-5} \nabla_{ \alpha \beta } f \nabla^\alpha \eta \nabla^\beta \eta \\
\notag &\qquad - 2 \eta^{-4} \nabla_\alpha \eta \nabla^\alpha f \nabla_\beta f \nabla^\beta f + 2 f \eta^{-5} \nabla_\alpha \eta \nabla^\alpha \eta \nabla_\beta f \nabla^\beta f \\
\notag &\qquad + 2 f \eta^{-5} \nabla_\alpha \eta \nabla^\alpha f \nabla_\beta f \nabla^\beta \eta - 2 f^2 \eta^{-6} \nabla_\alpha \eta \nabla^\alpha \eta \nabla_\beta f \nabla^\beta \eta \\
\notag &\qquad + 2 f \eta^{-5} \nabla_\alpha \eta \nabla^\alpha f \nabla_\beta \eta \nabla^\beta f - 4 f^2 \eta^{-6} \nabla_\alpha \eta \nabla^\alpha f \nabla_\beta \eta \nabla^\beta \eta \\
\notag &\qquad + 2 f^3 \eta^{-7} \nabla_\alpha \eta \nabla^\alpha \eta \nabla_\beta \eta \nabla^\beta \eta - f \eta^{-4} \nabla_{ \alpha \beta } \eta \nabla^\alpha f \nabla^\beta f \\
\notag &\qquad + 2 f^2 \eta^{-5} \nabla_{ \alpha \beta } \eta \nabla^\alpha f \nabla^\beta \eta - f^3 \eta^{-6} \nabla_{ \alpha \beta } \eta \nabla^\alpha \eta \nabla^\beta \eta \\
\notag &= I_1 + \dots + I_{13} \text{.}
\end{align}
Notice that \eqref{eq.hf_gauss_2} immediately implies
\begin{equation}
\label{eql.misc_gauss_11} I_1 = \frac{1}{2} f \eta^{-3} = \frac{1}{2} f \eta^{-7} \cdot \eta^4 \text{.}
\end{equation}
Next, using \eqref{eq.hf_gauss_2}, \eqref{eq.hf_q}, \eqref{eq.hf_q_est}, \eqref{eq.pc_eta_est}, \eqref{eql.misc_gauss_2}, and \eqref{eql.misc_gauss_3}, we also obtain
\begin{align}
\label{eql.misc_gauss_12} I_2 + I_4 &= - f \eta^{-4} \nabla_\alpha f \nabla^\alpha \eta - 2 f \eta^{-4} \nabla_\alpha f \nabla^\alpha \eta \\
\notag &= \frac{1}{2} f \eta^{-7} \cdot ( 6 \eta^3 \cdot \varepsilon t^2 ) \text{,} \\
\notag I_3 + I_5 &= f^2 \eta^{-5} \left[ \frac{1}{2} \nabla_\alpha \eta \nabla^\alpha \eta + \mc{O} \left( \frac{ \mc{C} \mc{C}_0 }{n} \frac{ r^2 }{ r_0^2 } \cdot \varepsilon^2 t^2 \right) \right] + 2 f^2 \eta^{-5} \nabla_\alpha \eta \nabla^\alpha \eta \\
\notag &= \frac{1}{2} f \eta^{-7} \cdot \mc{O} ( \varepsilon^2 t^2 \rho^2 ) \text{,} \\
\notag I_6 + I_8 &= f \eta^{-5} ( 4 \nabla^\alpha f \nabla_\alpha \eta \nabla^\beta f \nabla_\beta \eta ) \\
\notag &= \frac{1}{2} f \eta^{-7} ( 8 \eta^2 \cdot \varepsilon^2 t^4 ) \text{,} \\
\notag I_7 + I_9 &= - 6 f^2 \eta^{-6} \nabla^\alpha f \nabla_\alpha \eta \nabla^\beta \eta \nabla_\beta \eta \\
\notag &= \frac{1}{2} f \eta^{-7} \cdot \mc{O} ( \varepsilon^2 t^2 \rho^2 ) \text{,} \\
\notag I_{10} &= \frac{1}{2} f \eta^{-7} \cdot \mc{O} ( \varepsilon^2 t^2 \rho^2 ) \text{.}
\end{align}
Similarly, using the same observations, along with \eqref{eq.hf_gauss}, \eqref{eq.hf_E_rho}, and Proposition \ref{thm.pc_eta_est}, we have
\begin{align}
\label{eql.misc_gauss_13} I_{11} &= - f \eta^{-4} \cdot \frac{ r^2 }{4} \nabla_{ \rho \rho } \eta \\
\notag &= \frac{1}{2} f \eta^{-7} ( \eta^3 \cdot \varepsilon t^2 ) \\
\notag I_{12} &= 2 f^2 \eta^{-5} \cdot \frac{r}{2} \nabla_{ \rho \beta } \eta \nabla^\beta \eta \\
\notag &= f^2 \eta^{-5} r \left( \frac{ r^2 }{ \rho^2 } \nabla_{ \rho \rho } \eta \nabla_\rho \eta - \frac{ r^2 }{ \rho^2 } \nabla_{ \rho \theta } \eta \nabla_\theta \eta + \sum_A \nabla_{ \rho A } \eta \nabla_A \eta \right) \\
\notag &= f^2 \eta^{-5} r \left[ \frac{ r^2 }{ \rho^2 } \frac{ 2 \varepsilon t^2 }{ r^2 } \frac{ 2 \varepsilon t^2 }{r} - \frac{ r^2 }{ \rho^2 } \frac{ 2 \varepsilon t }{r} 2 \varepsilon t + \mc{O} \left( \frac{ \varepsilon^2 t^2 }{ r } \right) \right] \\
\notag &= \frac{1}{2} f \eta^{-7} \cdot \mc{O} ( \varepsilon^2 t^2 \rho^2 ) \\
\notag I_{13} &= f^3 \eta^{-6} \left[ \frac{ r^4 }{ \rho^4 } \frac{ 2 \varepsilon t^2 }{ r^2 } \frac{ 2 \varepsilon t^2 }{r} \frac{ 2 \varepsilon t^2 }{r} - \frac{ 2 r^4 }{ \rho^4 } \frac{ 2 \varepsilon t }{r} \frac{ 2 \varepsilon t^2 }{r} 2 \varepsilon t + \frac{ r^2 }{ \rho^4 } 2 \varepsilon 2 \varepsilon t 2 \varepsilon t + \mc{O} \left( \frac{ \varepsilon^3 t^2 r^2 }{ \rho^2 } \right) \right] \\
\notag &= \frac{1}{2} f \eta^{-6} [ \varepsilon^3 t^2 \rho^4 + \mc{O} ( \varepsilon^3 t^2 r^2 \rho^2 ) ] \\
\notag &= \frac{1}{2} f \eta^{-7} \cdot \mc{O} ( \varepsilon^2 t^2 \rho^2 ) \text{.}
\end{align}
Thus, combining \eqref{eql.misc_gauss_10}--\eqref{eql.misc_gauss_13} yields
\begin{align}
\label{eql.misc_gauss_14} \nabla_{ \alpha \beta } \bar{f} \nabla^\alpha \bar{f} \nabla^\beta \bar{f} &= \frac{1}{2} f \eta^{-7} [ \eta^4 + 7 \eta^3 \cdot \varepsilon t^2 + 8 \eta^2 \cdot \varepsilon^2 t^4 + \mc{O} ( \varepsilon^2 t^2 \rho^2 ) ] \\
\notag &= \frac{1}{2} f \eta^{-5} [ 1 + 5 \varepsilon t^2 + 2 \varepsilon^2 t^4 + \mc{O} ( \varepsilon^2 t^2 \rho^2 ) ] \text{,}
\end{align}
from which we obtain the second part of \eqref{eq.misc_gauss}.
\end{proof}

For our subsequent computations, it will be useful complete the frame corresponding to the vector fields $\bar{E}_\theta, \bar{E}_A$. For this purpose, we will introduce a vector field $\bar{E}_\rho$. Note that, due to \eqref{eq.hf_rho} and \eqref{eq.hf_E_rho}, we get the following
\begin{align}
f = \frac{\rho^2}{4}, & \quad E_\rho = \frac{\rho}{r} \partial_\rho = \partial_r + \frac{t}{r} \partial_t, \\
E_\rho f &= \frac{2f}{r} = \frac{\rho^2}{2r}. \notag
\end{align}
\begin{definition}
Define the vector field $ \bar{E}_\rho $ as follows
\begin{equation} \label{eq_def_barrho}
\bar{E}_\rho := \frac{2}{r} \nb^\sharp \bar{f}.
\end{equation}
\end{definition}
\begin{proposition} \label{prop_barrho}
We have the following representation for $ \bar{E}_\rho $
\begin{equation} \label{eq_barrho1}
\bar{E}_\rho = \eta^{-2} E_\rho + \eta^{-2} \frac{r E_\theta \eta}{2} E_\theta - \eta^{-2} \frac{\rho^2}{r^2} \sum_A \frac{r E_A \eta}{2} E_A. 
\end{equation}
Moreover, we have
\begin{equation} \label{eq_barrho2}
\bar{E}_\rho = \eta^{-2} [1 + \mc{O}(\varepsilon^2 t^2 r^2)] E_\rho + \eta^{-2} \frac{r E_\theta \eta}{2} \bar{E}_\theta  - \eta^{-2} \frac{\rho^2}{r^2} \sum_A \frac{r E_A \eta}{2} \bar{E}_A.
\end{equation}
In particular,
\begin{equation} \label{eq_del_rbrf}
\bar{E}_\rho \bar{f} = \frac{2 f}{r} \eta^{-4} [1 + \mc{O}(\varepsilon t^2)].
\end{equation}
\end{proposition}
\begin{proof}
We compute as follows
\begin{align*}
\bar{E}_\rho & = \frac{2}{r} \nb^\alpha \bar{f} \nb_\alpha \\
& = \frac{2}{r} \eta^{-1} \nb^\alpha f \nb_\alpha - \frac{2}{r} \eta^{-2} f \nb^\alpha \eta \nb_\alpha \\
& = \frac{2}{r} \eta^{-1} \frac{r}{2} E_\rho - \eta^{-2} \frac{\rho^2}{2r} \left[ \frac{r^2}{\rho^2} E_\rho \eta E_\rho - \frac{r^2}{\rho^2} E_\theta \eta E_\theta + \sum_A E_A \eta E_A \right] \\
& = \eta^{-1} E_\rho - \eta^{-2} \frac{r}{2} \left( \frac{-2 \varepsilon t^2}{r} \right) E_\rho + \eta^{-2} \frac{r E_\theta \eta}{2} E_\theta - \eta^{-2} \frac{\rho^2}{r^2} \sum_A \frac{r E_A \eta}{2} E_A \\
& = \eta^{-2} E_\rho + \eta^{-2} \frac{r E_\theta \eta}{2} E_\theta - \eta^{-2} \frac{\rho^2}{r^2} \sum_A \frac{r E_A \eta}{2} E_A,
\end{align*}
which shows that \eqref{eq_barrho1} is true.
To prove \eqref{eq_barrho2}, using \eqref{eq.pc_Eb} shows
\begin{align*}
\bar{E}_\rho & = \eta^{-2} E_\rho + \eta^{-2} \frac{r E_\theta \eta}{2} \bar{E}_\theta - \eta^{-2} \left( \frac{r E_\theta \eta}{2} \right)^2 E_\rho - \eta^{-2} \frac{\rho^2}{r^2} \sum_A \frac{r E_A \eta}{2} \bar{E}_A \\
& \qquad + \eta^{-2} \frac{\rho^2}{r^2} \sum_A \left( \frac{r E_A \eta}{2} \right)^2 E_\rho \\
& = \eta^{-2} \left[ 1 - \left( \frac{r E_\theta \eta}{2} \right)^2 + \frac{\rho^2}{r^2} \sum_A \left( \frac{r E_A \eta}{2} \right)^2 \right] E_\rho + \eta^{-2} \frac{r E_\theta \eta}{2} \bar{E}_\theta  - \eta^{-2} \frac{\rho^2}{r^2} \sum_A \frac{r E_A \eta}{2} \bar{E}_A \\
& = \eta^{-2} [1 + \mc{O}(\varepsilon^2 t^2 r^2)] E_\rho + \eta^{-2} \frac{r E_\theta \eta}{2} \bar{E}_\theta  - \eta^{-2} \frac{\rho^2}{r^2} \sum_A \frac{r E_A \eta}{2} \bar{E}_A,
\end{align*}
which completes the proof of \eqref{eq_barrho2}.

Next, since $\bar{E}_\theta$ and $\bar{E}_A$ are tangent to $\bar{f}$, \eqref{eq_barrho2} implies that
\begin{align*}
\bar{E}_\rho \bar{f} = \eta^{-2} [1 + \mc{O}(\varepsilon^2 t^2 r^2)] E_\rho \bar{f} =  \eta^{-2} [1 + \mc{O}(\varepsilon t^2)] \frac{2 f}{r} \eta^{-2}.
\end{align*}
This completes the proof of the proposition.
\end{proof}

Now we will compute the metric components with respect to the frames $\bar{E}_\alpha$. In particular, we will see that the frames $\bar{E}_\alpha$'s are not mutually orthogonal. 
\begin{proposition} \label{prop_gcomp}
We have the following
\begin{align} \label{eq_gcomp}
g (\bar{E}_\rho, \bar{E}_\rho) = \frac{\rho^2}{r^2} \cdot \eta^{-3} [1+ \mc{O} (\varepsilon t^2) ], & \qquad g(\bar{E}_\theta,\bar{E}_\theta) = \left[ -1 + \left(\frac{r E_\theta \eta}{2}\right)^2 \right] \frac{\rho^2}{r^2},\\ 
g(\bar{E}_\theta, \bar{E}_A) = \left( \frac{r E_\theta \eta}{2} \cdot \frac{r E_A \eta}{2} \right) \frac{\rho^2}{r^2}, & \qquad g(\bar{E}_A,\bar{E}_B) = \delta_{AB} + \left( \frac{r E_A \eta}{2} \cdot \frac{r E_B \eta}{2} \right) \frac{\rho^2}{r^2}. \notag
\end{align}
Also,
\begin{equation} \label{eq_gcomp_00}
g (\bar{E}_\rho, \bar{E}_\theta) = g (\bar{E}_\rho, \bar{E}_A) = 0.
\end{equation}
\end{proposition}
\begin{proof}
For the first equation, using \eqref{eq_def_barrho} shows that
\begin{align*}
\bar{E}_\rho =  \eta^{-1} E_\rho - \frac{2}{r} \eta^{-2} f \nb^\alpha \eta \nb_\alpha.
\end{align*}
Hence,
\begin{align*}
g (\bar{E}_\rho, \bar{E}_\rho) & = g (\eta^{-1} E_\rho - \frac{2}{r} \eta^{-2} f \nb^\sharp \eta,\eta^{-1} E_\rho - \frac{2}{r} \eta^{-2} f \nb^\sharp \eta) \\
& = \eta^{-2} g(E_\rho,E_\rho) - \frac{2}{r} \eta^{-3} f E_\rho \eta + \frac{4}{r^2} \eta^{-4} f^2 \nb^\alpha \eta \nb_\alpha \eta \\
& = \eta^{-2} \frac{\rho^2}{r^2} + \eta^{-3} \varepsilon t^2 \frac{\rho^2}{r^2} +  \eta^{-4} \frac{\rho^2}{r^2} f \nb^\alpha \eta \nb_\alpha \eta \\
& = \eta^{-3} \frac{\rho^2}{r^2} [1+ \eta^{-1} f \nb^\alpha \eta \nb_\alpha \eta ],
\end{align*}
which completes the proof of the first equation, since $\eta^{-1} f \nb^\alpha \eta \nb_\alpha \eta = \mc{O} (\varepsilon t^2)$ (similar ideas were used in \eqref{eql.misc_gauss_3} and \eqref{eql.misc_gauss_4}).

For the second equation, we have
\begin{align*}
g(\bar{E}_\theta,\bar{E}_\theta) & = g(E_\theta, E_\theta) + \frac{r}{2} E_\theta \eta \cdot g(E_\theta, E_\rho) + \frac{r^2}{4} (E_\theta \eta)^2 g(E_\rho, E_\rho) \\
& = -\frac{\rho^2}{r^2} + \frac{r^2}{4} (E_\theta \eta)^2 \cdot \frac{\rho^2}{r^2} \\
& = \left[ -1 + \left(\frac{r E_\theta \eta}{2}\right)^2 \right] \frac{\rho^2}{r^2}.
\end{align*}
Similarly, we can prove the remaining two equations in \eqref{eq_gcomp}. Finally, \eqref{eq_gcomp_00} is a consequence of the fact that $\bar{E}_\theta$ and $\bar{E}_A$ are $\bar{f}$-tangent.
\end{proof}

The next result is a nice consequence of the Gauss Lemma, and it will help us to estimate the components of $\nb^2 \bar{f}$.
\begin{proposition} \label{prop_del2f}
The following is satisfied
\begin{equation} \label{eq_del2f}
\nabla_{\rho \alpha} f = \frac{1}{r} \nabla_\alpha f .
\end{equation}
\end{proposition}
\begin{proof}
Due to \eqref{eq.hf_gauss} and \eqref{eq.hf_gauss_2}, we have
\[
\nabla_{\rho \alpha} f = \frac{2}{r} \nb^\mu f \nb_{\mu \alpha} f = \frac{1}{r} \nb_\alpha f. \qedhere 
\]
\end{proof}

\begin{proposition} \label{prop_del2_fbar}
The components of $ \nb^2 \bar{f} $ satisfy the following equations
\begin{align*}
\nabla_{\rho\rho} \bar{f} & = \left( 1 + 3 \varepsilon t^2 \right)\frac{\eta^{-3}}{2} \cdot \frac{\rho^2}{r^2}, \\
\nabla_{\rho\bar{\theta}} \bar{f} & = (-r E_\theta \eta - r^2 \nabla_{\rho \theta} \eta) \frac{\eta^{-2}}{4} \cdot \frac{\rho^2}{r^2}, \numberthis \label{eq_del2_fbar} \\
\nabla_{\rho\bar{A}} \bar{f} & = (-r E_A \eta - r^2 \nabla_{\rho A} \eta) \frac{\eta^{-2}}{4} \cdot \frac{\rho^2}{r^2}.
\end{align*}
\end{proposition}
\begin{proof}
Firstly, note that 
\begin{align*}
\label{eq_naalbe} \numberthis \nabla_{\alpha\beta} \bar{f} & = \eta^{-1} \nabla_{\alpha\beta} f + \nabla_\alpha f \nabla_\beta \eta^{-1} + \nabla_\alpha \eta^{-1} \nabla_\beta f + f \nabla_{\alpha \beta} \eta^{-1} \\
& = \eta^{-1} \nabla_{\alpha \beta} f - \eta^{-2} \nabla_\alpha f \nabla_\beta \eta - \eta^{-2} \nabla_\alpha \eta \nabla_\beta f + 2 \eta^{-3} f \nabla_\alpha \eta \nabla_\beta \eta - \eta^{-2} f \nabla_{\alpha \beta} \eta.
\end{align*}
Then, we use \eqref{eq_del2f} to get
\begin{equation}\label{eq_rho2f}
\nabla_{\rho\rho} f = \frac{1}{r} \nabla_\rho f = \frac{1}{r} \cdot \frac{2f}{r} = \frac{2f}{r^2}.
\end{equation}
For the \( (E_\rho,E_\rho) \) component, using \eqref{eq_del2f}, \eqref{eq_naalbe}, \eqref{eq_rho2f}, and Proposition \ref{thm.pc_eta_est}, we get
\begin{align*}
\nabla_{\rho\rho} \bar{f} & = \eta^{-1} \nabla_{\rho\rho} f - 2 \eta^{-2} \nabla_\rho f \nabla_\rho \eta + 2 \eta^{-3} f \nabla_\rho \eta \nabla_\rho \eta - \eta^{-2} f \nabla_{\rho\rho} \eta \\
& = \eta^{-1} \frac{2f}{r^2} - 2 \eta^{-2} \cdot \frac{2f}{r} \left( \frac{-2 \varepsilon t^2}{r} \right) + 2 \eta^{-3} f \left( \frac{-2 \varepsilon t^2}{r} \right)^2 - \eta^{-2} f \left( \frac{-2 \varepsilon t^2}{r^2} \right) \\
& = \left( 1 + 4 \eta^{-1} \varepsilon t^2 + 4 \eta^{-2} \varepsilon^2 t^4 + \eta^{-1} \varepsilon t^2 \right) \cdot \frac{2 \eta^{-1} f}{r^2} \\
& = \left( \eta^2 + 5 \varepsilon t^2 \eta + 4 \varepsilon^2 t^4 \right) \cdot \frac{2 \eta^{-3} f}{r^2} \\
& = \left( 1 + \varepsilon^2 t^4 - 2 \varepsilon t^2 + 5 \varepsilon t^2 - 5 \varepsilon^2 t^4 + 4 \varepsilon^2 t^4 \right) \cdot \frac{2 \eta^{-3} f}{r^2} \\
& = \left( 1 + 3 \varepsilon t^2 \right) \cdot \frac{2 \eta^{-3} f}{r^2} .
\end{align*}

Now, we look at the \( (E\rho, \bar{E}_\theta) \) component. For this, we consider
\begin{align*}
\nb_{\rho \bar{\theta}} f = \nb_{\rho \theta} f + \frac{r}{2} E_\theta \eta \nb_{\rho\rho}f = \frac{r}{2} E_\theta \eta \frac{2f}{r^2} = E_\theta \eta \cdot \frac{f}{r},
\end{align*}
and that
\begin{align*}
\nb_{\rho \bar{\theta}} \eta = \nb_{\rho \theta} \eta + \frac{r}{2} E_\theta \eta \nb_{\rho\rho} \eta = \nb_{\rho \theta} \eta + \frac{r}{2} E_\theta \eta \left( -\frac{2 \varepsilon t^2}{r^2} \right) = \nb_{\rho \theta} \eta - \frac{ \varepsilon t^2}{r}  E_\theta \eta.
\end{align*}
We further note that \( \bar{E}_\theta\eta = \eta E_\theta \eta \), and calculate as follows
\begin{align*}
\nabla_{\rho\bar{\theta}} \bar{f} & = \eta^{-1} \nabla_{\rho\bar{\theta}} f - \eta^{-2} \nabla_\rho f \nabla_{\bar{\theta}} \eta - \eta^{-2} \nabla_\rho \eta \nabla_{\bar{\theta}} f + 2 \eta^{-3} f \nabla_\rho \eta \nabla_{\bar{\theta}} \eta - \eta^{-2} f \nabla_{\rho\bar{\theta}} \eta \\
& = \eta^{-1} E_\theta \eta \frac{f}{r} - \eta^{-2} \bar{E}_{\theta} \eta \frac{2f}{r} - \eta^{-2} \left( \frac{-2\varepsilon t^2}{r} \right) E_\theta \eta f + 2 \eta^{-3} f \left( \frac{-2 \varepsilon t^2}{r} \right) \bar{E}_\theta \eta \\
& \qquad \qquad - \eta^{-2} f \left( \nabla_{\rho \theta} \eta - \frac{ \varepsilon t^2}{r} \nabla_\theta \eta \right) \\
& = \left( \eta^{-1} E_\theta \eta - 2 \eta^{-1} E_\theta \eta + 2 \eta^{-2} \varepsilon t^2 E_\theta \eta - 4 \eta^{-2} \varepsilon t^2 E_{\theta} \eta - r \eta^{-2} \nabla_{\rho\theta} \eta + \eta^{-2} \varepsilon t^2 E_\theta \eta \right) \cdot \frac{f}{r} \\
& = (-\eta^{-1} E_\theta \eta - \eta^{-2} \varepsilon t^2 E_\theta \eta - \eta^{-2} r \nabla_{\rho \theta} \eta ) \cdot \frac{f}{r} \\
& = \left[ - ( 1 + \varepsilon t^2 \eta^{-1} ) \eta^{-1} E_\theta \eta  - \eta^{-2} r \nabla_{\rho \theta} \eta \right] \cdot \frac{f}{r} \\
& = (-E_\theta \eta - r \nabla_{\rho \theta} \eta) \cdot \frac{f \eta^{-2}}{r}\\
& = (-r E_\theta \eta - r^2 \nabla_{\rho \theta} \eta) \frac{\eta^{-2}}{4} \cdot \frac{\rho^2}{r^2} \text{,}
\end{align*}
where we also used \eqref{eq.pc_eta}. The \( (E_\rho, \bar{E}_A) \) component is computed analogously.
\end{proof}

\begin{proposition} \label{prop_del2_f_barrho}
The following is satisfied
\begin{equation} \label{eq_del2_f_barrho}
\nb_{\bar{\rho}\bar{\rho}} \bar{f} = 2 \eta^{-4} \cdot \frac{\bar{f}}{r^2} [1+ \mc{O}(\varepsilon t^2)].
\end{equation}
\end{proposition}
\begin{proof}
For the first equation, using \eqref{eq_def_barrho} shows that
\begin{align*}
\nb^{\alpha\beta} \bar{f} \nb_\alpha \bar{f} \nb_\beta \bar{f} & = \nb^2 \bar{f} (\nb^\sharp \bar{f}, \nb^\sharp \bar{f}) \\
& = \frac{r^2}{4} \nb^2 \bar{f}(\bar{E}_\rho , \bar{E}_\rho) \\
& = \frac{r^2}{4} \nb_{\bar{\rho}\bar{\rho}} \bar{f}.
\end{align*}
Then, Proposition \ref{thm.misc_gauss} implies that
\begin{align*}
\frac{1}{2} \bar{f} \eta^{-4} [1+ \mc{O}(\varepsilon t^2)] = \frac{r^2}{4} \nb_{\bar{\rho}\bar{\rho}} \bar{f},
\end{align*}
which shows that \eqref{eq_del2_f_barrho} holds.
\end{proof}
Now, we will describe a new tensor which will help us to use the pseudoconvexity condition. Define $\pi$ as follows
\begin{equation}\label{eq_pi}
\pi := \nabla^2 \bar{f} - \bar{h}\cdot g,
\end{equation}
where $\bar{h}$ is given by \eqref{eq.pc}.
Now, we already know the behaviour of $\pi$ in directions that describe pseudoconvexity; see Theorem \ref{thm.pc}. Hence, we will find estimates for those components of $\pi$ which are unrelated to the pseudoconvexity condition.
\begin{proposition} \label{prop_pi_comp}
The components of $\pi$ satisfy the following estimates
\begin{equation} \label{eq_pi_comp}
\pi_{\bar{\rho}\bar{\rho}} = \frac{\rho^2}{4r^2} \mc{O}(\varepsilon r^2), \qquad |\pi_{\bar{\rho}\bar{\theta}}| \lesssim  \frac{\rho^2}{r^2} \mc{O}(\varepsilon r^2), \qquad |\pi_{\bar{\rho}\bar{A}}| \lesssim  \frac{\rho^2}{r^2} \mc{O}(\varepsilon r^2).
\end{equation}
\end{proposition}
\begin{proof}
For the first equation, using Proposition \ref{prop_del2_f_barrho} gives us
\begin{align*}
\pi_{\bar{\rho}\bar{\rho}} & = \nb_{\bar{\rho}\bar{\rho}} \bar{f} - \bar{h}\cdot g_{\bar{\rho}\bar{\rho}}\\
& = 2 \eta^{-4}  \frac{\bar{f}}{r^2} [1+ \mc{O}(\varepsilon t^2)] - \left( \frac{1}{2} \eta^{-1} - \frac{\varepsilon r^2}{4} \right) \frac{4 \eta^{-3} \bar{f}}{r^2} [1 + \mc{O}(\varepsilon t^2)]\\
& = \frac{\rho^2}{4r^2} ( 2 \eta^{-4} - 2 \eta^{-4} + \eta^{-3} \varepsilon \bar{f} ) + \frac{\rho^2}{4 r^2} \varepsilon \bar{f} \mc{O}(\varepsilon t^2)\\
& = \frac{\rho^2}{4r^2} \mc{O}(\varepsilon r^2).
\end{align*}
Hence, the first part of \eqref{eq_pi_comp} is true.
For the second part, note that \( g(\bar{E}_\rho, \bar{E}_\theta) = 0 \). Then recalling the representation of $\bar{E}_\rho$ from \eqref{eq_barrho2}, we get
\begin{align*}
\pi_{\bar{\rho}\bar{\theta}} = \nb_{\bar{\rho}\bar{\theta}} \bar{f} = \eta^{-2} [1 + \mc{O}(\varepsilon^2 t^2 r^2)] \nb_{\rho\bar{\theta}} \bar{f} + \eta^{-2} \frac{r E_\theta \eta}{2} \nb_{\bar{\theta}\bar{\theta}} \bar{f} - \eta^{-2} \frac{\rho^2}{r^2} \sum_A \frac{r E_A \eta}{2} \nb_{\bar{A}\bar{\theta}} \bar{f}.
\end{align*}
Firstly, Proposition \ref{prop_del2_fbar} implies that
\begin{align*}
\nb_{\rho\bar{\theta}} \bar{f} & = (-r E_\theta \eta - r^2 \nabla_{\rho \theta} \eta) \cdot \frac{ \rho^2 }{4r^2} \eta^{-2}.
\end{align*}
Next, we use Proposition \ref{thm.pc_fb_hessian}, to get
\begin{align*}
\nb_{\bar{\theta}\bar{\theta}} \bar{f} & = \eta^{-1} q(E_\theta, E_\theta) + \frac{1}{2} \eta^{-1} g (\bar{E}_\theta, \bar{E}_\theta) - \eta^{-1} \bar{f} \nb^2 \eta (\bar{E}_\theta, \bar{E}_\theta).
\end{align*}
Then, using Proposition \ref{thm.hf_q_est} and \eqref{eq_gcomp}, shows that
\begin{align*}
|\nb_{\bar{\theta}\bar{\theta}} \bar{f}| & \lesssim \eta^{-1} \frac{\rho^4}{r^4} \varepsilon r^2 + \frac{1}{2} \eta^{-1} \frac{\rho^2}{r^2} \left[ 1 +  \left( \frac{r E_\theta \eta}{2} \right)^2 \right] + \frac{\rho^2}{r^2} \varepsilon r^2.
\end{align*}
Analogously, we also have
\begin{align*}
\nb_{\bar{\theta}\bar{A}} \bar{f} & = \eta^{-1} q(E_\theta, E_A) + \frac{1}{2} \eta^{-1} g (\bar{E}_\theta, \bar{E}_A) - \eta^{-1} \bar{f} \nb^2 \eta (\bar{E}_\theta, \bar{E}_A). \\
\Rightarrow |\nb_{\bar{\theta}\bar{A}} \bar{f}| & \lesssim \eta^{-1} \frac{\rho^2}{r^2} \varepsilon r^2 + \frac{1}{2} \eta^{-1} \frac{\rho^2}{r^2} \frac{r E_\theta \eta}{2} \sum_A \frac{r E_A \eta}{2} + \eta^{-2} \frac{\rho^2}{r^2} \varepsilon r^2.
\end{align*}
Thus,
\begin{align*}
|\pi_{\bar{\rho}\bar{\theta}}| \lesssim [1 + \mc{O}(\varepsilon^2 t^2 r^2)] \mc{O}(|\varepsilon t r |) \frac{\rho^2}{r^2} + \mc{O}(|\varepsilon t r |) \mc{O} (\varepsilon r^2) \frac{\rho^2}{r^2} + \mc{O}(|\varepsilon t r |) \mc{O} (\varepsilon r^2) \frac{\rho^2}{r^2} \lesssim \frac{\rho^2}{r^2} \mc{O}(|\varepsilon t r |),
\end{align*}
which shows that the second part of \eqref{eq_pi_comp} is true. Similarly we can also show that the $\pi_{\bar{\rho}\bar{A}}$ estimate holds.
\end{proof}

\subsection{Carleman estimate derivation}
Now, we are ready to prove Theorem \ref{thm_carl_est_LM}. Throughout this section, we assume that the hypothesis of Theorem \ref{thm_carl_est_LM} holds. We have to be careful with the order of operations because here we only have estimates for most of the terms, instead of exact expressions.

\begin{definition} \label{def1_carl_LM}
Define the function $F := F ( \bar{f} )$ as
\begin{equation}
\label{eq.misc_F} F := - a ( \log \bar{f} + b \bar{f} ) \text{.}
\end{equation}
Next, define the following
\begin{equation} \label{eq_carl_def1_LM}
\mc{L} := e^{-F} \square e^F, \qquad S:= \nb^\sharp \bar{f}, \qquad w := \frac{1}{2} \square \bar{f} - \bar{h}, \qquad S_w := S + w,
\end{equation} 
where $\bar{h}$ is defined as in \eqref{eq.pc}.
\end{definition}
Note that, \eqref{eq_thm_carl_def_LM} and \eqref{eq_carl_def1_LM} imply that
\begin{equation} \label{eq_e2F_zeta_LM}
e^{-2F} = \zeta.
\end{equation}

\begin{remark}
Compare \eqref{eq.misc_F} with the definition of the function $F$ from \cite{Arick}. Here we take $\bar{f}$, instead of $\bar{f}^{\frac{1}{2}}$ as in \cite{Arick}, for simplifying our computations. Now, this will lead to a loss of the weight $f^{-\frac{1}{2}}$ in the zero order term in the Carleman estimate \eqref{eq_thm_carl_est_LM}. However, eventually it will not matter because while proving observability $f^{-\frac{1}{2}}$ is estimated from below by a constant.
\end{remark}

Throughout, we use $\prime$ to denote derivatives of $F$ with respect to $\bar{f}$. For the reader's convenience, we provide an outline of the proof here

\begin{itemize}

\item[1.] In Section \ref{ssec_carl_ptws_eq_LM}, we obtain a pointwise equation corresponding to the conjugated operator \(\mc{L}\). 

\item[2.] In Section \ref{ssec_carl_ptws_est_LM}, we obtain a pointwise estimate for the $\mc{L}$-equation. First, we will show that the zero order term is strictly positive. Then, we will estimate the first order term. Since $\bar{E}_\alpha$'s are not mutually orthogonal (see Proposition \ref{prop_gcomp} above), we will not be able to find exact expressions when we raise/lower indices. However, we can still estimate them, as shown in \eqref{eq_Y_barE_LM}, and that will be enough for our purposes. Also, we will use the pseudoconvexity result from Theorem \ref{thm.pc} to obtain positive coefficients for the $\bar{E}_\theta (\cdot)$ and $ \bar{E}_A (\cdot)$ terms .

\item[3.] In Section \ref{ssec_carl_abar_a_LM}, we change the frames used in the latest estimate and go from $\bar{E}_\alpha$ to the original frames $E_\alpha$. For this purpose, we express $\bar{E}_\alpha$'s in terms of $E_\alpha$'s using \eqref{eq.pc_Eb}  and \eqref{eq_barrho1}, and estimate the resulting terms accordingly.

\item[4.] In Section \ref{ssec_carl_conj_rev_LM}, we reverse the conjugation process introduced in \eqref{eq_carl_def1_LM} and obtain a pointwise Carleman estimate corresponding to the $\square$ operator.

\item[5.] In Section \ref{ssec_carl_int_eq_LM}, we integrate the pointwise estimate over the region \(\mc{U} \cap \mc{D} \) and complete the proof of Theorem \ref{thm_carl_est_LM}.

\end{itemize}

Henceforth, we will use the estimate $\eta = 1- \varepsilon t^2 \simeq 1$, whenever necessary.

\subsubsection{Pointwise equation} \label{ssec_carl_ptws_eq_LM}

\begin{proposition} \label{prop_carl_1}
Let $\psi \in C^2(\mc{U})$. Then, at every point of $\mc{U}$, the following is satisfied
\begin{equation} \label{eq_carl_1}
-\mc{L} \psi S_w \psi + \nb^\beta P_\beta = \pi_{\alpha\beta} \nb^\alpha \psi \nb^\beta \psi - 2 F' |S_w \psi|^2 + \left[ \frac{1}{2} \nb^\beta \bar{f} \nb_\beta \mc{A} + \mc{A} \bar{h} - \frac{1}{2} \square w \right] \cdot \psi^2,
\end{equation}
where \( P_\beta \) is defined as
\begin{equation}
P_\beta := S_w \psi \nb_\beta\psi - \frac{1}{2} \nb_\beta \bar{f} \nb^\mu \psi \nb_\mu \psi + \frac{1}{2} (\mc{A} \nb_\beta \bar{f} - \nb_\beta w ) \cdot \psi^2,
\end{equation}
and the quantity $\mc{A}$ is given by
\begin{equation} \label{eq.misc_A}
\mc{A} := [ ( F' )^2 + F'' ] \nabla^\alpha \bar{f} \nabla_\alpha \bar{f} + 2 \bar{h} F' \text{.}
\end{equation}
\end{proposition}
\begin{proof}

We start with the following computation
\begin{align*}
\mc{L} \psi & = e^{-F} \square (e^{F} \psi) \\
& = \square \psi + 2 F' S_w \psi + (F''+(F')^2) \nb^\beta \bar{f} \nb_\beta \bar{f} \cdot \psi + 2 \bar{h} F' \cdot \psi \\
& = \square \psi + 2 F' S_w \psi + \mc{A} \psi.
\end{align*}
This implies that,
\begin{equation} \label{eq_carl_1_1}
- \mc{L} \psi S_w \psi = - \square \psi S_w \psi - 2 F' |S_w \psi|^2 - \mc{A} \psi S_w \psi.
\end{equation}

Now, we have
\begin{align*}
\frac{1}{2} \nb^\beta ( \mc{A} \nb_\beta \bar{f} \cdot \psi^2 ) & = \frac{1}{2} \nb^\beta \mc{A} \nb_\beta \bar{f} \cdot \psi^2 + \frac{1}{2} \mc{A} \square \bar{f} \cdot \psi^2 + \mc{A} \nb^\beta \bar{f} \nb_\beta \psi \cdot \psi \numberthis \label{eq_carl_1_2} \\
& = \frac{1}{2} \nb^\beta \mc{A} \nb_\beta \bar{f} \cdot \psi^2 + \mc{A} (\bar{h} + w) \psi^2 + \mc{A} \psi S \psi \\
& = \frac{1}{2} \nb^\beta \mc{A} \nb_\beta \bar{f} \cdot \psi^2 + \mc{A} \bar{h} \psi^2 + \mc{A} \cdot \psi S_w \psi \text{.}
\end{align*}
Next, computing the derivative of $P$, shows that
\begin{align*}
\nb^\beta P_\beta & = \square \psi S_w \psi +\nb_{\alpha \beta} \bar{f} \cdot \nb^\alpha \psi \nb^\beta \psi - \frac{1}{2} \square f \cdot \nb^\mu \psi \nb_\mu \psi + w \cdot \nb^\beta \psi \nb_\beta \psi \numberthis \label{eq_carl_1_3}\\
& \qquad + \frac{1}{2} \nb^\beta \mc{A} \nb_\beta \bar{f} \cdot \psi^2 + \mc{A} \bar{h} \psi^2 + \mc{A} \cdot \psi S_w \psi - \frac{1}{2} \square w \cdot \psi^2 .
\end{align*}

Combining \eqref{eq_carl_1_1}, \eqref{eq_carl_1_2}, and \eqref{eq_carl_1_3}, shows that
\begin{align*}
-\mc{L} \psi S_w \psi + \nb^\beta P_\beta & = (\nb_{\alpha\beta} \bar{f} - \bar{h} \cdot g_{\alpha\beta}) \nb^\alpha \psi \nb^\beta \psi - 2 F' |S_w \psi|^2 \\
& \qquad + \left[ \frac{1}{2} \nb^\beta \bar{f} \nb_\beta \mc{A}  + \mc{A} \bar{h} - \frac{1}{2} \square w \right] \cdot \psi^2 .
\end{align*}
Now, using \eqref{eq_pi} shows that \eqref{eq_carl_1} holds.
\end{proof}

\subsubsection{Pointwise estimate} \label{ssec_carl_ptws_est_LM}
Now, we will obtain a pointwise estimate by estimating some of the non-essential terms in \eqref{eq_carl_1}. First, we estimate the zero-order coefficient. In particular, we establish that this coefficient is indeed positive. Let us use $\mc{B}$ to denote the following
\begin{equation}
\label{eq.misc_B} \mc{B} := \frac{1}{2} \nabla^\alpha \bar{f} \nabla_\alpha \mc{A} + \bar{h} \mc{A} \text{.}
\end{equation}


\begin{proposition} \label{thm.misc_zero}
If $a$ and $b$ in \eqref{eq.misc_F} are chosen to satisfy 
\eqref{eq_thm_carl_a_LM}, then the following inequality holds for $\mc{B}$:
\begin{equation}
\label{eq.misc_zero} \mc{B} \geq \frac{1}{2} a^2 b + \frac{1}{4} a^2 \bar{f}^{-1} \cdot \varepsilon r^2 \text{.}
\end{equation}
\end{proposition}

\begin{proof}
Since the hypothesis of Theorem \ref{thm_carl_est_LM} is true, we get that there exist universal constants $\mc{C}_\ast, \mc{C}_\dagger > 0$ such that Assumption \ref{ass.hf_curv} and equation \eqref{eq.pc_ass} holds.

Now, note that direct computations yield
\begin{equation}
\label{eql.misc_zero_0} F' = - a \bar{f}^{-1} - a b \text{,} \qquad F'' = a \bar{f}^{-2} \text{,} \qquad F''' = - 2 a \bar{f}^{-3} \text{.}
\end{equation}
We begin by applying \eqref{eq.misc_B} and \eqref{eq.misc_A} to expand
\begin{align}
\label{eql.misc_zero_1} \mc{B} &= [ ( F' )^2 + F'' ] \nabla_{ \alpha \beta } \bar{f} \nabla^\alpha \bar{f} \nabla^\beta \bar{f} + \frac{1}{2} ( 2 F' F'' + F''' ) \nabla^\alpha \bar{f} \nabla_\alpha \bar{f} \nabla^\beta \bar{f} \nabla_\beta \bar{f} \\
\notag &\qquad + \bar{h} F'' \nabla^\alpha \bar{f} \nabla_\alpha \bar{f} + F' \nabla^\alpha \bar{f} \nabla_\alpha \bar{h} + \bar{h} \mc{A} \\
\notag &= \mc{B}_1 + \mc{B}_2 + \mc{B}_3 + \mc{B}_4 + \mc{B}_5 \text{.}
\end{align}
For $\mc{B}_1$, we expand using \eqref{eq.hf_f}, Proposition \ref{thm.misc_gauss}, and \eqref{eq.misc_F}:
\begin{align}
\label{eql.misc_zero_11} \mc{B}_1 &= \frac{1}{2} [ a^2 ( \bar{f}^{-1} + 2 b + b^2 \bar{f} ) + a \bar{f}^{-1} ] \eta^{-4} [ 1 + 5 \varepsilon t^2 + 2 \varepsilon^2 t^4 + \mc{O} ( \varepsilon^2 t^2 f ) ] \\
\notag &= a^2 \bar{f}^{-1} \eta^{-4} \left( \frac{1}{2} + \frac{5}{2} \varepsilon t^2 + \varepsilon^2 t^4 \right) + a^2 b \eta^{-4} + \frac{1}{2} a \bar{f}^{-1} \eta^{-4} + \mc{O} ( a^2 \varepsilon \eta^{-4} ) \\
\notag &\qquad + \mc{O} ( a^2 b \eta^{-4} \cdot \varepsilon t^2 ) + \mc{O} ( a^2 b^2 \bar{f} \eta^{-4} ) + \mc{O} ( a \bar{f}^{-1} \eta^{-4} \cdot \varepsilon t^2 ) \\
\notag &= a^2 \bar{f}^{-1} \eta^{-4} \left( \frac{1}{2} + \frac{5}{2} \varepsilon r^2 \right) + a^2 b \eta^{-4} + \frac{1}{2} a \bar{f}^{-1} \eta^{-4} \\
\notag &\qquad + \left( \frac{ \varepsilon_0 }{ b_0 } + b_0 \right) \mc{O} ( a^2 b \eta^{-4} ) + \left( \frac{1}{a} + \varepsilon_0 \right) \mc{O} ( a^2 \bar{f}^{-1} \eta^{-4} \cdot \varepsilon r^2 ) \text{.}
\end{align}
A similar computation for $\mc{B}_2$ then yields
\begin{align}
\label{eql.misc_zero_12} \mc{B}_2 &= - [ a^2 ( \bar{f}^{-1} + b ) + a \bar{f}^{-1} ] \eta^{-4} [ 1 + \varepsilon t^2 + \mc{O} ( \varepsilon^2 t^2 f ) ]^2 \\
\notag &= - a^2 \bar{f}^{-1} \eta^{-4} ( 1 + 2 \varepsilon t^2 + \varepsilon^2 t^4 ) - a^2 b \eta^{-4} - a \bar{f}^{-1} \eta^{-4} \\
\notag &\qquad + \mc{O} ( a^2 \varepsilon \eta^{-4} ) + \mc{O} ( a^2 b \eta^{-4} \cdot \varepsilon t^2 ) + \mc{O} ( a \bar{f}^{-1} \eta^{-4} \cdot \varepsilon t^2 ) \\
\notag &= - a^2 \bar{f}^{-1} \eta^{-4} ( 1 + 2 \varepsilon r^2 ) - a^2 b \eta^{-4} - a \bar{f}^{-1} \eta^{-4} \\
\notag &\qquad + \left( \frac{ \varepsilon_0 }{ b_0 } + \varepsilon_0 \right) \mc{O} ( a^2 b \eta^{-4} ) + \left( \frac{1}{a} + \varepsilon_0 \right) \mc{O} ( a^2 \bar{f}^{-1} \eta^{-4} \cdot \varepsilon r^2 ) \text{.}
\end{align}

Next, we turn to $\mc{B}_5$.
Using the same observations as before and \eqref{eq.pc}, we have
\begin{align*}
\mc{A} &= [ a^2 ( \bar{f}^{-1} + 2 b + b^2 \bar{f} ) + a \bar{f}^{-1} ] \eta^{-2} [ 1 + \varepsilon t^2 + \mc{O} ( \varepsilon t^2 f ) ] - a ( \bar{f}^{-1} + b ) \left( \eta^{-1} - \frac{1}{2} \varepsilon r^2 \right) \\
&= a^2 \bar{f}^{-1} \eta^{-2} \cdot ( 1 + \varepsilon t^2 ) + 2 a^2 b \eta^{-2} + a \bar{f}^{-1} \eta^{-2} - a \bar{f}^{-1} \eta^{-1} \\
&\qquad + \left( \frac{ \varepsilon_0 }{ b_0 } + b_0 \right) \mc{O} ( a^2 b \eta^{-4} ) + \frac{1}{a} \, \mc{O} ( a^2 \bar{f}^{-1} \eta^{-4} \cdot \varepsilon r^2 ) + \frac{1}{a} \cdot \mc{O} ( a^2 b \eta^{-4} ) \\
&= a^2 \bar{f}^{-1} \eta^{-2} ( 1 + \varepsilon t^2 ) + 2 a^2 b \eta^{-2} + \left( \frac{1}{a} + \frac{ \varepsilon_0 }{ b_0 } + b_0 \right) \mc{O} ( a^2 b \eta^{-4} ) \\
&\qquad + \frac{1}{a} \, \mc{O} ( a^2 \bar{f}^{-1} \eta^{-4} \cdot \varepsilon r^2 ) \text{.}
\end{align*}
Combining the above with \eqref{eq.pc} then yields
\begin{align}
\label{eql.misc_zero_13} \mc{B}_5 &= a^2 \bar{f}^{-1} \eta^{-4} \cdot \frac{1}{2} \eta ( 1 + \varepsilon t^2 ) + a^2 b \eta^{-4} - a^2 \bar{f}^{-1} \eta^{-4} \cdot \frac{1}{4} \varepsilon r^2 \eta^2 ( 1 + \varepsilon t^2 ) \\
\notag &\qquad + \left( \frac{1}{a} + \frac{ \varepsilon_0 }{ b_0 } + b_0 \right) \mc{O} ( a^2 b \eta^{-4} ) + \frac{1}{a} \, \mc{O} ( a^2 \bar{f}^{-1} \eta^{-4} \cdot \varepsilon r^2 ) \\
\notag &= a^2 \bar{f}^{-1} \eta^{-4} \left( \frac{1}{2} - \frac{1}{4} \varepsilon r^2 \right) + a^2 b \eta^{-4} + \left( \frac{1}{a} + \frac{ \varepsilon_0 }{ b_0 } + b_0 \right) \mc{O} ( a^2 b \eta^{-4} ) \\
\notag &\qquad + \left( \frac{1}{a} + \varepsilon_0 \right) \mc{O} ( a^2 \bar{f}^{-1} \eta^{-4} \cdot \varepsilon r^2 ) \text{.}
\end{align}
Furthermore, summing \eqref{eql.misc_zero_11}--\eqref{eql.misc_zero_13} yields
\begin{align}
\label{eql.misc_zero_20} \mc{B}_1 + \mc{B}_2 + \mc{B}_5 &= \frac{1}{4} a^2 \bar{f}^{-1} \cdot \varepsilon r^2 + a^2 b \eta^{-4} - \frac{1}{2} a \bar{f}^{-1} \eta^{-4} \\
\notag & \qquad + \left( \frac{1}{a} + \frac{ \varepsilon_0 }{ b_0 } + b_0 \right) \mc{O} ( a^2 b \eta^{-4} ) + \left( \frac{1}{a} + \varepsilon_0 \right) \mc{O} ( a^2 \bar{f}^{-1} \eta^{-4} \cdot \varepsilon r^2 ) \text{.}
\end{align}

For $\mc{B}_3$, we use the same tools as before to obtain
\begin{align}
\label{eql.misc_zero_21} \mc{B}_3 &= a \bar{f}^{-1} \eta^{-2} \left( \frac{1}{2} \eta^{-1} - \frac{1}{4} \varepsilon r^2 \right) [ 1 + \varepsilon t^2 + \mc{O} ( \varepsilon t^2 f ) ] \\
\notag &= \frac{1}{2} a \bar{f}^{-1} \eta^{-4} + \frac{1}{a} \, \mc{O} ( a^2 \bar{f}^{-1} \eta^{-4} \cdot \varepsilon r^2 ) \text{.}
\end{align}
For $\mc{B}_4$, a direct computation using \eqref{eq.pc_fb} and \eqref{eq.pc} yields
\begin{align*}
\nabla^\alpha \bar{f} \nabla_\alpha \bar{h} &= - \frac{1}{2} \eta^{-3} \nabla^\alpha f \nabla_\alpha \eta - \frac{1}{4} \eta^{-1} \nabla^\alpha f \nabla_\alpha r^2 + \frac{1}{2} \eta^{-4} f \nabla^\alpha \eta \nabla_\alpha \eta + \frac{1}{4} \varepsilon \eta^{-2} f \nabla^\alpha \eta \nabla_\alpha r^2 \\
&= - \frac{1}{2} \eta^{-3} \nabla^\alpha f \nabla_\alpha \eta - \frac{1}{4} \eta^{-1} \nabla^\alpha f \nabla_\alpha r^2 + \frac{1}{2} \eta^{-4} f \nabla^\alpha \eta \nabla_\alpha \eta \\
&\qquad + \varepsilon \eta^{-2} f \nabla^\alpha \eta \nabla_\alpha f + \frac{1}{4} \varepsilon \eta^{-2} f \nabla^\alpha \eta \nabla_\alpha t^2 \text{.}
\end{align*}
Applying \eqref{eq.hf_gauss}, \eqref{eq.pc_eta}, and \eqref{eq.pc_eta_est} to the above then results in the estimate
\[
\nabla^\alpha \bar{f} \nabla_\alpha \bar{h} = \mc{O} ( \varepsilon r^2 ) \text{,}
\]
and the above then yields that
\begin{equation}
\label{eql.misc_zero_22} \mc{B}_4 = a \bar{f}^{-1} \cdot \mc{O} ( \varepsilon r^2 ) = \frac{1}{a} \, \mc{O} ( a^2 \bar{f}^{-1} \eta^{-4} \cdot \varepsilon r^2 ) \text{.}
\end{equation}

Finally, combining \eqref{eql.misc_zero_1} and \eqref{eql.misc_zero_20}--\eqref{eql.misc_zero_22} yields
\begin{align*}
\mc{B} &= \frac{1}{4} a^2 \bar{f}^{-1} \cdot \varepsilon r^2 + a^2 b \eta^{-4} + \left( \frac{1}{a} + \frac{ \varepsilon_0 }{ b_0 } + b_0 \right) \mc{O} ( a^2 b \eta^{-4} ) \\
\notag &\qquad + \left( \frac{1}{a} + \varepsilon_0 \right) \mc{O} ( a^2 \bar{f}^{-1} \eta^{-4} \cdot \varepsilon r^2 ) \text{.}
\end{align*}
Taking the above and recalling \eqref{eq_thm_carl_a_LM} yields the desired \eqref{eq.misc_zero}.
\end{proof}

Now, there is one additional zero-order term in \eqref{eq_carl_1}, with the coefficient
\[
\mc{B}_\ast := - \frac{1}{2} \Box \bar{w} := - \frac{1}{2} \Box \left( \frac{1}{2} \Box \bar{f} - \bar{h} \right) \text{.}
\]
However, since both $\bar{f}$ and $\bar{h}$ are smooth on $\bar{\mc{D}}_{ r_0 }$, it follows that $\mc{B}_\ast$ must be bounded on $\mc{D}_{ r_0 }$.
Thus, taking $a$ to be sufficiently large, we can ensure $\mc{B}_\ast$ is dominated by $\mc{B}$. 

Next, using Propositions \ref{thm.hf_q_est} and \ref{thm.pc_eta_est}, shows that $w$ given by \eqref{eq_carl_def1_LM} satisfies
\[ w = c_n + \mc{O}(\varepsilon r^2), \]
for some constant $c_n > 0 $ dependent\footnote{One can actually show that $c_n = \frac{n-1}{4}$.} on $n$. Now, let us define the following vector field
\begin{equation} \label{eq_def_hatE_LM}
\hat{E}_\rho := \bar{f}^{- c_n} \bar{E}_\rho \bar{f}^{c_n}.
\end{equation}

Then, we have the following pointwise Carleman estimate.
\begin{proposition} \label{prop_carl_2}
Let $\psi \in C^2(\mc{U})$. Then, at every point of $\mc{U}$, the following estimate is satisfied
\begin{align} \label{eq_carl_2}
\frac{1}{4a} \bar{f} |\mc{L} \psi|^2 + \nb^\beta P_\beta & \geqslant \frac{r^2}{\rho^2} (\hat{E}_\rho \psi)^2 + \frac{1}{16} \varepsilon r^2 \cdot \frac{r^2}{\rho^2} (\bar{E}_\theta \psi)^2 + \frac{1}{16} \varepsilon r^2 \cdot \sum_A (\bar{E}_A \psi)^2 \\
\notag & \qquad + \frac{1}{4} a^2 b \cdot \psi^2 + \frac{1}{8} a^2 \bar{f}^{-1} \varepsilon r^2 \cdot \psi^2.
\end{align}
\end{proposition}

\begin{proof}

Since $-F'>0$, we have
\begin{align*}
|\mc{L} \psi| |S_w \psi| \leqslant -\frac{1}{4F'} |\mc{L} \psi|^2 - F' |S_w \psi|^2.
\end{align*}
Furthermore, also note that
\begin{align*}
- \frac{1}{4F'} \leqslant \frac{\bar{f}}{4 a (1+ b \bar{f})} \leqslant \frac{\bar{f}}{4a}.
\end{align*}
Using the above and Proposition \ref{thm.misc_zero}, shows that \eqref{eq_carl_1} reduces to
\begin{equation} \label{eq_carl_2_1}
\frac{1}{4a} \bar{f} |\mc{L} \psi|^2 + \nb^\beta P_\beta \geqslant \pi_{\alpha\beta} \nb^\alpha \psi \nb^\beta \psi - F' |S_w \psi|^2 + \frac{1}{4} a^2 b \cdot \psi^2  + \frac{1}{4} a^2 \bar{f}^{-1} \varepsilon r^2 \cdot \psi^2.
\end{equation}
Now we will estimate the term $\pi_{\alpha\beta} \nb^\alpha \psi \nb^\beta \psi$. For this purpose let us define the following quantity
\begin{equation}
\nb^\alpha \psi := Y^\alpha.
\end{equation}
This implies that
\begin{align*}
\nb^\sharp \psi = Y^\rho \bar{E}_\rho + Y^a \bar{E}_a.
\end{align*}
Hence,
\begin{align*}
\pi_{\alpha\beta} \nb^\alpha \psi \nb^\beta \psi & = \pi_{\alpha\beta} (Y^\rho \bar{E}_\rho^\alpha + X^\alpha, Y^\rho \bar{E}_\rho^\beta + X^\beta) \numberthis \label{eq_pi_est}\\
& =  \pi_{\bar{\rho}\bar{\rho}} Y^\rho Y^\rho + 2 \pi_{\bar{\rho}\bar{\theta}} Y^\rho Y^\theta + 2 \sum_A \pi_{\bar{\rho}\bar{A}} Y^\rho Y^A + \pi(\bar{X}, \bar{X}),
\end{align*}
where $\bar{X}:= Y^a \bar{E}_a$. Note that, due to the pseudoconvexity result Theorem \ref{thm.pc}, we get
\begin{equation} \label{eq_pi_X}
\pi(\bar{X}, \bar{X}) \geqslant \frac{ \varepsilon_0 }{8} \frac{ r^2 }{ r_0^2 } \cdot \bar{g}_+ ( \bar{X}, \bar{X} ).
\end{equation}
Now, we will look at the behaviour of the remaining terms in \eqref{eq_pi_est}. Note that, the $\bar{E}_\alpha$ frames are not mutually orthogonal. Thus, lowering the indices on $Y^\alpha$ is tricky. Although exact computation is difficult, we can nevertheless estimate them.

\textbf{I. Changing $Y^\alpha$ to $\bar{E}_\alpha$.}
Note that, we have the following mapping of $Y^\alpha$ and $\bar{E}_\alpha$.
\begin{align*}
\bar{E}_\rho \psi & = \frac{\rho^2}{r^2} \eta^{-3} [1+ \mc{O} (\varepsilon t^2)] \cdot Y^\rho, \numberthis \label{eq_barE_Y_LM} \\
\bar{E}_\theta \psi & = g(\bar{E}_\theta , \bar{E}_\theta ) Y^\theta + \sum_A g(\bar{E}_\theta, \bar{E}_A) Y^A, \\
\bar{E}_A \psi & = g(\bar{E}_\theta, \bar{E}_A) Y^\theta + \sum_B g(\bar{E}_A, \bar{E}_B) Y^B.
\end{align*}

\begin{remark} \label{rmk_A+B_inv}
To find the inverse map, we will use the following Taylor expansion
\[ (A+B)^{-1} = A^{-1} - A^{-1} B A^{-1} + A^{-1} B A^{-1} B A^{-1} + \cdots,\]
where $A,B$ are matrices satisfying $||A^{-1}B|| < 1$. In our case, using Proposition \ref{prop_gcomp} and letting $g(\bar{E},\bar{E}) := g (E,E) + B_\varepsilon$, shows that we can use the above expansion with $A = g(E,E)$ and $B = B_\varepsilon$ to find the inverse map of \eqref{eq_barE_Y_LM}. Note that, this is possible because the error terms that form $B_\varepsilon$ are very small compared to the corresponding leading terms in $g(E,E)$. 
\end{remark}
Then, estimating the inverse map by using Propositions \ref{prop_gcomp}, \ref{thm.pc_eta_est}, and equations \eqref{eq_thm_carl_a_LM}, \eqref{eq_thm_carl_cc0_LM}, gives us
\begin{align*}
Y^\rho & = [1+ \mc{O} (\varepsilon t^2) ] \frac{r^2}{\rho^2} \bar{E}_\rho \psi, \numberthis \label{eq_Y_barE_LM} \\
Y^\theta & = -(1 + \mc{O} (\varepsilon^2 r^2 t^2 ))\frac{r^2}{\rho^2} \bar{E}_\theta \psi + \mc{O} ({\varepsilon^2 r^2 t^2} ) \sum_A \bar{E}_A \psi, \\
\sum_A Y^A & = (1 + \mc{O} (\varepsilon^2 r^2 t^2 )) \sum_A \bar{E}_A \psi + \mc{O} ({\varepsilon^2 r^2 t^2} ) \bar{E}_\theta \psi.
\end{align*}
This further implies the following estimates
\begin{align*}
(Y^\rho)^2 & = [1+ \mc{O} (\varepsilon t^2) ] \frac{r^4}{\rho^4} (\bar{E}_\rho \psi)^2, \numberthis \label{eq_Y2_list} \\
(Y^\theta)^2 & \gtrsim \frac{r^4}{\rho^4} (\bar{E}_\theta \psi)^2 + \mc{O} ({\varepsilon^4 r^4 t^4} ) \sum_A (\bar{E}_A \psi)^2, \\
\sum_A (Y^A)^2 & \gtrsim \sum_A (\bar{E}_A \psi)^2 + \mc{O} ({\varepsilon^4 r^4 t^4} ) (\bar{E}_\theta \psi)^2, \\
|Y^\rho Y^\theta| & \leqslant K \cdot \frac{r^4}{\rho^4} (\bar{E}_\rho \psi)^2 + \frac{1}{4K} \cdot \frac{r^4}{\rho^4} (\bar{E}_\theta \psi)^2 + \mc{O} (\varepsilon^4 r^4 t^4) \sum_A (\bar{E}_A \psi)^2, \\
|Y^\rho \sum_A Y^A| & \leqslant K \cdot \frac{r^4}{\rho^4} (\bar{E}_\rho \psi)^2 + \frac{1}{4K}\sum_A (\bar{E}_A \psi)^2 + \mc{O} (\varepsilon^4 r^4 t^4) (\bar{E}_\theta \psi)^2.,
\end{align*}
where $K>0$ is some constant to be fixed later. 
Then, using Proposition \ref{prop_pi_comp} and equation \eqref{eq_Y2_list}, shows that
\begin{align*}
(a) \pi_{\bar{\rho}\bar{\rho}} Y^\rho Y^\rho & = \frac{\rho^2}{4r^2} \mc{O}(\varepsilon r^2) \cdot [1+ \mc{O} (\varepsilon t^2) ] \frac{r^4}{\rho^4} (\bar{E}_\rho \psi)^2 \\
& = \frac{r^2}{\rho^2} (\bar{E}_\rho \psi)^2 \mc{O}(\varepsilon r^2),\\
(b) \pi_{\bar{\rho}\bar{\theta}} Y^\rho Y^\theta & = \frac{\rho^2}{r^2} \mc{O}(\varepsilon r^2) \cdot Y^\rho Y^\theta \\
& \geqslant - \frac{\rho^2}{r^2} \mc{O}(\varepsilon r^2) \left[ K \frac{r^4}{\rho^4} (\bar{E}_\rho \psi)^2 + \frac{1}{4K} \frac{r^4}{\rho^4} (\bar{E}_\theta \psi)^2 + \mc{O} (\varepsilon^4 t^4 r^4) \sum_A (\bar{E}_A \psi)^2 \right] \\
& \geqslant - \frac{r^2}{\rho^2}(\bar{E}_\rho \psi)^2 \cdot K \cdot \mc{O}(\varepsilon r^2) - \frac{1}{64} \varepsilon r^2 \cdot \frac{r^2}{\rho^2}(\bar{E}_\theta \psi)^2 - \frac{\rho^2}{r^2} \sum_A (\bar{E}_A \psi)^2 \mc{O} (\varepsilon^4 t^4 r^4),
\end{align*}
where the constant $K$ is chosen such that it offsets the constant from the $\mc{O}(\varepsilon r^2)$ coefficient of the $(\bar{E}_\theta \psi)^2$ term and gives us the above expression. Analogously, we also get the following
\begin{align*}
(c) \sum_A \pi_{\bar{\rho}\bar{A}} Y^\rho Y^A \geqslant - \frac{r^2}{\rho^2}(\bar{E}_\rho \psi)^2 \cdot K \cdot \mc{O}(\varepsilon r^2) - \frac{r^2}{\rho^2} (\bar{E}_\theta \psi)^2 \mc{O} (\varepsilon^4 t^4 r^4) - \frac{1}{64} \varepsilon r^2 \cdot \sum_A (\bar{E}_A \psi)^2.
\end{align*}
Combining the three equations $(a),(b),(c)$ above and using \eqref{eq_pi_X}, shows that \eqref{eq_pi_est} reduces to
\begin{align*}
\pi_{\alpha\beta} \nb^\alpha\psi \nb^\beta \psi & \geqslant \pi_{\bar{\rho}\bar{\rho}} Y^\rho Y^\rho + 2 \pi_{\bar{\rho}\bar{\theta}} Y^\rho Y^\theta + 2 \sum_A \pi_{\bar{\rho}\bar{A}} Y^\rho Y^A + \frac{ \varepsilon_0 }{8} \frac{ r^2 }{ r_0^2 } \cdot \bar{g}_+ ( \bar{X}, \bar{X} ) \numberthis \label{eq_pi_final}\\
& \geqslant \frac{r^2}{\rho^2} (\bar{E}_\rho \psi)^2 \cdot K \cdot \mc{O}(\varepsilon r^2) - \frac{1}{32} \varepsilon r^2 \cdot \frac{r^2}{\rho^2}(\bar{E}_\theta \psi)^2 - \frac{1}{32} \varepsilon r^2 \cdot \sum_A (\bar{E}_A \psi)^2 \\
& \qquad + \frac{1}{8} \varepsilon r^2 \left[ \frac{r^2}{\rho^2} (\bar{E}_\theta \psi)^2 + \sum_A (\bar{E}_A \psi)^2 \right]\\
& \geqslant \frac{r^2}{\rho^2} (\bar{E}_\rho \psi)^2 \cdot K \cdot \mc{O}(\varepsilon r^2) + \frac{1}{16} \varepsilon r^2 \cdot \frac{r^2}{\rho^2} (\bar{E}_\theta \psi)^2 + \frac{1}{16} \varepsilon r^2 \cdot \sum_A (\bar{E}_A \psi)^2,
\end{align*}
for $\varepsilon_0$ sufficiently small.

\textbf{II. Estimating $- F' |S_w \psi|^2$.} Now, we will consider the term $ - F' |S_w \psi|^2 $. Recall the following definition from \eqref{eq_def_hatE_LM}:
\begin{equation*}
\hat{E}_\rho := \bar{f}^{- c_n} \bar{E}_\rho \bar{f}^{c_n}.
\end{equation*}
Using the above and \eqref{eq_del_rbrf}, we get
\begin{equation} \label{eq_Ehat_1_LM}
\hat{E}_\rho \psi = \bar{E}_\rho \psi + c_n \cdot \frac{2}{r} \eta^{-3} [1 + \mc{O}(\varepsilon r^2)] \cdot \psi.
\end{equation}
Now, \eqref{eq_def_barrho} and  \eqref{eq_carl_def1_LM} imply that
\begin{align*}
S \psi = \frac{r}{2} \bar{E}_\rho \psi.
\end{align*}
Then using \eqref{eq_Ehat_1_LM}, show that
\begin{align*}
S_w \psi & = S\psi + w\psi\\
& = \frac{r}{2} \hat{E}_\rho \psi - c_n \cdot \eta^{-3} [1 + \mc{O}(\varepsilon r^2)] \cdot \psi + c_n \psi^2 + \mc{O}(\varepsilon r^2) \cdot \psi^2 \\
& = \frac{r}{2} \hat{E}_\rho \psi + \mc{O}(\varepsilon r^2) \cdot \psi^2.
\end{align*}
This implies that
\begin{align*}
- F' |S_w \psi|^2 & \geqslant a \bar{f}^{-1} (1 + b \bar{f}) \left[ \frac{r^2}{4}(\hat{E}_\rho \psi)^2 - c_n^2 \mc{O}(\varepsilon^2 r^4) \cdot \psi^2 - r |\hat{E}_\rho \psi| |\psi| \mc{O}(\varepsilon r^2) \right] \numberthis \label{eq_Erho_a} \\
& \geqslant \frac{a}{8} \cdot \frac{r^2}{\rho^2}(\hat{E}_\rho \psi)^2 - 2 c_n^2  \bar{f}^{-1} \mc{O}(\varepsilon^2 r^4) \cdot \psi^2 .
\end{align*}
Next, we will estimate the $\bar{E}_\rho \psi$ term in the RHS of \eqref{eq_pi_final} by $\hat{E}_\rho \psi$. Note that, \eqref{eq_Ehat_1_LM} offers us the following estimate
\begin{equation} \label{eq_Ebar_Ehat_K}
K \cdot \mc{O}(\varepsilon r^2) \frac{r^2}{\rho^2} (\bar{E}_\rho \psi)^2 \geqslant K \cdot \mc{O}(\varepsilon r^2) \frac{r^2}{\rho^2} (\hat{E}_\rho \psi)^2 - K \cdot \mc{O}(\varepsilon r^2) c_n^2 \bar{f}^{-1} \cdot \psi^2.
\end{equation}
Then, combining \eqref{eq_pi_final}, \eqref{eq_Erho_a}, and \eqref{eq_Ebar_Ehat_K} gives us
\begin{align*}
& \pi_{\alpha\beta} \nb^\alpha\psi \nb^\beta \psi - F' |S_w \psi|^2 + \frac{1}{4} a^2 b \cdot \psi^2  + \frac{1}{4} a^2 \bar{f}^{-1} \cdot \varepsilon r^2\\
& \geqslant K \cdot \mc{O}(\varepsilon r^2) \frac{r^2}{\rho^2} (\hat{E}_\rho \psi)^2 + \frac{1}{16} \varepsilon r^2 \cdot \frac{r^2}{\rho^2} (\bar{E}_\theta \psi)^2 + \frac{1}{16} \varepsilon r^2 \cdot \sum_A (\bar{E}_A \psi)^2 + \frac{a}{8} \cdot \frac{r^2}{\rho^2}(\hat{E}_\rho \psi)^2 \\
& \qquad - c_n^2 \bar{f}^{-1} [ K + 2 \mc{O}(\varepsilon r^2) ] \mc{O}(\varepsilon r^2) \cdot \psi^2  + \frac{1}{4} a^2 b \cdot \psi^2  + \frac{1}{4} a^2 \bar{f}^{-1} \cdot \varepsilon r^2 \\
& \geqslant \frac{a}{16} \cdot \frac{r^2}{\rho^2}(\hat{E}_\rho \psi)^2 + \frac{1}{16} \varepsilon r^2 \cdot \frac{r^2}{\rho^2} (\bar{E}_\theta \psi)^2 + \frac{1}{16} \varepsilon r^2 \cdot \sum_A (\bar{E}_A \psi)^2 + \frac{1}{4} a^2 b \cdot \psi^2  + \frac{1}{8} a^2 \bar{f}^{-1} \varepsilon r^2 \cdot \psi^2,
\end{align*}
for sufficiently large $a$. Now, we can drop the coefficient $a$ from the $\hat{E}_\rho \psi$ term, since we absorbed any possible negative coefficients using $a$ and it is no longer needed. Thus, we have the following estimate
\begin{align*}
& \pi_{\alpha\beta} \nb^\alpha\psi \nb^\beta \psi - F' |S_w \psi|^2 + \frac{1}{4} a^2 b \cdot \psi^2  + \frac{1}{4} a^2 \bar{f}^{-1} \cdot \varepsilon r^2 \numberthis \label{eq_carl_2_3}\\
& \geqslant \frac{r^2}{\rho^2}(\hat{E}_\rho \psi)^2 + \frac{1}{16} \varepsilon r^2 \cdot \frac{r^2}{\rho^2} (\bar{E}_\theta \psi)^2 + \frac{1}{16} \varepsilon r^2 \cdot \sum_A (\bar{E}_A \psi)^2 + \frac{1}{4} a^2 b \cdot \psi^2  + \frac{1}{8} a^2 \bar{f}^{-1} \varepsilon r^2 \cdot \psi^2. 
\end{align*}
Substituting \eqref{eq_carl_2_3} in \eqref{eq_carl_2_1}, shows that \eqref{eq_carl_2} holds. This completes the proof of Proposition \ref{prop_carl_2}.
\end{proof}

\subsubsection{Changing frames back to $E_\alpha$.} \label{ssec_carl_abar_a_LM}
Now we will change the frames used in Proposition \ref{prop_carl_2} from $ (\hat{E}_\rho, \bar{E}_\theta, \bar{E}_A ) $ to $ (E_\rho, E_\theta, E_A ) $, since these were the frames we started with.

\begin{proposition} \label{prop_carl_CF}
Let $\psi \in C^2(\mc{U})$. Then, at every point of $\mc{U}$, the following estimate is satisfied
\begin{align} \label{eq_carl_CF}
\frac{1}{4a} \bar{f} |\mc{L} \psi|^2 + \nb^\beta P_\beta & \gtrsim \frac{1}{32} \varepsilon r^2 \frac{r^2}{\rho^2} (E_\rho \psi)^2 + \frac{1}{32} \varepsilon r^2 \frac{r^2}{\rho^2} (E_\theta \psi)^2 + \frac{1}{32} \varepsilon r^2 \sum_A (E_A \psi)^2 \\
\notag & \qquad + \frac{1}{4} a^2 b \cdot \psi^2 + \frac{1}{16} a^2 \bar{f}^{-1} \varepsilon r^2 \cdot \psi^2.
\end{align}
\end{proposition}

\begin{proof}
It is enough to estimate the first order terms in the RHS of \eqref{eq_carl_2}, from below and make sure that they are positive. First we consider $\hat{E}_\rho \psi$ by estimating it via the term $\bar{E}_\rho \psi$. For $\bar{E}_\rho \psi$, using \eqref{eq_barrho1} shows that
\begin{align*}
\bar{E}_\rho \psi & = \eta^{-2} E_\rho \psi + \eta^{-2} \frac{r E_\theta \eta}{2} E_\theta \psi - \eta^{-2} \frac{\rho^2}{r^2} \sum_A \frac{r E_A \eta}{2} E_A \psi. \\
\Rightarrow \frac{r^2}{\rho^2} (\bar{E}_\rho \psi)^2 & \gtrsim \frac{r^2}{\rho^2} (E_\rho \psi)^2 - \frac{r^2}{\rho^2} (E_\theta \psi)^2 \mc{O}(\varepsilon^2 t^2 r^2) - \sum_A (E_A \psi)^2 \mc{O}(\varepsilon^2 t^2 r^2). \numberthis \label{eq_Eb_E_1}
\end{align*}
Now, using \eqref{eq_Ehat_1_LM} gives us
\begin{align*}
\frac{r^2}{\rho^2} (\hat{E}_\rho \psi)^2 \geqslant \varepsilon r^2 \frac{r^2}{\rho^2} (\hat{E}_\rho \psi)^2 \gtrsim \varepsilon r^2 \frac{r^2}{\rho^2} (\bar{E}_\rho \psi)^2 - c_n^2 \bar{f}^{-1} \varepsilon r^2 \cdot \psi^2.
\end{align*}
Then, using the above estimate and \eqref{eq_Eb_E_1} shows that
\begin{align*}
\frac{r^2}{\rho^2} (\hat{E}_\rho \psi)^2 & \gtrsim \varepsilon r^2 \frac{r^2}{\rho^2} (E_\rho \psi)^2 - \frac{r^2}{\rho^2} (E_\theta \psi)^2 \mc{O}(\varepsilon^3 t^2 r^4) -  \sum_A (E_A \psi)^2 \mc{O}(\varepsilon^3 t^2 r^4) \numberthis \label{eq_Eb_E_1a} \\
& \qquad - c_n^2 \bar{f}^{-1} \varepsilon r^2 \cdot \psi^2. 
\end{align*}
For the $\bar{E}_\theta$ component, using \eqref{eq.pc_Eb} shows that
\begin{align*}
\bar{E}_\theta \psi & = E_\theta \psi + \frac{r}{2} E_\theta \eta E_\rho. \\
\Rightarrow \varepsilon r^2 \frac{r^2}{\rho^2} (\bar{E}_\theta \psi)^2 & \gtrsim \varepsilon r^2 \frac{r^2}{\rho^2} (E_\theta \psi)^2 - \varepsilon r^2 \frac{r^2}{\rho^2} (E_\rho \psi)^2 \mc{O}(\varepsilon^2 t^2 r^2). \numberthis \label{eq_Eb_E_2}
\end{align*}
Finally, for the $\bar{E}_A$ component, using \eqref{eq.pc_Eb} gives us
\begin{align*}
\bar{E}_A \psi & = E_A \psi + \frac{r}{2} E_A \eta E_\rho. \\
\Rightarrow \varepsilon r^2 \sum_A (\bar{E}_A \psi)^2 & \gtrsim \varepsilon r^2 \sum_A (E_A \psi)^2 - \varepsilon r^2 (E_\rho \psi)^2 \mc{O}(\varepsilon^2 t^2 r^2). \numberthis \label{eq_Eb_E_3}
\end{align*}
Then combining \eqref{eq_Eb_E_1a}, \eqref{eq_Eb_E_2}, and \eqref{eq_Eb_E_3} shows that
\begin{align}
\frac{r^2}{\rho^2} & (\hat{E}_\rho \psi)^2 + \frac{1}{16} \varepsilon r^2 \cdot \frac{r^2}{\rho^2} (\bar{E}_\theta \psi)^2 + \frac{1}{16} \varepsilon r^2 \cdot \sum_A (\bar{E}_A \psi)^2 \label{eq_carl_6} \\
& \quad \gtrsim \frac{1}{32} \varepsilon r^2 \cdot \frac{r^2}{\rho^2} (E_\rho \psi)^2 + \frac{1}{32} \varepsilon r^2 \cdot \frac{r^2}{\rho^2} (E_\theta \psi)^2 + \frac{1}{32} \varepsilon r^2 \cdot \sum_A (E_A \psi)^2  - c_n^2 \bar{f}^{-1} \varepsilon r^2 \cdot \psi^2. \notag
\end{align}
Using the above, along with \eqref{eq_carl_2}, shows that \eqref{eq_carl_CF} holds for sufficiently large $a$. This concludes the proof of the proposition.
\end{proof}

\subsubsection{Conjugation reversal} \label{ssec_carl_conj_rev_LM}

In this section, we will go back to the original function $\phi$, by undoing the conjugation process introduced in \eqref{eq_carl_def1_LM}. In particular, we make the substitution $\psi = e^{-F} \phi$.

First, we define $P^\star = P^\star [\phi]$ on $\mc{U}$ as
\begin{align} \label{eq_carl_revP}
P^\star_\beta &:= S ( e^{ - F } \phi )\nabla_\beta ( e^{ - F } \phi ) - \frac{1}{2} \nabla_\beta \bar{f} \cdot \nabla^\mu ( e^{ - F }\phi) \nabla_\mu ( e^{ - F } \phi ) + w \cdot e^{ - F } \phi \nabla_\beta ( e^{ - F } \phi ) \\
\notag & \qquad + \frac{1}{2} ( \mc{A} \nabla_\beta \bar{f} - \nabla_\beta w ) \cdot e^{ - 2 F } \phi^2.
\end{align}

\begin{proposition} \label{prop_carl_3}
Let $\phi \in C^2(\mc{U}) \cap C^1(\bar{\mc{U}})$. Then, we have the following inequality
\begin{align}
\frac{1}{4a} \bar{f} e^{-2F} |\square \phi|^2 + \nb^\beta P_\beta^\star & \gtrsim \frac{\varepsilon r^2}{64} e^{-2F} \left[\frac{r^2}{\rho^2} (E_\rho \phi)^2 + \frac{r^2}{\rho^2} (E_\theta \phi)^2 + \sum_A (E_A \phi)^2 \right] \label{eq_carl_3_1} \\
\notag & \qquad + \frac{1}{8} a^2 b e^{-2F} \phi^2 + \frac{1}{32} a^2 \bar{f}^{-1} \varepsilon r^2 e^{-2F} \cdot \phi^2. 
\end{align}
\end{proposition}
\begin{proof}

Firstly, for the $\bar{E}_\rho$ component, we have
\begin{align*}
E_\rho \psi & = E_\rho (e^{-F} \phi) \\
& = e^{-F} E_\rho \phi - e^{-F} F' E_\rho \bar{f} \cdot \phi \\
& = e^{-F} E_\rho \phi + e^{-F} a \bar{f}^{-1} (1+b\bar{f}) \frac{2f}{r} \eta^{-2} \cdot \phi \\
& = e^{-F} E_\rho \phi + e^{-F} a (1+b\bar{f}) \frac{\eta^{-1}}{r}  \cdot \phi.
\end{align*}
Then,
\begin{align*}
\frac{1}{2} e^{-2F} (E_\rho \phi)^2 & \leqslant (E_\rho \psi)^2 + e^{-2F} a^2 \frac{1}{r^2} \cdot \phi^2,\\
e^{-2F} \frac{\varepsilon r^2 }{64} \frac{r^2}{\rho^2} (E_\rho \phi)^2 & \leqslant \frac{\varepsilon r^2 }{32} \frac{r^2}{\rho^2} (E_\rho \psi)^2 + \frac{1}{32} e^{-2F}  a^2 \bar{f}^{-1} \varepsilon r^2 \cdot \phi^2. \numberthis \label{eq_conj_rev_a}
\end{align*}
For the $\bar{E}_\theta$ component, we have
\begin{align*}
E_\theta \psi & = e^{-F} E_\theta \phi - e^{-F} a (1+b\bar{f}) \eta^{-1} E_\theta \eta \cdot \phi, \\
\frac{1}{2} \varepsilon r^2 e^{-2F} \frac{r^2}{\rho^2} (E_\theta \phi)^2 & \leqslant \varepsilon r^2 \frac{r^2}{\rho^2} (E_\theta \psi)^2 + e^{-2F} a^2 \bar{f}^{-1} \mc{O}(\varepsilon^3 t^2 r^4) \cdot \phi^2. \numberthis \label{eq_conj_rev_b}
\end{align*}
For the $\bar{E}_A$ component, we have
\begin{align*}
\frac{1}{2} \varepsilon r^2 e^{-2F} \sum_A (E_A \phi)^2 \leqslant \varepsilon r^2 \sum_A (E_A \psi)^2 + e^{-2F} a^2 b \mc{O}(\varepsilon^2 t^2 r^2) \cdot \phi^2, \numberthis \label{eq_conj_rev_c}
\end{align*}
where we used the fact that $\varepsilon \ll b$ for the zero-order term.

Combining \eqref{eq_conj_rev_a}-\eqref{eq_conj_rev_c}, gives us
\begin{align}
& \frac{1}{32} \varepsilon r^2 \cdot \frac{r^2}{\rho^2} (E_\rho \psi)^2 + \frac{1}{32} \varepsilon r^2 \cdot \frac{r^2}{\rho^2} (E_\theta \psi)^2 + \frac{1}{32} \varepsilon r^2 \cdot \sum_A (E_A \psi)^2 \label{eq_carl_3_3} \\
& \ \geqslant \frac{\varepsilon r^2}{64} e^{-2F} \left[ \frac{r^2}{\rho^2} (E_\rho \phi)^2 + \frac{r^2}{\rho^2} (E_\theta \phi)^2 + \sum_A (E_A \phi)^2 \right] - \frac{1}{8}  e^{-2F} a^2 b \phi^2 - \frac{1}{32} e^{-2F} a^2 \bar{f}^{-1} \varepsilon r^2 \phi^2. \notag
\end{align}
Next, we have
\begin{align*}
\mc{L} (e^{-F} \phi) = e^{-F} \square \phi, \qquad \frac{1}{4} a^2 b \psi^2 = \frac{1}{4} a^2 b e^{-2F} \phi^2,
\end{align*}
where we used \eqref{eq_carl_def1_LM} to get the first equation. Using the above expressions and \eqref{eq_carl_3_3} in \eqref{eq_carl_CF}, shows that \eqref{eq_carl_3_1} is true.
\end{proof}

The next proposition will allow us to estimate the boundary term when we integrate the latest pointwise Carleman estimate \eqref{eq_carl_3_1}. We will foliate the integral region using level sets of $f$, as that is more convenient for our computations. Hence, in the following proposition we measure the quantities with respect to $f$.
\begin{proposition} \label{prop_revP}
On $\bar{\mc{U}}$, the quantity $P^\star$ satisfies the following
\begin{equation} \label{eq_carl_revP1}
|P^\star (\nb^\sharp f)| \lesssim r_0 e^a f^{2a} \left[(E_\rho \phi)^2 + (E_\theta \phi)^2 + \sum_A (E_A \phi)^2 + a^2 f^{-1} \phi^2 \right].
\end{equation}
Moreover, if \eqref{eq_thm_carl_p0_LM} holds, then
\begin{equation} \label{eq_carl_revP2}
P^\star (\mc{N}) \Big|_{\partial \mc{U}} = \frac{1}{2} e^{-2F} \mc{N} \bar{f} |\mc{N} \phi|^2 \Big|_{\partial \mc{U}}.
\end{equation}
\end{proposition}

\begin{proof}
We have,
\begin{align*}
P^\star (\nb^\sharp f) & = S ( e^{ - F } \phi ) \nb^\beta f \nabla_\beta ( e^{ - F } \phi ) - \frac{1}{2} \nb^\beta f \nabla_\beta \bar{f} \cdot \nabla^\mu ( e^{ - F }\phi) \nabla_\mu ( e^{ - F } \phi )  \\
\notag & \qquad + w \cdot e^{ - F } \phi \nb^\beta f \nabla_\beta ( e^{ - F } \phi ) + \frac{1}{2} ( \mc{A} \nb^\beta f \nabla_\beta \bar{f} - \nb^\beta f \nabla_\beta w ) \cdot e^{ - 2 F } \phi^2.
\end{align*}
But, since
\begin{equation*}
\nb^\alpha f \nb_\alpha = \frac{r}{2} E_\rho,
\end{equation*}
we get that
\begin{align*}
P^\star (\nb^\sharp f) & = \frac{r}{2} S ( e^{ - F } \phi ) E_\rho ( e^{ - F } \phi ) - \frac{1}{2} \cdot \frac{r}{2} E_\rho \bar{f} \cdot \nabla^\mu ( e^{ - F }\phi) \nabla_\mu ( e^{ - F } \phi ) \\
\notag & \qquad + \frac{r}{2} w \cdot e^{ - F } \phi E_\rho ( e^{ - F } \phi ) + \frac{1}{2} \cdot \left( \frac{r}{2} \mc{A} E_\rho \bar{f} - \nb^\beta f \nb_\beta w \right) \cdot e^{ - 2 F } \phi^2.
\end{align*}
Now we will estimate each of the five terms in the above equation individually.

\textbf{(i)} Using \eqref{eq_barrho1}, shows that
\begin{align*}
|S ( e^{ - F } \phi )| & = \left|\frac{r}{2} \bar{E}_\rho (e^{-F} \phi)\right| \\
& = \left|e^{-F} \frac{r}{2} \bar{E}_\rho \phi - e^{-F} F' \frac{r}{2} \bar{E}_\rho \bar{f} \cdot \phi \right| \\
& \lesssim r e^{-F} \left( |E_\rho \phi| + |E_\theta \phi| + \bigg| \sum_A E_A \phi \bigg| \right) + e^{-F} a \cdot |\phi|,
\end{align*}
and we also have
\begin{align*}
\frac{r}{2} |E_\rho ( e^{ - F } \phi )| & = \frac{r}{2}|e^{-F} E_\rho \phi - e^{-F} F' E_\rho f \cdot \phi| \\
& \lesssim r e^{-F} |E_\rho \phi| + e^{-F} a\cdot |\phi|.
\end{align*}
Thus 
\begin{align*}
\frac{r}{2} |S ( e^{ - F } \phi ) E_\rho ( e^{ - F } \phi )| \lesssim r_0^2 e^{-2F} \left[  (E_\rho \phi)^2 + (E_\theta \phi)^2 + \sum_A (E_A \phi)^2 + a^2 f^{-1} \phi^2 \right],
\end{align*}
where we also used the fact that $f \leqslant r_0^2$ for the $\phi^2$ term.

\textbf{(ii)} For the second term, we have
\begin{align*}
|\nabla^\mu ( e^{ - F }\phi) \nabla_\mu ( e^{ - F } \phi )| & = \big| e^{-2F} (F')^2 \bar{f} \eta^{-2} [1+\mc{O}(\varepsilon t^2)] \cdot \phi^2 - 2 e^{-2F} F' \frac{r}{2} \bar{E}_\rho \phi \cdot \phi \\
& \qquad + e^{-2F} \nb^\mu \phi \nb_\mu \phi \big| \\
& \lesssim e^{-2F} \left[  (E_\rho \phi)^2 + (E_\theta \phi)^2 + \sum_A (E_A \phi)^2 + a^2 f^{-1} \phi^2 + a^2 f^{-2} r^2 \phi^2 \right] \\
& \qquad + e^{-2F} \frac{r^2}{\rho^2} \left[  (E_\rho \phi)^2 + (E_\theta \phi)^2 + \frac{\rho^2}{r^2} \sum_A (E_A \phi)^2 \right],
\end{align*}
which implies that
\begin{align*}
\left| \frac{1}{2} \cdot \frac{r}{2} E_\rho \bar{f} \cdot \nabla^\mu ( e^{ - F }\phi) \nabla_\mu ( e^{ - F } \phi ) \right| & \lesssim f | \nabla^\mu ( e^{ - F }\phi) \nabla_\mu ( e^{ - F } \phi )| \\
& \lesssim r_0^2 e^{-2F} \left[  (E_\rho \phi)^2 + (E_\theta \phi)^2 + \sum_A (E_A \phi)^2 + a^2 f^{-1} \phi^2 \right].
\end{align*}

\textbf{(iii)} Similarly, for the third term, we have
\begin{align*}
\frac{r}{2} |w \cdot e^{ - F } \phi E_\rho ( e^{ - F } \phi ) | \lesssim r_0^2 e^{-2F} \left[  (E_\rho \phi)^2 + (E_\theta \phi)^2 + \sum_A (E_A \phi)^2 + a^2 f^{-1} \phi^2 \right].
\end{align*}

\textbf{(iv)} For the fourth term, we note that
\begin{align*}
r |\mc{A} E_\rho \bar{f}| \lesssim |\mc{A} f| \lesssim a^2 , \qquad | \nb^\beta f \nb_\beta w | \lesssim a^2,
\end{align*}
for sufficiently large $a$. Then,
\begin{align*}
\left| \frac{1}{2} \cdot \left( \frac{r}{2} \mc{A} E_\rho \bar{f} - \nb^\beta f \nb_\beta w \right) \cdot e^{ - 2 F } \phi^2 \right| \lesssim r_0^2 e^{-2F} a^2 f^{-1} \phi^2.
\end{align*}
Combining points \textbf{(i)}-\textbf{(iv)} above, and estimating $e^{-2F}$ using \eqref{eq_thm_carl_a_LM} and \eqref{eq_e2F_zeta_LM} shows that \eqref{eq_carl_revP1} is true. 

Now, let us prove \eqref{eq_carl_revP2}. Due to \eqref{eq_thm_carl_p0_LM}, we see that
\begin{align*}
P^\star (\mc{N})|_{\partial \mc{U}} = e^{-2F} \left( S \phi \mc{N} \phi - \frac{1}{2} \mc{N} \bar{f} \cdot |\mc{N} \phi|^2 \right) \bigg|_{\partial \mc{U}}.
\end{align*}
Next, as $S - \mc{N} \bar{f} \cdot \mc{N}$ is the $g$-orthogonal projection of $S$ onto $\partial \mc{U} \cap \mc{D}$, we get
\[ S \phi|_{\partial \mc{U} \cap \mc{D}} = \mc{N} \bar{f} \cdot \mc{N} \phi|_{\partial \mc{U} \cap \mc{D}}.\]
The last two equations together show that \eqref{eq_carl_revP2} holds.
\end{proof}

\subsubsection{Integration} \label{ssec_carl_int_eq_LM}
Now we will integrate \eqref{eq_carl_3_1} to obtain the main Carleman estimate.

Let $\delta$ be sufficiently small. Then, define the following sets
\begin{align*}
\mc{D}_\delta = \mc{D} \cap \{ f>\delta \}, \qquad \mc{H}_\delta = \mc{D} \cap \{ f = \delta \}
\end{align*}
We integrate \eqref{eq_carl_3_1} on $\mc{U} \cap \mc{D}_\delta$, to get
\begin{align*}
\frac{1}{4a} \int_{\mc{U} \cap \mc{D}_\delta} \zeta \bar{f} |\square \phi|^2 + \int_{\mc{U} \cap \mc{D}_\delta} \nb^\beta P_\beta^\star & \geqslant \frac{\varepsilon}{64} \int_{\mc{U} \cap \mc{D}_\delta} \zeta r^2 \left[ \frac{r^2}{\rho^2} (E_\rho \phi)^2 + \frac{r^2}{\rho^2} (E_\theta \phi)^2 + \sum_A (E_A \phi)^2 \right] \\
& \qquad + \frac{1}{8} a^2 b \int_{\mc{U} \cap \mc{D}_\delta} \zeta \phi^2  + \frac{1}{32} a^2 \varepsilon \int_{\mc{U} \cap \mc{D}_\delta} \zeta r^2 \bar{f}^{-1} \phi^2,
\end{align*}
where we also used \eqref{eq_e2F_zeta_LM} for replacing $e^{-2F}$ by $\zeta$. The final term on the RHS of the above estimate is no longer needed, so will drop it. Now, taking the limit as $\delta \searrow 0$, and using the monotone convergence theorem gives us the following estimate
\begin{align*}
\frac{1}{4a} \int_{\mc{U} \cap \mc{D}} & \zeta f |\square \phi|^2 + \lim_{\delta \searrow 0} \int_{\mc{U} \cap \mc{D}_\delta} \nb^\beta P_\beta^\star \\
& \geqslant \frac{\varepsilon}{64} \int_{\mc{U} \cap \mc{D}} \zeta r^2 \left[ \frac{r^2}{\rho^2}(E_\rho \phi)^2 + \frac{r^2}{\rho^2}(E_\theta \phi)^2 + \sum_A (E_A \phi)^2 \right] + \frac{1}{8} a^2 b \int_{\mc{U} \cap \mc{D}} \zeta \phi^2 ,
\end{align*}
if the integral term with the limit exists. Also note that, we replaced the $\bar{f}$ on the LHS by $f$; this can be done by estimating $\eta$ appropriately.

Now, $\mc{H}_\delta$ is a timelike hypersurface, with the outward unit normal given by $f^{-1} \nb^\sharp f$. Then, using the divergence theorem implies that
\begin{align*}
\lim_{\delta \searrow 0} \int_{\mc{U} \cap \mc{D}_\delta} \nb^\beta P_\beta^\star & = \lim_{\delta \searrow 0} \int_{\partial \mc{U} \cap \mc{D}_\delta} P^\star (\mc{N}) - \lim_{\delta \searrow 0} \int_{\mc{U} \cap \mc{H}_\delta} f^{-1} P^\star(\nb^\sharp f) \\
& = \frac{1}{2} \int_{\partial \mc{U} \cap \mc{D}} \zeta \mc{N} \bar{f} |\mc{N} \phi|^2 - \lim_{\delta \searrow 0} \delta^{-1} \int_{\mc{U} \cap \mc{H}_\delta} P^\star(\nb^\sharp f).
\end{align*}
Using \eqref{eq_carl_revP1} shows that
\begin{align*}
\left| \lim_{\delta \searrow} \delta^{-1} \int_{\mc{U} \cap \mc{H}_\delta} P^\star(\nb^\sharp f) \right| \lesssim C' \lim_{\delta \searrow 0} \delta^{2a-2} \int_{\mc{U} \cap \mc{H}_\delta} 1,
\end{align*}
for some constant $C'$ depending on $\phi$ and $r_0$.

To complete the proof of the theorem, it is enough to show that the RHS of the above estimate is zero. For this purpose, we foliate $\mc{H}_\delta$ using level sets of $t$. Let $\bar{\nb}^\sharp t$ denote the $g$-gradient of $t$ with respect to the induced metric. Then, 
\begin{equation*}
g(\bar{\nb}^\sharp t,\bar{\nb}^\sharp t) = -\frac{r^2}{\rho^2} (E_\theta t)^2 + \sum_A (E_A t)^2 = -\frac{r^2}{\rho^2} \left( 1 + \frac{\rho^2}{r^2} \mc{O}(\varepsilon r^2) \right) + \mc{O} (\varepsilon r^2) = -\frac{r^2}{4f} ( 1 +\mc{O}(\varepsilon r^2) ).
\end{equation*}
This implies that
\begin{align*}
|g(\bar{\nb}^\sharp t,\bar{\nb}^\sharp t)|^{-\frac{1}{2}} = 2 f^{\frac{1}{2}} r^{-1} ( 1 +\mc{O}(\varepsilon r^2) ).
\end{align*}
Hence, applying the coarea formula gives
\begin{align*}
\lim_{\delta \searrow 0} \delta^{2a-2} \int_{\mc{U} \cap \mc{H}_\delta} 1 & \lesssim \lim_{\delta \searrow 0} \delta^{2a-2} \int_{-r_0}^{r_0} \int_{\mb{S}^{n-1}} |g(\bar{\nb}^\sharp t,\bar{\nb}^\sharp t)|^{-\frac{1}{2}} r^{n-1}|_{(t, r= \sqrt{t^2 + 4\delta}, \omega)} d\mathring{\gamma}_{\omega} dt \\
& \lesssim r_0^{n-1}\lim_{\delta \searrow 0} \delta^{2a-\frac{3}{2}} \int_{-r_0}^{r_0} \int_{\mb{S}^{n-1}} r^{-1}|_{(t, r= \sqrt{t^2 + 4\delta}, \omega)} d\mathring{\gamma}_{\omega} dt \\
& = 0.
\end{align*}
This shows that \eqref{eq_thm_carl_est_LM} holds, which completes the proof of Theorem \ref{thm_carl_est_LM}.

\section{Conclusion} \label{sec_control_LM}

In this section, we show how our main Carleman estimate, Theorem \ref{thm_carl_est_LM}, can be applied to prove the main observability and controllability results of this paper.

\subsection{The Carleman estimate revisited}

Recall the Carleman estimate \eqref{eq_thm_carl_est_LM} was given in terms of the orthogonal frames $\{ E_\rho, E_\theta, E_A \}$.
However, since $E_\theta$ and $E_\rho$ become arbitrarily close to each other as one moves toward $\partial \mc{D}$, this frame is inconvenient for controlling $H^1$-norms.
To address this, we reformulate \eqref{eq_thm_carl_est_LM} here in terms of the frames $\{ E_\rho, E_0, E_A \}$, which do generate all the directions of $\mc{M}$ in a uniform manner.

\begin{proposition} \label{thm_carl_est_LM_2}
Under all the assumptions of Theorem \ref{thm_carl_est_LM}, we have the estimate
\begin{align}
\label{eq_carl_H1_norm_2} C' \varepsilon \int_{\mc{U} \cap \mc{D}} \zeta \rho^2 & \left[ (E_\rho \phi)^2 + (E_0 \phi)^2 + \sum_A (E_A \phi)^2 \right] + \frac{1}{8} a^2 b \int_{\mc{U} \cap \mc{D}} \zeta \phi^2 \\
\notag &\qquad \leqslant \frac{1}{4a} \int_{\mc{U} \cap \mc{D}} \zeta f |\square \phi|^2 + \frac{1}{2} \int_{\partial \mc{U} \cap \mc{D}} \zeta \mc{N} (f [1- \varepsilon t^2 ]^{-1}) |\mc{N} \phi|^2 \text{,}  
\end{align}
for some constant $C'>0$.
\end{proposition}

\begin{proof}
First, note that \eqref{eq.hf_E_special} implies
\begin{align*}
r E_\theta \phi = t E_\rho \phi + \frac{ \rho^2 }{ r^2 } r E_0 \phi \text{,}
\end{align*}
which implies that in $\mc{D}$, we have
\begin{align*}
\frac{\rho^4}{r^4} r^2 (E_0 \phi)^2 \lesssim r^2 (E_\theta \phi)^2 + t^2 (E_\rho \phi)^2 \lesssim r^2 (E_\theta \phi)^2 + r^2 (E_\rho \phi)^2.
\end{align*}
From the above, we then obtain the following:
\[
\rho^2 (E_0 \phi)^2 = \frac{\rho^2}{r^2} r^2 (E_0 \phi)^2 \lesssim \frac{r^2}{\rho^2} r^2 (E_\theta \phi)^2 + \frac{r^2}{\rho^2} r^2 (E_\rho \phi)^2 \text{.}
\]
Thus, using the above and \eqref{eq_thm_carl_est_LM}, we see that
\begin{align*}
\mc{C} \varepsilon \int_{\mc{U} \cap \mc{D}} \zeta & \left[ \frac{r^2}{\rho^2} r^2 (E_\rho \phi)^2 + \rho^2 (E_0 \phi)^2 + r^2 \sum_A (E_A \phi)^2 \right] + \frac{1}{8} a^2 b \int_{\mc{U} \cap \mc{D}} \zeta \phi^2 \\
&\qquad \leqslant \frac{1}{4a} \int_{\mc{U} \cap \mc{D}} \zeta f |\square \phi|^2 + \frac{1}{2} \int_{\partial \mc{U} \cap \mc{D}} \zeta \mc{N} (f [1- \varepsilon t^2 ]^{-1}) |\mc{N} \phi|^2 \text{,} 
\end{align*}
The desired \eqref{eq_carl_H1_norm_2} now follows by combining the above with the inequalities
\[ \rho^2 \leqslant r^2 \leqslant \frac{r^2}{\rho^2} \cdot \rho^2 \leqslant \frac{r^2}{\rho^2} \cdot r^2 \text{.} \qedhere \]
\end{proof}

\subsection{Observability}

In this subsection, we sketch how Theorem \ref{thm_carl_est_LM} can be used to prove observability inequalities.
Here, we will work with the adjoint system \eqref{eq_obs_wv_LM},
\begin{equation} \label{eq_obs_main_adeq_LM}
\begin{cases}
\square \phi + \nb_{\mc{X}} \phi + V \phi = 0, \qquad & \text{on } \mc{U},\\
\phi = 0, & \text{in } \partial \mc{U},\\
(\phi,\partial_t \phi) = (\phi_0,\phi_1), & \text{on } \mc{V}_-,
\end{cases}
\end{equation}
where $\mc{V}_-$ is a spacelike cross-section of $\mc{U}$ on which the initial data is posed.

To prove observability, we require suitable energy estimates satisfied by solutions of \eqref{eq_obs_main_adeq_LM}.
The following can be proved using standard energy arguments; see, for instance, \cite{Arick}. 

\begin{proposition} \label{thm_energy_prop}
Let $\mc{U}$ be as in Assumption \ref{ass_LM}, and let $\mc{V}_\pm$ be two spacelike cross-sections of $\mc{U}$.
Then, there exists $\tilde{C}:= \tilde{C}(V, \mc{X}, \mc{V}_\pm ) > 0$ such that
\footnote{More specifically, $\tilde{C}$ depends on the time separation between $\mc{V}_-$ and $\mc{V}_+$.}
\begin{equation} \label{eq_energy_prop}
|| (\phi, \nb_N \phi)||_{H^1(\mc{V}_-) \times L^2(\mc{V}_-)} \leqslant \tilde{C} || (\phi, \nb_N \phi)||_{H^1(\mc{V}_+) \times L^2(\mc{V}_+)} \text{,}
\end{equation}
where $N$ denotes the future unit normal to $\mc{V}_\pm$, and where $\phi \in C^2(\mc{U}) \cap C^1(\bar{\mc{U}})$ is any (classical) solution of the system \eqref{eq_obs_main_adeq_LM} that also satisfies $\phi|_{\partial\mc{U}} = 0$.
\end{proposition}

We now present two observability results that differ depending on the location of a chosen ``centre point" $p \in \mc{M}$, near which we apply our Carleman estimate.
Since the arguments are analogous to those previously done on Minkowski spacetime (which were presented in detail in \cite{Arick}), we will only sketch the proofs of our observability results here.

\subsubsection{Exterior observability}

The following result addresses the case when observability is obtained using the Carleman estimate about a point $p \not\in \bar{\mc{U}}$.

\begin{theorem} \label{thm_obs_main_final_LM}
Assume the setup given in Assumption \ref{ass_LM}, and fix $p \notin \bar{\mc{U}}$.
In addition, let $\mc{V}_-$ be a spacelike cross-section of $\mc{U}$.
Then, there exist universal constants $\mc{C}_\dagger >0$ and $0 < \varepsilon_0 \ll 1$ such that if the following curvature bounds hold,
\footnote{Here, the frames $\{ E_\rho, E_0, E_A \}$ are defined with respect to the point $p$, as in Section \ref{ssec_hyperq_func_LM}.}
\begin{align}
\label{eq_obs_curv_final_LM} \sup_{ X, Y, Z \in \{ E_\rho, E_0, E_A \} } | R ( E_\rho, X, Y, Z ) | & < \frac{\varepsilon_0 \mc{C}_\dagger}{r_0^2}, \\
\notag \sup_{ X, Y, Z \in \{ E_\rho, E_0, E_A \} } | \nabla_Z R ( E_\rho, X, E_\rho, Y ) | & < \frac{\mc{C}_\dagger}{r_0^3},
\end{align}
then there exists a constant $C := C(p,\mc{U},V,\mc{X},\varepsilon_0) >0$ such that the observability inequality
\begin{equation} \label{eq_obs_main_final_LM}
\| (\phi_0,\phi_1) \|_{H^1(\mc{V}_-) \times L^2(\mc{V}_-)}^2 \leqslant C \| \mc{N} \phi \|_{L^2(\Gamma_+)}^2
\end{equation}
holds for any $(\phi_0,\phi_1) \in H^1_0(\mc{V}_-) \times L^2(\mc{V}_-)$, where $\phi$ is the corresponding solution of \eqref{eq_obs_main_adeq_LM}, where $\mc{N}$ is the outward unit normal of $\mc{U}$, and where the observation region $\Gamma_+$ is given by
\begin{equation} \label{eq_gamma_final_LM}
\Gamma_+ := \partial \mc{U} \cap \{ \mc{N} (f [1- \varepsilon_0 r_0^{-2} t^2 ]^{-1}) > 0 \} \cap \mc{D} \text{.}
\end{equation}
\end{theorem}

\begin{proof}[Proof sketch]
We first apply the Carleman estimate variant \eqref{eq_carl_H1_norm_2} to $\mc{U}$ and $p$ to obtain
\begin{align*}
&C' \varepsilon_0 r_0^{-2} \int_{\mc{U} \cap \mc{D}} \zeta \rho^2 \left[ (E_\rho \phi)^2 + (E_0 \phi)^2 + \sum_A (E_A \phi)^2 \right] + \frac{1}{8} a^2 b_0 r_0^{-2} \int_{\mc{U} \cap \mc{D}} \zeta \phi^2 \\
&\qquad \leqslant \frac{ \mc{C} }{a} \int_{\mc{U} \cap \mc{D}} \zeta f |\nb_{\mc{X}} \phi|^2 + \frac{ \mc{C} }{a} \int_{\mc{U} \cap \mc{D}} \zeta f |V \phi|^2 + \frac{1}{2} \int_{\partial \mc{U} \cap \mc{D}} \zeta \mc{N} (f [1- \varepsilon t^2 ]^{-1}) |\mc{N} \phi|^2 \text{,}
\end{align*}
with $\varepsilon_0$ as in the theorem statement, and with all other quantities as in Theorem \ref{thm_carl_est_LM}.
Taking $a := a ( \mc{U}, V, \mc{X}, \varepsilon_0 )$ to be large enough, we can absorb the first two terms on the right,
\begin{equation}
\mc{C} \varepsilon_0 r_0^{-2} \int_{\mc{U} \cap \mc{D}} \zeta \rho^2 \left[ (E_\rho \phi)^2 + (E_0 \phi)^2 + \sum_A (E_A \phi)^2 \right] + \mc{C} a^2 b_0 r_0^{-2} \int_{\mc{U} \cap \mc{D}} \zeta \phi^2 \leqslant I_{\Gamma} \text{,} \label{eq_ext_obs_pf_2}
\end{equation}
where $I_\Gamma$ is the boundary term
\begin{equation}
\label{eq_ext_obs_pf_Gamma} I_{\Gamma} := \int_{\partial \mc{U} \cap \mc{D}} \zeta \mc{N} (f [1- \varepsilon t^2 ]^{-1}) |\mc{N} \phi|^2 \text{.}  
\end{equation}

Note the Carleman weight $\zeta$ vanishes on the boundary of the null cone $\mc{D}$; see \eqref{eq_thm_carl_def_LM}.
Thus, the LHS of \eqref{eq_ext_obs_pf_2} fails to capture the full $H^1$-norm of $\phi$ near $\partial \mc{D}$.
Hence, we must restrict the domain of integration away from $\partial \mc{D}$, so that the weight $\zeta$ is uniformly positive.
For this purpose, we reduce $\mc{U} \cap \mc{D}$ on the LHS of \eqref{eq_ext_obs_pf_2} to a slab of the form
\begin{equation}
\label{eq_ext_obs_pf_slab} \mc{U}_\text{slab} := \mc{U} \cap \left\{ |t| \ll r_- \right\} \subseteq \mc{U} \cap \mc{D} \text{,} \qquad r_- := \inf_{ \mc{U} \cap \{ t = 0 \} } r \text{.}
\end{equation}
Since $p \notin \bar{\mc{U}}$, and since $\partial \mc{U}$ is timelike, we can indeed bound $f$ and $\zeta$ from below on $\mc{U}_\text{slab}$:
\begin{equation}
\left. \frac{f}{1-\varepsilon t^2} \right|_{ \mc{U}_\text{slab} } \geqslant \frac{r_-^2}{16}, \qquad \zeta |_{ \mc{U}_\text{slab} } \geqslant \left( \frac{r_-}{4} e^{\frac{b r_-}{4}} \right)^{4a} \text{.}
\end{equation}
Using the above, the estimate \eqref{eq_ext_obs_pf_2} reduces to
\[
C' \left( \frac{r_-}{4} e^{\frac{b r_-}{4}} \right)^{4a} \int_{\mc{U}_\text{slab}} \left\{ \frac{\varepsilon_0 r_-^2}{r_0^2} \left[ (E_\rho \phi)^2 + (E_0 \phi)^2 + \sum_A (E_A \phi)^2 \right] + \frac{a^2 b_0}{r_0^2} \phi^2 \right\} \leqslant I_{\Gamma} \text{.}
\]

Next, we use Fubini's theorem to split the above integral over $\mc{U}_\text{slab}$ into nested integrals over spacelike cross-sections, and we then use the energy estimate from Proposition \ref{thm_energy_prop} to control the integral over each cross-section by the energy at $\mc{V}_-$.
This results in the estimate
\begin{align} \label{eq_ext_obs_pf_6}
C \| (\phi_0,\phi_1) \|_{H^1(\mc{V}_-) \times L^2(\mc{V}_-)}^2 \leqslant I_\Gamma \text{,} \qquad C := C ( p, \mc{U}, V, \mc{X}, \varepsilon_0 ) \text{.}
\end{align}
Finally, for $I_\Gamma$, since the integrand in \eqref{eq_ext_obs_pf_Gamma} is only positive on $\Gamma_+$, we have
\[
I_\Gamma \leq C' \int_{\Gamma_+} |\mc{N} \phi|^2 \text{,} \qquad C' := C' ( p, \mc{U}, V, \mc{X}, \varepsilon_0 ) \text{.}
\]
Combining the above with \eqref{eq_ext_obs_pf_6} completes the proof of \eqref{eq_obs_main_final_LM}.
\end{proof}

\subsubsection{Interior observability}

Next, we address the remaining case when $p \in \bar{\mc{U}}$.
Here, there is an additional technical issue that arises---if we apply the Carleman estimate about $p \in \bar{\mc{U}}$, then the Carleman weight $\zeta$ can vanish inside $\overline{\mc{U} \cap \mc{D}}$.
This prevents us from capturing the $H^1$-norm of $\phi$ along any spacelike cross-section of $\mc{U}$.

To get around this, we must apply the Carleman estimate differently than before:
\begin{itemize}
\item If $p \in \partial \mc{U}$, then we apply the estimate about a point $p' \notin \bar{\mc{U}}$ that is close to $p$, and we make use of estimates already obtained from Theorem \ref{thm_obs_main_final_LM}.

\item Otherwise, if $p \in \mc{U}$, then we apply the Carleman estimate with respect to a pair of points $p_1,p_2 \in \mc{U}$ close to $p$ and then sum the resulting estimates.
\end{itemize}
The price to be paid is that we must slightly enlarge our observation domain on the boundary, and we must also slightly weaken our curvature assumptions.

\begin{theorem} \label{thm_int_obs_1}
Assume the setup given in Assumption \ref{ass_LM}, and fix $p \in \bar{\mc{U}}$.
Moreover:
\begin{itemize}
\item Assume there exists a neighbourhood $\mc{V}$ of $p$ in $\mc{M}$, satisfying that $\mc{M}$ is a geodesically star-shaped neighbourhood of every $q \in \mc{V}$.

\item Let $\mc{V}_-$ be a spacelike cross-section of $\mc{U}$.
\end{itemize}
Then, there exist universal constants $\mc{C}_\dagger >0$ and $0 < \varepsilon_0 \ll 1$ such that if the bound \eqref{eq_obs_curv_final_LM} holds, then there exists $C := C(p, \mc{U}, V, \mc{X}, \varepsilon_0, \Gamma) > 0$ such that the observability inequality
\begin{equation} \label{eq_obs_main_final_LMi}
\| (\phi_0,\phi_1) \|_{H^1(\mc{V}_-) \times L^2(\mc{V}_-)}^2 \leqslant C \| \mc{N} \phi \|_{L^2(\Gamma)}^2
\end{equation}
holds for any $(\phi_0,\phi_1) \in H^1_0(\mc{V}_-) \times L^2(\mc{V}_-)$, where $\phi$ is the corresponding solution of \eqref{eq_obs_main_adeq_LM}, where $\mc{N}$ is the outward unit normal of $\mc{U}$, and where the observation region $\Gamma$ can be any open subset of $\partial \mc{U}$ that contains $\bar{\Gamma}_+$, with $\Gamma_+$ defined as in \eqref{eq_gamma_final_LM}.
\end{theorem}

Note the difference in the normal coordinate assumptions between Theorem \ref{thm_obs_main_final_LM} and Theorem \ref{thm_int_obs_1}---we need stronger assumptions in Theorem \ref{thm_int_obs_1}, as we must also apply the Carleman estimate about points near $p$.
Moreover, the observation domain $\Gamma$ in Theorem \ref{thm_int_obs_1} can be interpreted as being slightly larger than but arbitrarily close to $\Gamma_+$.

\begin{remark}
By continuity, both Assumption \ref{ass_LM} and the curvature conditions \eqref{eq_obs_curv_final_LM} also hold for points near $p$.
Thus, the Carleman estimate can be applied near $p$ as well.
\end{remark}

\begin{proof}[Proof sketch]
First, if $p \in \partial \mc{U}$, then we choose a point $p' \in \mc{V} \setminus \bar{\mc{U}}$ near $p$.
Repeating the proof of Theorem \ref{thm_obs_main_final_LM}, but now centred about the point $p'$, yields
\[
\| (\phi_0,\phi_1) \|_{H^1(\mc{V}_-) \times L^2(\mc{V}_-)}^2 \leqslant C \| \mc{N} \phi \|_{L^2(\Gamma_+')}^2 \text{,}
\]
where $\Gamma_+'$ is as in \eqref{eq_gamma_final_LM}, but defined with respect to $p'$ rather than $p$.
If $p'$ is sufficiently close to $p$, then $\Gamma_+' \subseteq \Gamma$, and the desired \eqref{eq_obs_main_final_LMi} follows immediately.

From now on, we assume $p \in \mc{U}$.
Also, choose $p_1, p_2 \in \mc{V} \cap \mc{U}$ near $p$, with $t (p_1) = t (p_2) = 0$.
As a general convention, we will attach the index $i$ to refer to quantities with respect to the coordinate system centred at $p_i$; for instance, $f_i = \frac{1}{4} (r_i^2 - t_i^2)$, where $r_i$ and $t_i$ denote normal polar coordinates with respect to $p_i$. 
Applying the Carleman estimate \eqref{eq_carl_H1_norm_2} about $p_i$ (with common constants $a$, $b_0$, $\varepsilon_0$) and proceeding as in the proof of Theorem \ref{thm_obs_main_final_LM}, we obtain
\begin{equation}
\mc{C} \varepsilon_0 r_0^{-2} \int_{ \mc{U} \cap \mc{D}_i } \zeta^i \rho_i^2 \left[ (E_\rho^i \phi)^2 + (E_0^i \phi)^2 + \sum_A (E_A^i \phi)^2 \right] + \mc{C} a^2 b_0 r_0^{-2} \int_{ \mc{U} \cap \mc{D}_i } \zeta^i \phi^2 \leq I_\Gamma^i \text{,} \label{eq_int_obs_pf_2}
\end{equation}
where $I_\Gamma^i$ denotes the corresponding boundary term
\begin{equation}
\label{eq_int_obs_pf_Gamma} I_\Gamma^i = \int_{\partial \mc{U} \cap \mc{D}_i} \zeta^i \mc{N} (f_i [ 1 - \varepsilon t_i^2 ]^{-1}) |\mc{N} \phi |^2 \text{.}
\end{equation}

Similar to before, we define the slab
\begin{equation}
\label{eq_int_obs_pf_slab} \mc{U}_\text{slab} := \mc{U} \cap \left\{ |t| \ll r_- \right\} \text{,} \qquad r_- := \min( r_1 ( p_2 ), r_2 ( p_1 ) ) \text{.}
\end{equation}
The key observations here are that if $p_1 \neq p_2$, then:
\begin{itemize}
\item $\mc{U}_\text{slab}$ is contained in $( \mc{U} \cap \mc{D}_1 ) \cup ( \mc{U} \cap \mc{D}_2 )$.

\item $\max ( \zeta^1, \zeta^2 )$ is uniformly bounded from below on $\mc{U}_\text{slab}$ by a positive constant.
\end{itemize}
Thus, summing the estimates \eqref{eq_int_obs_pf_2} for $i = 1, 2$ and noting the above, we obtain
\[
C' \int_{\mc{U}_\text{slab}} \left[ (E_\rho \phi)^2 + (E_0 \phi)^2 + \sum_A (E_A \phi)^2 + \phi^2 \right] \leqslant I_\Gamma^1 + I_\Gamma^2 \text{,}
\]
where $C'$---and all subsequent constants---also depends on $p$, $\mc{U}$, $V$, $\mc{X}$, $\varepsilon_0$, $\Gamma$.

Finally, applying the energy estimate from Proposition \ref{thm_energy_prop} to the above and then bounding $I_\Gamma^i$ as in the proof of Theorem \ref{thm_obs_main_final_LM} results in the bound
\[
C \| (\phi_0, \phi_1) \|_{H^1(\mc{V}_-) \times L^2(\mc{V}_-)}^2 \leqslant I_\Gamma^1 + I_\Gamma^2 \leqslant C' \sum_{ i = 1, 2 } \int_{ \Gamma_+^i } | \mc{N} \phi |^2 \text{.}
\]
The observability inequality \eqref{eq_obs_main_final_LMi} now follows from the above, along with the observation that $\Gamma_+^1 \cup \Gamma_+^2 \subseteq \Gamma$ as long as $p_1$ and $p_2$ are sufficiently close to $p$.
\end{proof}

\begin{remark}
One can obtain more detailed information on the observability constants $C$ in Theorems \ref{thm_obs_main_final_LM} and \ref{thm_int_obs_1} through a more careful analysis of the estimates within their proofs.
We avoid doing this here for brevity, but see \cite{Arick} for details of this process.
\end{remark}

\subsection{Controllability}

Through the Hilbert uniqueness method, the observability inequalities from Theorems \ref{thm_obs_main_final_LM} and \ref{thm_int_obs_1} imply our main controllability results.
In the following, we give the precise statements of our controllability results:

\begin{theorem}[Exterior Control] \label{thm_control_ext_mr}
Suppose Assumption \ref{ass_LM} holds, and fix $p \notin \bar{\mc{U}}$.
Moreover:
\begin{itemize}
\item Let $\mc{V}_-$ and $\mc{V}_+$ be two spacelike cross-sections of $\mc{U}$, with $\mc{V}_-$ lying in the past of $\mc{U} \cap \bar{\mc{D}}$, and with $\mc{V}_+$ lying in the future of $\mc{U} \cap \bar{\mc{D}}$.

\item Let $\mc{C}_\dagger > 0$ and $0 < \varepsilon_0 \ll 1$ be the universal constants from the statement of Theorem \ref{thm_obs_main_final_LM}, and assume that the curvature bounds \eqref{eq_obs_curv_final_LM} hold.

\item Let the control region $\Gamma_+ \subset \partial \mc{U}$ be defined as in \eqref{eq_gamma_final_LM}.
\end{itemize}
Then, given $(y_0^{\pm}, y_1^{\pm}) \in L^2(\mc{V}_{\pm}) \times H^{-1}(\mc{V}_{\pm})$, there exists $F \in L^2(\Gamma_+)$ so that the solution $y$ of 
\[
\begin{cases}
\square y + \nb_{\mc{X}} y + q y = 0, \qquad & \text{on } \mc{U} \text{,} \\
y = F \mathbf{1}_{ \Gamma_+ }, & \text{on } \partial \mc{U} \text{,} \\
(y,\partial_t y) = (y_0^-,y_1^-), & \text{on } \mc{V}_- \text{,}
\end{cases}
\]
satisfies
\[
(y,\partial_t y) = (y_0^+,y_1^+) \qquad \text{on } \mc{V}_+ \text{.}
\]
\end{theorem}

\begin{theorem}[Interior Control] \label{thm_control_int_mr}
Suppose Assumption \ref{ass_LM} holds, and fix $p \in \bar{\mc{U}}$.
Moreover:
\begin{itemize}
\item Let $\mc{V}_-$ and $\mc{V}_+$ be two spacelike cross-sections of $\mc{U}$, with $\mc{V}_-$ lying in the past of $\mc{U} \cap \bar{\mc{D}}$, and with $\mc{V}_+$ lying in the future of $\mc{U} \cap \bar{\mc{D}}$.

\item Let $\mc{C}_\dagger > 0$ and $0 < \varepsilon_0 \ll 1$ be the universal constants from the statement of Theorem \ref{thm_int_obs_1}, and assume that the curvature bounds \eqref{eq_obs_curv_final_LM} hold.

\item Let $\Gamma_+$ be as in \eqref{eq_gamma_final_LM}, and let $\Gamma$ be any open subset of $\partial \mc{U}$ that contains $\bar{\Gamma}_+$.
\end{itemize}
Then, given $(y_0^{\pm}, y_1^{\pm}) \in L^2(\mc{V}_{\pm}) \times H^{-1}(\mc{V}_{\pm})$, there exists $F \in L^2(\Gamma)$ so that the solution $y$ of 
\[
\begin{cases}
\square y + \nb_{\mc{X}} y + q y = 0, \qquad & \text{on } \mc{U} \text{,} \\
y = F \mathbf{1}_\Gamma, & \text{on } \partial \mc{U} \text{,} \\
(y,\partial_t y) = (y_0^-,y_1^-), & \text{on } \mc{V}_- \text{,}
\end{cases}
\]
satisfies
\[
(y,\partial_t y) = (y_0^+,y_1^+) \qquad \text{on } \mc{V}_+ \text{.}
\]
\end{theorem}

\subsection{Application to Inverse Problems}

Finally, recall that our main Carleman estimate, Theorem \ref{thm_carl_est_LM}, implies a unique continuation result on the null cone exterior $\mc{D}$.
This unique continuation result can then be plugged into the arguments of \cite{lor_cald} to obtain an inverse problems result on Lorentzian backgrounds---in particular, that of recovering the potential in a linear wave equation from the Dirichlet-to-Neumann map.

In particular, from the above, we obtain a variant of the main result of \cite{lor_cald}, in which we can treat wave equations with arbitrary first-order terms, but we assume the curvature bounds from this article.
A rough statement of the result is as follows:

\begin{theorem} \label{thm_inv_pblm_VX}
Let $( \mc{M}, g )$ be a $(n+1)$-dimensional connected, smooth Lorentzian manifold, with timelike boundary $\partial \mc{M}$.
In addition, assume the geometric conditions (H1)-(H5) in \cite{lor_cald} hold on $( \mc{M}, g )$, but with the curvature bound (H2) replaced by the condition
\footnote{The $L^\infty$-norms can be measured using any orthonormal frame from a fixed unit timelike vector field.}
\[
\|R\|_{ L^\infty ( \mc{M} ) } < \varepsilon_0 \mc{C}_0 \text{,} \qquad \| \nb R \|_{ L^\infty ( \mc{M} ) } < \mc{C}_1 \text{,}
\]
where $0 < \varepsilon_0 \ll 1$ is a universal constant, and where $\mc{C}_0, \mc{C}_1 > 0$ depend on $\mc{M}$.
\footnote{$\mc{C}_0$ and $\mc{C}_1$ depend on the ``size" of $\mc{M}$; these arise from the factors $\mc{C}_\dagger$ and $r_0$ in \eqref{eq_obs_curv_final_LM}.}
Now, given a smooth vector field $\mc{X}$ and scalar-valued functions $V_1, V_2$ on $\mc{M}$, we set
\[
\mi{I}( \mc{X}, V_i ) := \{ ( \phi, \mc{N} \phi) |_{\partial \mc{M}} \mid \square \phi + \nb_{\mc{X}} \phi + V_i \phi = 0 \} \text{,} \qquad i = 1, 2 \text{,}
\]
with $\mc{N}$ the unit outward normal to $\mc{U}$.
Then, if $\mi{I}( \mc{X}, V_1 ) = \mi{I} ( \mc{X}, V_2 )$, then $V_1 = V_2$.
\end{theorem}

\begin{remark}
From our unique continuation results combined with the techniques of \cite{lor_cald, lor_inv_new}, one should also be able to recover the first-order coefficients $\mc{X}$ in Theorem \ref{thm_inv_pblm_VX}.
However, this runs into additional technical issues involving gauge invariances for $\mc{X}$, hence we avoid treating this more general case here for simplicity.
\end{remark}

%
%

\raggedright

\providecommand{\bysame}{\leavevmode\hbox to3em{\hrulefill}\thinspace}
\providecommand{\MR}{\relax\ifhmode\unskip\space\fi MR }
\providecommand{\MRhref}[2]{%
  \href{http://www.ams.org/mathscinet-getitem?mr=#1}{#2}
}
\providecommand{\href}[2]{#2}

\end{document}